\documentclass[1p]{elsarticle}

\usepackage{latexsym, amssymb, amsfonts, eucal}
\usepackage{amsthm, amsmath}
\usepackage{mathrsfs}
\usepackage{hyperref}
\usepackage{leftidx}
\usepackage{enumerate}
\usepackage{calc}

\input xy
\xyoption{all}
\xyoption{ps}
\SelectTips{cm}{}
\CompileMatrices

\newdir{ >}{{}*!/-5pt/\dir{>}}

\DeclareMathOperator{\Ab}{\mathbf{Ab}} 
\DeclareMathOperator{\Ban}{\mathbf{Ban}} 
\DeclareMathOperator{\Fre}{\mathbf{Fre}}

\DeclareMathOperator{\Mod}{\mathbf{Mod}}
\DeclareMathOperator{\Pol}{\mathbf{Pol}}

\DeclareMathOperator{\Sheaves}{Sheaves}

\DeclareMathOperator{\id}{id}

\DeclareMathOperator{\Hom}{Hom}
\DeclareMathOperator{\Ext}{Ext}
\DeclareMathOperator{\Ex}{Ex}
\DeclareMathOperator{\Tor}{Tor}
\DeclareMathOperator{\mono}{\rightarrowtail}
\DeclareMathOperator{\epi}{\twoheadrightarrow}

\DeclareMathOperator{\Ker}{Ker}
\DeclareMathOperator{\Coim}{Coim}
\DeclareMathOperator{\imhelp}{Im} 
\renewcommand{\Im}{\imhelp} 
\DeclareMathOperator{\Coker}{Coker}

\DeclareMathOperator{\coim}{coim}
\DeclareMathOperator{\im}{im}
\DeclareMathOperator{\coker}{coker}

\DeclareMathOperator{\Ac}{\mathbf{Ac}}
\DeclareMathOperator{\Ch}{\mathbf{Ch}}
\DeclareMathOperator{\K}{\mathbf{K}}
\DeclareMathOperator{\D}{\mathbf{D}}

\DeclareMathOperator{\cone}{cone} 

\DeclareMathOperator{\prtens}{\widehat{\otimes}}

\DeclareMathOperator{\opp}{op}

\newcommand{\mat}[1]{\left[\begin{smallmatrix} #1 \end{smallmatrix}\right]}

\DeclareMathOperator{\scrA}{\mathscr{A}}
\DeclareMathOperator{\scrB}{\mathscr{B}}
\DeclareMathOperator{\scrC}{\mathscr{C}}
\DeclareMathOperator{\scrD}{\mathscr{D}}
\DeclareMathOperator{\scrE}{\mathscr{E}}
\DeclareMathOperator{\scrF}{\mathscr{F}}

\DeclareMathOperator{\scrI}{\mathscr{I}}

\DeclareMathOperator{\scrL}{\mathscr{L}}
\DeclareMathOperator{\scrM}{\mathscr{M}}

\DeclareMathOperator{\scrO}{\mathscr{O}}
\DeclareMathOperator{\scrP}{\mathscr{P}}

\DeclareMathOperator{\scrS}{\mathscr{S}}
\DeclareMathOperator{\scrT}{\mathscr{T}}

\DeclareMathOperator{\scrY}{\mathscr{Y}}

\DeclareMathOperator{\calS}{\mathcal{S}}

\DeclareMathOperator{\bfK}{\mathbf{K}}
\DeclareMathOperator{\bfL}{\mathbf{L}}

\DeclareMathOperator{\bfR}{\mathbf{R}}

\DeclareMathOperator{\bfr}{\mathbf{r}}

\if{0}

\makeatletter
\def\@strippedMR{}
\def\@scanforMR#1#2#3\endscan{%
  \ifx#1M\ifx#2R\def\@strippedMR{#3}%
  \else\def\@strippedMR{#1#2#3}%
  \fi\fi}
\renewcommand\MR[1]{\relax\ifhmode\unskip\spacefactor3000 \space\fi
  \@scanforMR#1\endscan
  MR\MRhref{\@strippedMR}{\@strippedMR}}
\makeatother 
\fi
\if{0}
\makeatletter
\newcommand\@dotsep{4.5}
\def\@tocline#1#2#3#4#5#6#7{\relax
  \ifnum #1>\c@tocdepth 
  \else
    \par \addpenalty\@secpenalty\addvspace{#2}%
    \begingroup \hyphenpenalty\@M
    \@ifempty{#4}{%
      \@tempdima\csname r@tocindent\number#1\endcsname\relax
    }{%
      \@tempdima#4\relax
    }%
    \parindent\z@ \leftskip#3\relax \advance\leftskip\@tempdima\relax
    \rightskip\@pnumwidth plus1em \parfillskip-\@pnumwidth
    #5\leavevmode\hskip-\@tempdima #6\relax
    \leaders\hbox{$\m@th
      \mkern \@dotsep mu\hbox{.}\mkern \@dotsep mu$}\hfill
    \hbox to\@pnumwidth{\@tocpagenum{#7}}\par
    \nobreak
    \endgroup
  \fi}
\makeatother 
\fi

\newtheoremstyle{mythm}
 {6pt}
 {6pt}
 {\itshape}
 {}
 {\scshape}
 {.}
 {.5em}
 {}%

\newtheoremstyle{mydef}
 {6pt}
 {6pt}
 {}
 {}
 {\scshape}
 {.}
 {.5em}
 {}%
 
\makeatletter
\renewenvironment{proof}[1][\proofname]{\par
  \pushQED{\qed}%
  \normalfont \topsep6\p@\@plus6\p@\relax
  \trivlist
  \item[\hskip\labelsep
        \scshape
    #1\@addpunct{.}]\ignorespaces
}{%
  \popQED\endtrivlist\@endpefalse
}
\makeatother

\swapnumbers

\theoremstyle{mythm}
\newtheorem{Thm}{Theorem}[section]
\newtheorem{Lem}[Thm]{Lemma}
\newtheorem{Cor}[Thm]{Corollary}
\newtheorem{Prop}[Thm]{Proposition}

\theoremstyle{mydef}

\newtheorem{Exer}[Thm]{Exercise}
\newtheorem{Exm}[Thm]{Example}

\newtheorem{Rem}[Thm]{Remark}

\newtheorem{Def}[Thm]{Definition}

\newtheorem*{Histrem}{Historical Note}
\newtheorem*{Disclaimer}{Disclaimer}
\newtheorem*{Prerequisites}{Prerequisites}

\title{Exact Categories}
\author{Theo B\"uhler}
\ead{theo@math.ethz.ch}
\address{FIM, HG G39.5, R\"amistrasse~101, 8092 ETH Z\"urich,
  Switzerland}

\date{\today}

\begin{document}


\begin{abstract}
  We survey the basics of homological algebra in exact categories
  in the sense of Quillen. All diagram lemmas are proved
  directly from the axioms, notably the five lemma, the 
  $3\times 3$-lemma and the snake lemma. We briefly discuss exact
  functors, idempotent completion and weak idempotent completeness.
  We then show that it is possible
  to construct the derived category of an exact category without
  any embedding into abelian categories and we sketch Deligne's approach
  to derived functors.
  The construction of classical derived functors with 
  values in an abelian category painlessly translates to exact
  categories, i.e., we give proofs of the comparison theorem for
  projective resolutions and the horseshoe lemma. 
  After discussing some examples
  we elaborate on Thomason's proof of the Gabriel-Quillen
  embedding theorem in an appendix.
\end{abstract}

\maketitle

\tableofcontents

\section{Introduction}

There are several notions of exact categories. On the one hand, there
is the notion in the context of additive categories commonly
attributed to Quillen~\cite{MR0338129} with which the present article
is concerned; on the other hand, there is the non-additive notion due
to Barr~\cite{barr-exact}, to mention but the two most prominent
ones. While Barr's definition is intrinsic and an additive category is 
exact in his sense if and only if it is abelian, Quillen's
definition is extrinsic in that one has to specify a distinguished 
class of short exact sequences (an exact structure) in order to
obtain an exact category.

From now on we shall only deal with additive categories, so functors
are tacitly assumed to be additive. On every additive category
$\scrA$ the class of all split exact sequences provides the smallest
exact structure, i.e., every other exact structure must
contain it. In general, an exact structure consists of kernel-cokernel
pairs subject to some closure requirements, so the class of all
kernel-cokernel pairs is a candidate for the largest exact
structure. It is quite often the case that the class of all
kernel-cokernel pairs is an exact structure, but this fails in
general: Rump~\cite{rump-counterexample} constructs an example of an
additive category with kernels and cokernels whose kernel-cokernel
pairs fail to be an exact structure.

It is commonplace that basic homological algebra in categories of
modules over a ring extends to abelian
categories. By using the Freyd-Mitchell full embedding
theorem (\cite{MR0166240} and \cite{MR0167511}), diagram lemmas
can be transferred from module categories to general abelian categories, 
i.e., one may argue by chasing elements around in diagrams. 
There is a point in proving the fundamental diagram lemmas
directly, and be it only to familiarize oneself
with the axioms. A careful study of what is actually needed 
reveals that in most situations
the axioms of exact categories are sufficient. An \emph{a posteriori}
reason is provided by the Gabriel-Quillen embedding theorem which
reduces homological algebra in exact categories to the case of abelian
categories, the slogan is ``relative homological algebra made
absolute'' (Freyd~\cite{MR0146234}). More specifically, the embedding
theorem asserts that the Yoneda functor embeds a small exact category
$\scrA$ fully faithfully into the \emph{abelian} category $\scrB$ 
of left exact functors $\scrA^{\opp} \to \Ab$
in such a way that the essential image is closed under
extensions and that a short sequence in $\scrA$ is short exact if and
only if it is short exact in $\scrB$. 
Conversely, it is not hard to see that an extension-closed
subcategory of an abelian category is exact---this is \emph{the} basic
recognition principle for exact categories.
In appendix~\ref{sec:the-embedding-theorem} we present Thomason's
proof of the Gabriel-Quillen embedding theorem for the
sake of completeness, but we will not apply it in these notes. 
The author is convinced that the
embedding theorem should be used to transfer the intuition from
abelian categories to exact categories rather than to prove (simple) theorems
with it. A direct proof from the axioms provides much more insight
than a reduction to abelian categories.

The interest of exact categories is manifold. First of all they are a
natural generalization of abelian categories and there is no need to
argue that abelian categories are both useful and important. There are
several reasons for going beyond abelian categories. The
fact that one may \emph{choose} an exact structure gives more
flexibility which turns out to be essential in many contexts.
Even if one is working with abelian categories one soon
finds the need to consider other exact structures than the canonical
one, for instance in relative homological algebra~\cite{MR0080654}. 
Beyond this, there are quite a few ``cohomology theories'' which
involve functional analytic categories like locally convex modules
over a topological group~\cite{MR0147577,MR1721403},
locally compact abelian groups~\cite{MR2329311} or Banach
modules over a Banach algebra~\cite{MR0417787,MR1093462}
where there is no obvious abelian category around to which one could
resort. In more advanced topics of algebra and representation theory,
(e.g. filtered objects, tilting theory), 
exact categories arise naturally, while the theory of abelian
categories simply does not fit. 
It is an observation due to Happel~\cite{MR935124} that in guise of 
\emph{Frobenius categories,} exact categories give rise to
triangulated categories by passing to the associated stable
categories, see section~\ref{sec:frobenius-cats}.
Further fields of application are
algebraic geometry (certain categories of vector bundles), algebraic
analysis ($\scrD$-modules) and, of course, algebraic $K$-theory
(Quillen's $Q$-construction~\cite{MR0338129}, 
Balmer's Witt groups~\cite{MR2181829} and Schlichting's
Grothendieck-Witt groups~\cite{schlichting}).
The reader will find a slightly more extensive
discussion of some of the topics mentioned above in
section~\ref{sec:examples}.

The author hopes to convince the reader that the axioms of exact
categories are quite convenient for giving relatively painless proofs
of the basic results in homological algebra and that the gain in
generality comes with almost no effort. Indeed, it even seems that the
axioms of exact categories are more adequate for proving the
fundamental diagram lemmas than Grothendieck's axioms for abelian
categories. For instance, it is quite a challenge to find a complete
proof (directly from the axioms) of the snake lemma for abelian
categories in the literature.

That being said, we turn to a short description of the contents of
this paper.

In section~\ref{sec:def-basic-properties} we state and discuss the axioms
and draw the basic consequences, in particular we give the
characterization of pull-back squares and Keller's proof of the
obscure axiom.

In section~\ref{sec:diagram-lemmas} we prove the (short) five lemma,
the Noether isomorphism theorem and the $3\times 3$-lemma.

Section~\ref{sec:quasi-ab-cats} briefly discusses quasi-abelian
categories, a source of many examples of exact categories. Contrary to
the notion of an exact category, the property of being quasi-abelian
is intrinsic.

Exact functors are briefly touched upon in
section~\ref{sec:exact-functors} and after that we treat the
idempotent completion and the property of weak idempotent completeness
in sections~\ref{sec:idempotent-completion}
and~\ref{sec:weak-idempotent-compl}.

We come closer to the heart of homological algebra when discussing
admissible morphisms, long exact sequences, the five lemma and the
snake lemma in section~\ref{sec:adm-morph-snake-lemma}. In order for
the snake lemma to hold, the assumption of
weak idempotent completeness is necessary.

After that we briefly remind the reader of the notions of chain
complexes and chain homotopy in
section~\ref{sec:ch-cxes-ch-htpy}, before we turn to acyclic complexes
and quasi-i\-so\-mor\-phisms in section~\ref{sec:ac-cxes-qis}. Notably, we
give an elementary proof of Neeman's crucial result that the category
of acyclic complexes is triangulated. We do not indulge in the details
of the construction of the derived category of an exact category because
this is well treated in the literature. We give a brief summary of the
derived category of fully exact subcategories and then sketch the
main points of Deligne's approach to total derived functors on the
level of the derived category as expounded by Keller~\cite{MR1421815}.

On a more leisurely level, projective and injective objects are
discussed in section~\ref{sec:proj-inj-obj} preparing the grounds for a
treatment of classical derived functors (with values in an abelian
category) in section~\ref{sec:resolutions}, where we state and prove
the resolution lemma, the comparison theorem and the horseshoe lemma,
i.e., the three basic ingredients for the classical construction.

We end with a short list of examples and applications in
section~\ref{sec:examples}.

In appendix~\ref{sec:the-embedding-theorem} we give Thomason's
proof of the Gabriel-Quillen
embedding theorem of an exact category into an abelian one. 
Finally, in appendix~\ref{sec:hellers-axioms}
we give a proof of the folklore fact that under the
assumption of weak idempotent completeness Heller's axioms for an
``abelian'' category are equivalent to Quillen's axioms for an exact category.

\begin{Histrem}
  Quillen's notion of an exact category has its predecessors e.g. in
  Heller~\cite{MR0100622}, 
  Buchsbaum~\cite{MR0140556}, 
  Yoneda~\cite{MR0225854}, 
  Butler-Horrocks~\cite{MR0188267} 
  and Mac~Lane~\cite[XII.4]{MR0156879}. 
  It should be noted that Buchsbaum, Butler-Horrocks and Mac~Lane
  assume the existence of an ambient abelian category and miss the
  crucial push-out and pull-back axioms, while Heller and Yoneda
  anticipate Quillen's definition. According to
  Quillen~\cite[p.~``92/16/100'']{MR0338129}, assuming idempotent
  completeness, Heller's notion of an ``abelian
  category''~\cite[\S~3]{MR0100622}, 
  i.e., an additive category
  equipped with an ``abelian class of short exact
  sequences''
  coincides
  with the present definition of an exact category. We give a proof of
  this assertion in appendix~\ref{sec:hellers-axioms}.  Yoneda's
  quasi-abelian $\calS$-categories are nothing but Quillen's exact
  categories and it is a remarkable fact that Yoneda proves that
  Quillen's ``obscure axiom'' follows from his definition,
  see~\cite[p.~525, Corollary]{MR0225854}, a fact rediscovered thirty
  years later by Keller in \cite[A.1]{MR1052551}.
\end{Histrem}

\begin{Prerequisites} 
  The prerequisites are kept at a minimum. The
  reader should know what an additive category is and be familiar with
  fundamental categorical concepts such as kernels, pull-backs,
  products and duality. Acquaintance with basic category theory as
  presented in Hilton-Stammbach~\cite[Chapter~II]{MR1438546}
  or Weibel~\cite[Appendix~A]{MR1269324} should amply suffice for a complete
  understanding of the text, up to section~\ref{sec:ac-cxes-qis} where
  we assume some familiarity with the theory of
  triangulated categories.
\end{Prerequisites}

\begin{Disclaimer}
  This article is written for the reader who \emph{wants}
  to learn about exact categories and knows \emph{why}. 
  Very few motivating examples are given in this text.
  
  The author makes no claim to originality. All the results
  are well-known in some form and they are scattered around in the
  literature. The \emph{raison d'\^{e}tre} of this article is the lack
  of a systematic \emph{elementary} exposition of the theory. 
  The works of Heller~\cite{MR0100622},
  Keller~\cite{MR1052551,MR1421815} and Thomason~\cite{MR1106918}
  heavily influenced the present paper and many proofs given here
  can be found in their papers.
\end{Disclaimer}

\section{Definition and Basic Properties}
\label{sec:def-basic-properties}

In this section we introduce the notion of an exact category and draw
the basic consequences of the axioms. We do not use the minimal
axiomatics as provided by Keller~\cite[Appendix~A]{MR1052551} but
prefer to use a convenient self-dual presentation of the axioms due to
Yoneda~\cite[\S~2]{MR0225854} (modulo some of Yoneda's numerous 
$3 \times 2$-lemmas and our Proposition~\ref{prop:pushout-exact}). The
author hopes that the Bourbakists among the readers will 
pardon this \emph{faux pas}. We will discuss that the present axioms are
equivalent to Quillen's~\cite[\S~2]{MR0338129} in the course of
events. The main points of this section are a characterization of
push-out squares (Proposition~\ref{prop:pushout-exact}) and the
obscure axiom (Proposition~\ref{prop:obscure-axiom}).

\begin{Def}
  Let $\scrA$ be an additive category. A \emph{kernel-cokernel pair}
  $(i,p)$ in $\scrA$ is a pair of composable morphisms
  \[
  A' \xrightarrow{i} A \xrightarrow{p} A''
  \]
  such that $i$ is a kernel of $p$ and $p$ is a cokernel of $i$. If a
  class $\scrE$ of kernel-cokernel pairs on $\scrA$ is fixed, an
  \emph{admissible monic} is a morphism $i$ for which there exists a
  morphism $p$ such that $(i,p) \in \scrE$. \emph{Admissible epics}
  are defined dually. We depict admissible monics by $\mono$
  and admissible epics by $\epi$ in diagrams.
  
  An \emph{exact structure} on $\scrA$ is a
  class $\scrE$ of kernel-cokernel pairs which is closed under
  isomorphisms and satisfies the following
  axioms:
  \begin{itemize}
    \item[{[E0]}]
      For all objects $A \in \scrA$, the identity morphism
      $1_{A}$ is an admissible monic.

    \item[{[E0$^{\opp}$]}]
      For all objects $A \in \scrA$, the identity morphism
      $1_{A}$ is an admissible epic.

    \item[{[E1]}]
      The class of admissible monics is closed under composition.

    \item[{[E1$^{\opp}$]}]
      The class of admissible epics is closed under composition.

    \item[{[E2]}]
      The push-out of an admissible monic along an arbitrary morphism
      exists and yields an admissible monic.

    \item[{[E2$^{\opp}$]}]
      The pull-back of an admissible epic along an arbitrary morphism
      exists and yields an admissible epic.
  \end{itemize}
  Axioms [E2] and [E2$^{\opp}$] are subsumed in the diagrams
  \[
  \vcenter{
    \xymatrix{
      A \ar[d] \ar@{ >->}[r] \ar@{}[dr]|{\text{PO}} & B \ar@{.>}[d] \\
      A' \ar@{ >.>}[r] & B'
    }
  } 
  \qquad \text{and} \qquad
  \vcenter{
    \xymatrix{
      A' \ar@{.>}[d] \ar@{.>>}[r] \ar@{}[dr]|{\text{PB}} & B' \ar[d]
      \\
      A \ar@{->>}[r] & B
    }
  }
  \]
  respectively.
  
  An \emph{exact category} is a pair $(\scrA,\scrE)$ consisting of an
  additive category $\scrA$ and an exact structure $\scrE$ on
  $\scrA$. Elements of $\scrE$ are called \emph{short exact sequences}.
\end{Def}

\begin{Rem}
  Note that $\scrE$ is an exact structure on
  $\scrA$ if and only if $\scrE^{\opp}$ is an exact structure on
  $\scrA^{\opp}$. This allows for reasoning by dualization.
\end{Rem}
\if{0}
\begin{Rem}
  The reader who ignored the disclaimer in the introduction and
  insists on seeing examples should consult
  section~\ref{sec:examples}.
\end{Rem}

\begin{Exm}
  Every abelian category $\scrA$ is exact with respect to the class
  $\scrE$ of all short exact sequences. 
\end{Exm}

\begin{Exm}
  Let $\scrA$ be a full subcategory of an abelian category $\scrB$ and
  assume that $\scrA$ is \emph{closed under extensions} in $\scrB$,
  i.e., for all short exact sequences $A' \mono B \epi A''$ in $\scrB$
  with $A',A'' \in \scrA$ we have that $B \in \scrA$. Then $\scrA$ is
  an exact category with respect to the class $\scrE$ of sequences in
  $\scrA$ which are exact in $\scrB$. Indeed, axiom~[E$0$] is obvious and
  axiom [E$1$] follows from the Noether isomorphism theorem in
  $\scrB$ (see Lemma~\ref{lem:c/b=(c/a)/(b/a)}) together with the hypothesis
  that $\scrA$ is closed under extensions. To verify axiom [E$2$], let
  $i: A' \mono A$ be an admissible monic and let $f':A' \to B'$ be an
  arbitrary morphism. Form the push-out under $i$ and $f$ in
  $\scrB$ to obtain the diagram with exact rows in $\scrB$
  \[
  \xymatrix{
    A' \ar[d]_{f'} \ar@{}[dr]|{\text{PO}} \ar@{ >->}[r]^{i} &
    A \ar[d]^{f} \ar@{->>}[r] & 
    A'' \ar@{=}[d] \\
   B' \ar@{ >->}[r]^{j} & B \ar@{->>}[r]^{q} & A'' 
  }
  \]
  (see Proposition~\ref{prop:pushout-exact}). Since $B'$ and $A''$ are
  in $\scrA$, we conclude that $B$ is in $\scrA$
  by appealing to 
  consequence of Proposition~\ref{prop:pushout-exact} in $\scrB$
\end{Exm}

\fi

\begin{Rem}
  \label{rem:isos-adm-mono-adm-epic}
  Isomorphisms are admissible monics and admissible epics. 
  Indeed, this follows from the commutative diagram
  \[
  \xymatrix{
    A \ar[d]^{\cong}_{1_{A}} \ar[r]^{f}_{\cong} & B \ar[r]
    \ar[d]^{\cong}_{f^{-1}} & 0 \ar[d]_{\cong} \\
    A \ar@{ >->}[r]^{1_{A}} &
    A \ar@{->>}[r] & 0,
  }
  \]
  the fact that exact structures are assumed to be closed under
  isomorphisms and that the axioms are self-dual.
\end{Rem}

\begin{Rem}[Keller~{\cite[App.~A]{MR1052551}}]
  \label{rem:weakening-the-axioms}
  The axioms are somewhat redundant and can be weakened. For instance,
  let us assume instead of [E0] and [E0$^{\opp}$]
  that $1_{0}$, the identity of the zero object, is an admissible epic.
  For any object $A$ there is the pull-back diagram
  \[
  \xymatrix{
    A \ar[d] \ar[r]^{1_{A}} \ar@{}[dr]|{\text{PB}} & A \ar[d] \\
    0 \ar[r]^{1_{0}} & 0
  }
  \]
  so [E2$^{\opp}$] together with our assumption on $1_{0}$ 
  shows that [E0$^{\opp}$] holds. Since
  $1_{0}$ is a kernel of itself, it is also an admissible monic, so
  we conclude by [E2] that [E0] holds as well.
  More importantly, Keller proves in \emph{loc. cit.}
  (A.1, proof of the proposition, step~3), that
  one can also dispose of one of [E1] or [E1$^{\opp}$]. Moreover, he
  mentions (A.2, Remark), that one may
  also weaken one of~[E2] or~[E2$^{\opp}$]---this is a straightforward
  consequence of (the proof of)
  Proposition~\ref{prop:decompose-morphisms-of-exact-sequences}.
\end{Rem}

\begin{Rem}
  Keller \cite{MR1052551, MR1421815}
  uses \emph{conflation}, \emph{inflation} and \emph{deflation}
  for what we call short exact sequence, admissible monic and
  admissible epic. This terminology stems from
  Gabriel-Ro{\u\i}ter~\cite[Ch.~9]{MR1239447}
  who give a list of axioms for exact categories whose
  underlying additive category is weakly idempotent complete in the
  sense of section~\ref{sec:weak-idempotent-compl}, see Keller's
  appendix to \cite{MR1608305} for a thorough comparison of the 
  axioms. 
  A variant of the Gabriel-Ro{\u\i}ter-axioms
  appear in Freyd's book on abelian 
  categories~\cite[Ch.~7, Exercise~G, p.~153]{MR0166240} (the
  Gabriel-Ro{\u\i}ter-axioms are obtained from Freyd's axioms by
  adding the dual of Freyd's condition~(2)).
\end{Rem}

\begin{Exer}
  \label{exer:adm-epic+monic=iso}
  An admissible epic which is additionally monic is an isomorphism.
\end{Exer}

\begin{Lem}
  \label{lem:split-sequences-exact}
  The sequence
  \[
  \xymatrix{
    A \ar@{ >->}[r]^-{\mat{1 \\ 0}} &  
    A \oplus B \ar@{->>}[r]^-{\mat{0 & 1}} &
    B
  }
  \]
  is short exact.
\end{Lem}
\begin{proof}
  The following diagram is a push-out square
  \[
  \xymatrix{
    0 \ar@{ >->}[r] \ar@{ >->}[d] \ar@{}[dr]|{\text{PO}}
    & B \ar@{ >->}[d]^{\mat{0 \\ 1}} \\
    A \ar@{ >->}[r]^<<<<<{\mat{1 \\ 0}} 
    & A \oplus B.
   }
   \]
   The top arrow and the left hand arrow are admissible monics
   by~[E0$^{\opp}$] while the bottom arrow and the right hand arrow
   are admissible monics by~[E2]. The lemma now follows from the
   facts that the sequence in question is a kernel-cokernel pair and
   that $\scrE$ is closed under isomorphisms.
\end{proof}

\begin{Rem}
  Lemma~\ref{lem:split-sequences-exact} shows that Quillen's axiom~a)
  \cite[\S~2]{MR0338129} stating that split exact sequences belong to
  $\scrE$ follows from our axioms. Conversely,
  Quillen's axiom~a) obviously implies [E0] and [E0$^{\opp}$].
  Quillen's axiom~b) coincides with our axioms [E1], [E1$^{\opp}$], [E2] and
  [E2$^{\opp}$]. We will prove that Quillen's axiom~c) follows from
  our axioms in Proposition~\ref{prop:obscure-axiom}.
\end{Rem}

\begin{Prop}
  \label{prop:sum-exact}
  The direct sum of two short exact sequences is short exact.
\end{Prop}
\begin{proof}
  Let $A' \mono A \epi A''$ and $B' \mono B \epi B''$ be two short exact
  sequences. First observe that for every object $C$ the sequence
  \[
  A' \oplus C \mono A \oplus C \epi A''
  \]
  is exact---the second morphism is an admissible epic because it is
  the composition of the admissible epics
  $\mat{1 & 0}: A \oplus C \epi A$ and $A \epi A''$;
  the first morphism in the sequence is a kernel of the
  second one, hence it is an admissible monic. 
  Now it follows from~[E1] that
  \[
  A' \oplus B' \mono A \oplus B
  \]
  is an admissible monic because it is the composition of the two
  admissible monics $A' \oplus B' \mono A \oplus B'$ and
  $A \oplus B' \mono A \oplus B$. It is obvious that
  \[
  A' \oplus B' \mono A \oplus B \epi A'' \oplus B''
  \]
  is a kernel-cokernel pair, hence the proposition is proved.
\end{proof}

\begin{Cor}
  \label{cor:exact-structure-additive}
  The exact structure $\scrE$ is an additive subcategory of
  the additive category $\scrA^{\to\to}$ of composable morphisms
  of $\scrA$. \qed
\end{Cor}

\begin{Rem}
  In Exercise~\ref{exer:ses-exact} the reader is asked to show
  that $\scrE$ is exact with respect to a natural exact
  structure.
\end{Rem}

\begin{Prop}
    \label{prop:pushout-exact}
  Consider a commutative square
  \[
  \xymatrix{
    A \ar@{ >->}[r]^{i} \ar[d]_{f} &
    B \ar[d]^{f'} \\
    A' \ar@{ >->}[r]^{i'} & B'
  }
  \]
  in which the horizontal arrows are admissible monics.
  The following assertions are equivalent:
  \begin{enumerate}[(i)]
    \item
      The square is a push-out.

    \item
      The sequence
      $\xymatrix@1{A \ar@{ >->}[r]^-{\mat{i \\ -f}} & B \oplus A'
      \ar@{->>}[r]^-{\mat{f' & i'}} & B'}$ is short exact.

    \item
      The square is bicartesian, i.e., both a push-out and a
      pull-back.

    \item
      The square is part of a commutative diagram
      \[
      \xymatrix{
        A \ar@{ >->}[r]^{i} \ar[d]_{f} &
        B \ar[d]^{f'} \ar@{->>}[r]^{p} & C \ar@{=}[d] \\
        A' \ar@{ >->}[r]^{i'} & B' \ar@{->>}[r]^{p'} & C
      }
      \]
      with exact rows.
  \end{enumerate}
\end{Prop}
\begin{proof}
  (i) $\Rightarrow$ (ii):
  The push-out property is equivalent to the assertion that
  $\mat{f' & i'}$ is a cokernel of $\mat{i \\ -f}$, so it suffices to
  prove that the latter is an admissible monic. But this follows from
  [E1] since $\mat{i \\ -f}$ is equal to the composition of the morphisms
  \[
  \xymatrix{
    A \ar@{ >->}[r]^>>>>>{\mat{1 \\ 0}} &
    A \oplus A' \ar[r]^{\mat{1 & 0 \\ -f & 1}}_{\cong} &
    A \oplus A' \ar@{ >->}[r]^{\mat{i & 0 \\ 0 & 1}} &
    B \oplus A'
  }
  \]
  which are all admissible monics by 
  Lemma~\ref{lem:split-sequences-exact}, 
  Remark~\ref{rem:isos-adm-mono-adm-epic} and
  Proposition~\ref{prop:sum-exact}, respectively.
  
  (ii) $\Rightarrow$ (iii) and (iii) $\Rightarrow$ (i): obvious.

  (i) $\Rightarrow$ (iv): Let $p: B \epi C$ be a cokernel of $i$. The
  push-out property of the square yields that there is a unique
  morphism $p':B' \to C$ such that $p'f' = p$ and $p'i' = 0$. Observe
  that $p'f' = p$ implies that $p'$ is epic. In order to see that $p'$
  is a cokernel of $i'$, let $g: B' \to X$ be such that $gi' = 0$. 
  Then $gf'i = gi'f = 0$, so $gf'$ factors uniquely over a  morphism
  $h: C \to X$ such that $gf' = hp$. We claim that $hp' = g$:
  this follows from the push-out property of the square because $hp'f'
  = hp = gf'$ and $hp'i' = 0 = gi'$. Since $p'$ is epic, the
  factorization $h$ of $g$ is unique, so $p'$ is a cokernel of $i'$.

  (iv) $\Rightarrow$ (ii): Form the pull-back over $p$ and $p'$ in
  order to obtain the commutative diagram
  \[
  \xymatrix{
    &
    A \ar@{=}[r] \ar@{ >->}[d]^{j} & 
    A \ar@{ >->}[d]^{i} \\
    A' \ar@{ >->}[r]^{j'} \ar@{=}[d] &
    D \ar@{}[dr]|{\text{PB}} \ar@{->>}[r]^{q'} \ar@{->>}[d]^{q} &
    B \ar@{->>}[d]^{p} \\
    A' \ar@{ >->}[r]^{i'} & B' \ar@{->>}[r]^{p'} & C
  }
  \]
  with exact rows and columns using the dual of the implication 
  (i)~$\Rightarrow$~(iv). Since the square
  \[
  \xymatrix{
    B \ar@{=}[r] \ar[d]^{f'} & B \ar@{->>}[d]^{p} \\
    B' \ar@{->>}[r]^{p'} & C
  }
  \]
  is commutative, there is a unique morphism $k: B \to D$ such that
  $q'k = 1_{B}$ and $qk = f'$. Since $q'(1_{D} - kq') = 0$, there is a
  unique morphism $l: D \to A'$ such that $j'l = 1_{D} - kq'$. Note
  that $lk = 0$ because $j'lk = (1_{D} - kq')k = 0$ and $j'$ is
  monic, while the calculation
  $j'lj' = (1_{D}-kq')j' = j'$ implies $lj' = 1_{A'}$,
  again because $j'$ is monic.
  Furthermore
  \[
  i'lj = (qj')lj = q (1_{D} - kq')j = - (qk)(q'j) = - f'i = - i'f
  \]
  implies $lj = -f$ since $i'$ is monic. 

  The morphisms
  \[
  \mat{k & j' }: B \oplus A' \to D \qquad
  \text{and} \qquad
  \mat{q' \\ l}: D \to B \oplus A'
  \]
  are mutually inverse since
  \[
  \mat{k & j' } \mat{q' \\ l} = kq' + j'l = 1_{D} 
  \qquad \text{and} \qquad
  \mat{q' \\ l} \mat{k & j'} = \mat{q'k & q'j' \\ lk & lj'} = 
  \mat{1_{B} & 0 \\ 0 & 1_{A'}}.
  \]
  Now
  \[
  \mat{f' & i'}  = q \mat{k & j'} \qquad \text{and} \qquad
  \mat{i \\ -f} = \mat{q' \\ l} j
  \]
  show that 
  $A \xrightarrow{\mat{i \\ -f}} B \oplus A' \xrightarrow{\mat{f' & i'}} B'$
  is isomorphic to
  $A \xrightarrow{j} D \xrightarrow{q} B'$.
\end{proof}

\begin{Rem}
  \label{rem:cokernels-in-push-outs}
  Consider the push-out diagram
  \[
  \xymatrix{
    A' \ar@{ >->}[r]^{i'} \ar[d]^{a} \ar@{}[dr]|{\text{PO}} & 
    B' \ar[d]^{b} \\
    A \ar@{ >->}[r]^{i} &
    B.
  }
  \]
  If $j': B' \epi C'$ is a cokernel of $i'$ then the unique morphism
  $j: B \to C'$ such that $ji = 0$ and $jb = j'$ is a cokernel of $i$. 
  If $j: B \epi C$ is a cokernel of $i$ then $j' = jb$ is a cokernel
  of~$i'$.
  
  The first statement was established in the proof of the implication 
  (i)~$\Rightarrow$~(iv) of Proposition~\ref{prop:pushout-exact} and
  we leave the easy verification of the second statement as an
  exercise for the reader.
\end{Rem}

The following simple observation will only be used in the proof of
Lemma~\ref{lem:cone-of-acyclics-is-acyclic}. We state it here for ease
of reference.

\begin{Cor}
  \label{cor:gluing-pb-po}
  The surrounding rectangle in a diagram of the form
  \[
  \xymatrix{
    A \ar[d]^{a} \ar@{}[dr]|{\text{\emph{PB}}} \ar@{->>}[r]^{f} &
    B \ar[d]^{b} \ar@{}[dr]|{\text{\emph{PO}}} \ar@{ >->}[r]^{g} & 
    C \ar[d]^{c} \\
    A' \ar@{->>}[r]^{f'} & B' \ar@{ >->}[r]^{g'} & C'
  }
  \]
  is bicartesian and
  $\xymatrix{
    A \ar@{ >->}[r]^-{\mat{-a \\ gf}} &
    A' \oplus C \ar@{->>}[r]^-{\mat{g' \! f' & c}} & C'
  }$
  is short exact.
\end{Cor}

\begin{proof}
  It follows from Proposition~\ref{prop:pushout-exact} and its dual
  that both squares are bicartesian. Gluing two bicartesian
  squares along a common arrow yields another bicartesian square,
  which entails the first part and the fact that the sequence of the
  second part is a kernel-cokernel pair. The equation
  $\mat{g' \! f' & & c} = \mat{g' & c} \mat{f' & 0 \\ 0 & 1_{C}}$
  exhibits $\mat{g'f' && c}$ as a composition of admissible epics by
  Proposition~\ref{prop:sum-exact} and
  Proposition~\ref{prop:pushout-exact}.
\end{proof}

\begin{Prop}
  \label{prop:pb-adm-monic-adm-monic}
  The pull-back of an admissible monic along an admissible epic yields
  an admissible monic.
\end{Prop}

\begin{proof}
  Consider the diagram
  \[
  \xymatrix{
    A' \ar@{->>}[d]^{e'} \ar@{}[dr]|{\text{PB}} \ar[r]^{i'} &
    B' \ar@{->>}[d]^{e} \ar@{->>}[r]^{pe} &
    C \ar@{=}[d] \\
    A \ar@{ >->}[r]^{i} &
    B \ar@{->>}[r]^{p} &
    C.
  }
  \]
  The pull-back square exists by axiom~[E$2^{\opp}$].
  Let $p$ be a cokernel of
  $i$, so it is an admissible epic and $pe$ is an 
  admissible epic by axiom~[E$1^{\opp}$]. In any
  category, the pull-back of a monic is a monic (if it exists). In order
  to see that $i'$ is an admissible monic, it suffices  to prove that
  $i'$ is a kernel of $pe$. Suppose that 
  $g':X \to B'$ is such that $peg' = 0$. Since $i$ is a kernel of $p$,
  there exists a unique $f: X \to A$ such that $eg' = if$. Applying the
  universal property of the pull-back square, we find a unique 
  $f': X \to A'$ such that $e'f' = f$ and $i'f' = g'$. Since $i'$ is
  monic, $f'$ is the unique morphism such that $i'f' = g'$ and we are
  done.
\end{proof}

\begin{Prop}[Obscure Axiom]
  \label{prop:obscure-axiom}
  Suppose that $i: A \to B$ is a morphism in $\scrA$ admitting a cokernel.
  If there exists a morphism $j: B \to C$ in $\scrA$ such that the composite
  $ji : A \mono C$ is an admissible monic then $i$ is an admissible monic.
\end{Prop}

\begin{Rem}
  The statement of the previous proposition is given as axiom~c) in
  Quillen's definition of an exact category~\cite[\S~2]{MR0338129}. At
  that time, it was already proved to be a consequence of the other axioms
  by Yoneda~\cite[Corollary, p.~525]{MR0225854}. The redundancy of the
  obscure axiom was rediscovered by
  Keller~\cite[A.1]{MR1052551}. Thomason baptized axiom~c) the
  ``obscure axiom'' in~\cite[A.1.1]{MR1106918}.

  A convenient and quite powerful strengthening of the obscure axiom
  holds under the rather mild additional hypothesis of weak idempotent
  completeness, see 
  Proposition~\ref{prop:weakly-split-obscure-axiom}.
\end{Rem}

\begin{proof}[Proof of Proposition~\ref{prop:obscure-axiom} (Keller)]
  Let $k: B \to D$ be a cokernel of $i$. From the push-out diagram
  \[
  \xymatrix{
    A \ar@{ >->}[r]^{ji} \ar[d]_{i} \ar@{}[dr]|{\text{PO}} & C \ar[d] \\
    B \ar@{ >->}[r] & E
  }
  \]
  and Proposition~\ref{prop:pushout-exact} we conclude
  that $\mat{i \\ ji}: A \mono A \oplus B$
  is an admissible monic. Because 
  $\mat{1_{B} & 0 \\ - j & 1_{C}}: B \oplus C \to B \oplus C$
  is an isomorphism it is in particular an admissible monic, hence
  $\mat{i \\ 0} = \mat{1_{B} & 0 \\ - j & 1_{C}} \mat{i \\ ji}$
  is an admissible monic as well. Because $\mat{k & 0 \\ 0 & 1_{C}}$
  is a cokernel of~$\mat{i \\ 0}$, it is an admissible epic.
  Consider the following diagram
  \[
  \xymatrix{
    A \ar@{=}[d] \ar[r]^{i} & 
    B \ar[d]^{\mat{1 \\ 0}} \ar[r]^{k} \ar@{}[dr]|{\text{PB}} &
    D \ar[d]^{\mat{1 \\ 0}} \\
    A \ar@{ >->}[r]_>>>>{\mat{i \\ 0}} &
    {B \oplus C} \ar@{->>}[r]_{\mat{k & 0 \\ 0 & 1_{C}}} &
    {D \oplus C.}
  }
  \]
  Since the right hand square is a pull-back, it follows that $k$ is an
  admissible epic and that $i$ is a kernel of $k$, so $i$ is an admissible 
  monic.
\end{proof}

\begin{Cor}
  \label{cor:summands-exact}
  Let $(i,p)$ and $(i',p')$ be two pairs of composable morphisms. If
  the direct sum $(i \oplus i', p \oplus p')$ is exact then $(i,p)$
  and $(i',p')$ are both exact.
\end{Cor}
\begin{proof}
  It is clear that $(i,p)$ and $(i',p')$ are kernel-cokernel
  pairs. Since $i$ has $p$ as a cokernel and since
  \[
  \mat{1 \\ 0} i = \mat{i & 0 \\ 0 & i'} \mat{1 \\ 0}
  \]
  is an admissible monic, the obscure axiom implies that $i$ is an
  admissible monic.
\end{proof}

\begin{Exer}
  \label{exer:push-out-to-mono-is-mono}
  Suppose that the commutative square 
  \[
  \xymatrix{
    A' \ar@{ >->}[r]^{f'} \ar[d]^{a} \ar@{}[dr]|{\text{PO}} &
    B' \ar@{ >->}[d]^{b} \\
    A \ar@{ >->}[r]^{f} & B
  }
  \]
  is a push-out. Prove that $a$ is an admissible monic.

  \emph{Hint:}
  Let $b': B \epi B''$ be a cokernel of $b: B' \mono B$. Prove that
  $a' = b'f: A \to B''$ is a cokernel of $a$, then apply the obscure
  axiom. 
\end{Exer}

\section{Some Diagram Lemmas}
\label{sec:diagram-lemmas}

In this section we will prove variants of diagram lemmas which are 
well-known in the context of abelian categories, in particular we will
prove the five lemma and the $3 \times 3$-lemma. Further familiar
diagram lemmas will be proved in
section~\ref{sec:adm-morph-snake-lemma}. The proofs will 
be based on the following simple observation:

\begin{Prop}
  \label{prop:decompose-morphisms-of-exact-sequences} 
  Let $(\scrA,\scrE)$ be an exact category.
  A morphism from a short exact sequence $A' \mono B' \epi C'$ to
  another short exact sequence $A \mono B \epi C$ factors over
  a short exact sequence $A \mono D \epi C'$
  \[
  \xymatrix{
    A' \ar@{ >->}[r]^{f'} \ar[d]^{a} \ar@{}[dr]|{\text{\emph{BC}}} &
    B' \ar@{->>}[r]^{g'} \ar[d]^{b'} &
    C' \ar@{=}[d] \\
    A \ar@{ >->}[r]^{m} \ar@{=}[d] &
    D \ar@{->>}[r]^{e} \ar[d]_{b''}  \ar@{}[dr]|{\text{\emph{BC}}} &
    C' \ar[d]^{c} \\
    A \ar@{ >->}[r]^{f} &
    B \ar@{->>}[r]^{g} & C
  }
  \]
  in such a way that the two squares marked
  $\text{\emph{BC}}$ are bicartesian. In particular
  there is a canonical isomorphism of the
  push-out $A \cup_{A'} B'$ with the pull-back $B \times_{C} C'$.
\end{Prop}

\begin{proof}
  Form the push-out under $f'$ and $a$ in order to obtain the object
  $D$ and the morphisms $m$ and $b'$. Let $e:D \to C'$ be the unique
  morphism such that $eb' = g'$ and $em = 0$ and let $b'': D \to B$ be
  the unique morphism $D \to B$ such that $b''b' = b: B' \to B$ and
  $b''m = f$. It is easy to see that $e$ is a cokernel of $m$ 
  (Remark~\ref{rem:cokernels-in-push-outs}) and hence 
  the result follows from Proposition~\ref{prop:pushout-exact} 
  since the square $DC'BC$ is commutative 
  [this is because $a$ and $b''b'$ determine $c$ uniquely].
\end{proof}

\begin{Cor}[Five Lemma, I]
  \label{cor:five-lemma}
  Consider a morphism of short exact sequences
  \[
  \xymatrix{
    A' \ar@{ >->}[r] \ar[d]^{a} &
    B' \ar@{->>}[r] \ar[d]^{b} &
    C' \ar[d]^{c} \\
    A \ar@{ >->}[r] & B \ar@{->>}[r] & C.
  }
  \]
  If $a$ and $c$ are isomorphisms (or admissible monics, 
  or admissible epics) then so is~$b$.
\end{Cor}
\begin{proof}
  Assume first that $a$ and $c$ are isomorphisms.
  Because isomorphisms are preserved by push-outs and pull-backs, it
  follows from the diagram of
  Proposition~\ref{prop:decompose-morphisms-of-exact-sequences} 
  that $b$ is the composition of two isomorphisms $B' \to D \to B$.
  If $a$ and $c$ are both admissible monics, it follows
  from the diagram of
  Proposition~\ref{prop:decompose-morphisms-of-exact-sequences}  
  together with~[E$2$] and
  Proposition~\ref{prop:pb-adm-monic-adm-monic} 
  that $b$ is the composition of two admissible monics. 
  The case of admissible epics is dual.
\end{proof}

\begin{Exer}
  \label{exer:two-out-of-three-five-lemma}
  If in a morphism   
  \[
  \xymatrix{
    A' \ar@{ >->}[r] \ar[d]^{a} &
    B' \ar@{->>}[r] \ar[d]^{b} &
    C' \ar[d]^{c} \\
    A \ar@{ >->}[r] & B \ar@{->>}[r] & C.
  }
  \]
  of short exact sequences as in the five lemma~\ref{cor:five-lemma}
  two out of $a,b,c$ are isomorphisms then so is the third.
  
  \emph{Hint:}
  Use e.g. that $c$ is uniquely determined by $a$ and $b$.
\end{Exer}

\begin{Rem}
  The reader insisting that Corollary~\ref{cor:five-lemma}
  should be called ``three lemma'' rather than ``five lemma'' 
  is cordially invited to give the details of the proof of
  Lemma~\ref{lem:long-five-lemma} and to solve
  Exercise~\ref{exer:sharp-four-five-lemma}. 
  We will however use the more customary name five lemma.
\end{Rem}

\begin{Lem}[``Noether Isomorphism $C/B \cong (C/A) / (B / A)$'']
  \label{lem:c/b=(c/a)/(b/a)}
  Consider the diagram
  \[
  \xymatrix{
    A \ar@{=}[d] \ar@{ >->}[r] &
    B \ar@{ >->}[d] \ar@{->>}[r] & 
    X \ar@{ >.>}[d] \\
    A \ar@{ >->}[r] &
    C \ar@{->>}[r] \ar@{->>}[d] & 
    Y \ar@{.>>}[d] \\
     & Z \ar@{=}[r] & Z
  }
  \]
  in which the first two horizontal rows and the middle column are
  short exact. Then the third column exists, is short exact, and is
  uniquely determined by the requirement that it makes the diagram
  commutative.
  Moreover, the upper right hand square is bicartesian.
\end{Lem}
\begin{proof}
  The morphism $X \to Y$ exists since the first row is exact and the
  composition $A \to C \to Y$ is zero while the
  morphism $Y \to Z$ exists since the second row is exact and the
  composition $B \to C \to Z$ vanishes. By 
  Proposition~\ref{prop:pushout-exact} the square
  containing $X \to Y$ is bicartesian. It follows that $X \to Y$ is
  an admissible monic and that $Y \to Z$ is its cokernel. The
  uniqueness assertion is obvious.
\end{proof}

\begin{Cor}[$3 \times 3$-Lemma]
  \label{cor:3x3-lemma}
  Consider a commutative diagram
  \[
  \xymatrix{
    A' \ar[r]^{f'} \ar@{ >->}[d]^{a} &
    B' \ar[r]^{g'} \ar@{ >->}[d]^{b} &
    C' \ar@{ >->}[d]^{c} \\
    A \ar[r]^{f} \ar@{->>}[d]^{a'} &
    B \ar[r]^{g} \ar@{->>}[d]^{b'} &
    C \ar@{->>}[d]^{c'} \\
    A'' \ar[r]^{f''} &
    B'' \ar[r]^{g''}  &
    C''
  }
  \]
  in which the columns are exact and assume in addition that
  one of the following conditions holds:
  \begin{enumerate}[(i)]
    \item 
      the middle row and either one of the outer rows is short
      exact;
    \item
      the two outer rows are short exact and $gf = 0$.
  \end{enumerate}
  Then the remaining row is short exact as well.
\end{Cor}

\begin{proof}
  Let us prove~(i). The two possibilities are dual to each other, 
  so we need only consider the case that the first two rows are
  exact. Apply 
  Proposition~\ref{prop:decompose-morphisms-of-exact-sequences} to 
  the first two rows so as to obtain the commutative diagram
  \[
  \xymatrix{
    A' \ar@{ >->}[r]^{f'} \ar@{ >->}[d]_{a} \ar@{}[dr]|{\text{BC}} & 
    B' \ar@{->>}[r]^{g'} \ar@{ >->}[d]^{i} &
    C' \ar@{=}[d] \\
    A \ar@{ >->}[r]^{\bar{f}} \ar@{=}[d] &
    D \ar@{->>}[r]^{\bar{g}} \ar@{ >->}[d]_{j}
    \ar@{}[dr]|{\text{BC}} &
    C' \ar@{ >->}[d]^{c} \\
    A \ar@{ >->}[r]^{f} & B \ar@{->>}[r]^{g} & C
  }
  \]
  where $ji = b$---notice that $i$ and $j$ are admissible monics by
  axiom [E$2$] and Proposition~\ref{prop:pb-adm-monic-adm-monic},
  respectively. By Remark~\ref{rem:cokernels-in-push-outs} the
  morphism $i': D \to A''$ determined by
  $i'i = 0$ and $i'\bar{f} = a'$ is a cokernel of $i$ and the
  morphism $j': B \epi C''$ given by $j' = c'g = g'' b'$ is a cokernel
  of $j$.

  If we knew that the diagram  
  \[
  \xymatrix{
    B' \ar@{ >->}[r]^{i} \ar@{=}[d] &
    D \ar@{->>}[r]^{i'} \ar@{ >->}[d]^{j} &
    A'' \ar[d]^{f''} \\
    B' \ar@{ >->}[r]^{b} & B \ar@{->>}[r]^{b'} \ar@{->>}[d]^{j'} & 
    B'' \ar[d]^{g''} \\
    & C'' \ar@{=}[r] & C''
  }
  \]
  is commutative then we would conclude from the Noether
  isomorphism~\ref{lem:c/b=(c/a)/(b/a)} that $(f'',g'')$ is a short
  exact sequence. It therefore remains to prove that $f'' i' = b' j$ since
  the other commutativity relations $b = ji$ and
  $g''b' = j'$ hold by construction. We are going to
  apply the push-out property of the square $A'B'AD$. We have
  \[
  (f'' i') i = 0 = b'b = (b'j) i
  \qquad \text{and} \qquad
  (b'j)\bar{f} = b'f = f'' a' = (f''i') \bar{f}
  \]
  which together with
  \[
  (f''i'\bar{f})a  = (f''i'i)f' = 0 \qquad \text{and} \qquad
  (b'j\bar{f})a = f''a'a = 0 = b'bf' = (b'ji)f'
  \]
  show that both $f''i'$ and $b'j$ are solutions to the same push-out
  problem, hence they are equal. This settles case~(i).
  
  In order to prove~(ii) we start by forming the push-out under $g'$
  and $b$ so that we have the following commutative diagram with exact
  rows and columns
  \[
  \xymatrix{
    A' \ar@{ >->}[r]^{f'} \ar@{=}[d] & 
    B' \ar@{->>}[r]^{g'} \ar@{ >->}[d]^{b} \ar@{}[dr]|{\text{PO}} &
    C' \ar@{ >->}[d]^{k} \\
    A \ar@{ >->}[r]^{i} & 
    B \ar@{->>}[r]^{j} \ar@{->>}[d]^{b'} &
    D \ar@{->>}[d]^{k'} \\
    & B'' \ar@{=}[r] & B''
  }
  \]
  in which the cokernel $k'$ of $k$ is determined by $k'j = b'$ 
  and $k'k = 0$, while $i = bf'$ is a kernel of the admissible 
  epic $j$, see Remark~\ref{rem:cokernels-in-push-outs} and
  Proposition~\ref{prop:pb-adm-monic-adm-monic}.
  The push-out property
  of the square $B'C'BD$ applied to the square $B'C'BC$ yields a
  unique morphism $d': D \to C$ such that $d'k = c$ and $d'j = g$. 
  The
  diagram 
  \[
  \xymatrix{
    C' \ar@{=}[d] \ar@{ >->}[r]^{k} &
    D \ar[d]^{d'} \ar@{->>}[r]^{k'} &
    B'' \ar@{->>}[d]^{g''} \\
    C' \ar@{ >->}[r]^{c} &
    C \ar@{->>}[r]^{c'} &
    C''
  }
  \]
  has exact rows and it is commutative: Indeed, 
  $c = d'k$ holds by construction of~$d'$, while $c'd' = g''k'$ follows
  from $c'd'j = c'g = g''b' = g''k'j$ and the fact that $j$ is
  epic. We conclude from Proposition~\ref{prop:pushout-exact} that
  $DCB''C''$ is a pull-back, so $d'$ is an admissible epic and 
  so is $g = d'j$. The
  unique morphism $d: A'' \to D$ such that $k'd = f''$ and $d'd = 0$
  is a kernel of $d'$. By the pull-back property of $DCB''C''$ the diagram
  \[
  \xymatrix{
    A' \ar@{=}[r] \ar@{ >->}[d]^{a} & 
    A' \ar@{ >->}[d]^{i} \\
    A \ar[r]^{f} \ar@{->>}[d]^{a'} &
    B \ar@{->>}[r]^{g} \ar@{->>}[d]^{j} & 
    C \ar@{=}[d] \\
    A'' \ar@{ >->}[r]^{d} &
    D \ar@{->>}[r]^{d'} &
    C
  }
  \]
  is commutative as $k'(da') = f''a' = b'f = k'(jf)$ and 
  $d'(da') = 0 = gf = d'(jf)$. Notice that the hypothesis 
  that $gf = 0$ enters at this point of the argument. 
  It follows from the dual of
  Proposition~\ref{prop:pushout-exact} that $ABA''D$ is bicartesian,
  so $f$ is a kernel of $g$ by 
  Proposition~\ref{prop:pb-adm-monic-adm-monic}.    
\end{proof}

\begin{Exer}
  Consider the solid arrow diagram
  \[
  \xymatrix{
    A' \ar@{ >->}[r] \ar@{ >->}[d] &
    B' \ar@{ >->}[d] \ar@{->>}[r] &
    C' \ar@{ >.>}[d] \\
    A \ar@{ >->}[r] \ar@{->>}[d] &
    B \ar@{->>}[r] \ar@{->>}[d] &
    C \ar@{.>>}[d] \\
    A'' \ar@{ >->}[r] & 
    B'' \ar@{->>}[r] & 
    C''     
  }
  \]
  with exact rows and columns. Strengthen the Noether
  isomorphism~\ref{lem:c/b=(c/a)/(b/a)} to the statement that there
  exist unique maps $C' \to C$ and $C \to C''$ making the diagram
  commutative and the sequence $C' \mono C \epi C''$ is short exact.
\end{Exer}

\begin{Exer}
  In the situation of the $3 \times 3$-lemma prove that there are two
  exact sequences
  $A' \mono A \oplus B' \to B \epi C''$ and 
  $A' \mono B \to B'' \oplus C \epi C''$
  in the sense that the morphisms $\to$ factor as $\epi \mono$ in such
  a way that consecutive $\mono \epi$ are short exact [compare also
  with Definition~\ref{def:ex-seq-adm-mor}].
  
  \emph{Hint:}
  Apply Proposition~\ref{prop:decompose-morphisms-of-exact-sequences}
  to the first two rows in order to obtain a short exact sequence
  $A' \mono A \oplus B' \epi D$ using
  Proposition~\ref{prop:pushout-exact}. 
  Conclude from the push-out property of $DC'BC$ that
  $D \mono B$ has $C''$ as cokernel.
\end{Exer}
\newpage
\begin{Exer}[{Heller~\cite[6.2]{MR0100622}}]
  \label{exer:ses-exact}
  Let $(\scrA, \scrE)$ be an exact category and consider $\scrE$ as a
  full subcategory of $\scrA^{\to\to}$. We have shown that $\scrE$ is
  additive in Corollary~\ref{cor:exact-structure-additive}. Let
  $\scrF$ be the class of short sequences 
  \[
  \xymatrix{
    (A' \ar@{ >->}[r] \ar@{ >->}[d] &
    B' \ar@{->>}[r] \ar@{ >->}[d] &
    C') \ar@{ >->}[d] \\
    (A \ar@{ >->}[r] \ar@{->>}[d] &
    B \ar@{->>}[r] \ar@{->>}[d] &
    C) \ar@{->>}[d] \\
    (A'' \ar@{ >->}[r] &
    B'' \ar@{->>}[r]  &
    C'')
  }
  \]
  over $\scrE$
  with short exact columns [we write $(A \mono B \epi C)$ to
  indicate that we think of the sequence as an object of $\scrE$]. 
  Prove that $(\scrE,\scrF)$ is an exact category.
%
\end{Exer}

\begin{Rem}
  The category of short exact sequences in a nonzero abelian category is
  \emph{not} abelian, see \cite[XII.6, p.~375]{MR0156879}.
\end{Rem}

\begin{Exer}[{K\"unzer's Axiom, cf. e.g.~\cite{MR2366951}}]
  \leavevmode
  \begin{enumerate}[(i)]
    \item
      If $f: A \to B$ is a morphism and $g: B \epi C$ is an admissible
      epic such that $h = gf: A \mono C$ is an admissible monic then $f$ is an
      admissible monic and the morphism $\Ker{g} \to \Coker{f}$ 
      is an admissible monic as well.
   
      \emph{Hint:}
      Form the pull-back $P$ over $h$ and $g$, use 
      Proposition~\ref{prop:pb-adm-monic-adm-monic} and factor $f$
      over $P$ to see the first part 
      (see also Remark~\ref{rem:split-special-idempotents}). 
      For the second part use the
      Noether isomorphism~\ref{lem:c/b=(c/a)/(b/a)}.

    \item 
      Let $\scrE$ be a class of kernel-cokernel pairs in the additive
      category $\scrA$. Assume that $\scrE$ is closed under
      isomorphisms and contains the split exact sequences. If 
      $\scrE$ enjoys the property of point~(i) and its dual 
      then it is an exact structure.

      \emph{Hint:}
      Let $f:A \mono B$ be an admissible monic and let $a: A \to A'$ be
      arbitrary. The push-out axiom follows from the commutative diagram
      \[
      \xymatrix{
        A \ar@{ >->}[rr]^-{f} \ar@{ >->}[dr]_-{\mat{a \\ f}} & &
        B \\
        &
        A' \oplus B \ar@{->>}[ur]_-{\mat{0 & 1}} 
        \ar@{->>}[dr]^-{\mat{f' & -b}} \\
        A' \ar@{ >->}[rr]^-{f'} \ar@{ >->}[ur]^-{\mat{1 \\ 0}} & & B'
      }
      \]
      in which $\mat{a \\ f}$ and $f'$ are admissible monics by~(i) and 
      $\mat{f' & -b}$ is a cokernel of $\mat{a \\ f}$. Next observe
      that the dual of~(i) implies that if in
      addition $a$ is an admissible epic then so is $b$. In order to
      prove the composition axiom, let $f$ and $g$ be admissible
      monics and choose a cokernel $f'$ of $f$. Form the push-out
      under $g$ and $f'$ and verify that $gf$ is a kernel of the
      push-out of $f'$.
  \end{enumerate}
\end{Exer}

\section{Quasi-Abelian Categories}
\label{sec:quasi-ab-cats}

\begin{Def}
  An additive category $\scrA$ is called \emph{quasi-abelian} if
  \begin{enumerate}[(i)]
    \item
      Every morphism has a kernel and a cokernel.
    \item
      The class of kernels is stable under push-out along arbitrary
      morphisms and the class of cokernels is stable under pull-back
      along arbitrary morphisms.
  \end{enumerate}
\end{Def}

\begin{Rem}
  The concept of a quasi-abelian category is self-dual, that is to say
  $\scrA$ is quasi-abelian if and only if $\scrA^{\opp}$ is quasi-abelian.
\end{Rem}

\begin{Exer}
  Let $\scrA$ be an additive category with kernels.
  Prove that every pull-back of a kernel is a kernel.
\end{Exer}

\begin{Prop}[{Schneiders~\cite[1.1.7]{MR1779315}}]
  The class $\scrE_{\max}$ of all kernel-cokernel pairs in a quasi-abelian
  category is an exact structure.
\end{Prop}
\begin{proof}
  It is clear that $\scrE_{\max}$ is closed under isomorphisms and that the
  classes of kernels and cokernels contain the identity
  morphisms. The pull-back and push-out axioms are part of the
  definition of quasi-abelian categories. By duality it only remains
  to show that the class of cokernels is closed under composition.
  So let $f: A \epi B$ and $g: B \epi C$ be
  cokernels and put $h = gf$.
  In the diagram
  \[
  \xymatrix{
    \Ker{f} \ar@{.>}[r]^-{u} \ar@{=}[d] &
    \Ker{h} \ar@{.>}[r]^-{v} \ar@{ >->}[d]^-{\ker{h}} &
    \Ker{g} \ar@{ >->}[d]^-{\ker{g}} \\
    \Ker{f} \ar@{ >->}[r]^-{\ker{f}} & 
    A \ar@{->>}[r]^-{f} \ar[d]^-{h} &
    B \ar@{->>}[d]^-{g} \\
    & C \ar@{=}[r] & C
  }
  \]
  there exist unique morphisms $u$ and $v$ making it
  commutative. The upper right hand square is a pull-back, so $v$ is a
  cokernel and $u$ is its kernel. But then it follows by duality that
  the upper right hand square is also a push-out and this together with the
  fact that $h$ is epic implies that $h$ is a cokernel of $\ker{h}$.
\end{proof}

\begin{Rem}
  Note that we have just re-proved the Noether
  isomorphism~\ref{lem:c/b=(c/a)/(b/a)} in the special case of
  quasi-abelian categories.
\end{Rem}

\begin{Def}
  The \emph{coimage} of a morphism $f$ in a category with kernels and
  cokernels is $\Coker{(\ker{f})}$, while the \emph{image} is defined to
  be $\Ker{(\coker{f})}$. The \emph{analysis}
  (cf. \cite[IX.2]{MR0156879}) of $f$ is the commutative diagram
  \[
  \xymatrix@R=0.5pc{
    & A \ar[rrr]^{f} \ar@{->>}[dr]_-{\coim{f}} & & & B
    \ar@{->>}[dr]^-{\coker{f}} \\
    \Ker{f} \ar@{ >->}[ur]^-{\ker{f}} & & \Coim{f} \ar[r]^-{\hat{f}} & 
    \Im{f} \ar@{ >->}[ur]_-{\im{f}} & & \Coker{f}
  }
  \]
  in which $\hat{f}$ is uniquely determined by requiring the
  diagram to be commutative.
\end{Def}

\begin{Rem}
  The difference between quasi-abelian categories and abelian
  categories is that in the quasi-abelian case 
  the canonical morphism $\hat{f}$ in the
  analysis $f$ is not in general an isomorphism. Indeed, it is easy to
  see that a quasi-abelian category is abelian \emph{provided} that
  $\hat{f}$ is always an isomorphism. Equivalently, not every
  monic is a kernel and not every epic is a cokernel.
\end{Rem}

\begin{Prop}[{\cite[1.1.5]{MR1779315}}]
  Let $f$ be a morphism in the quasi-abelian category $\scrA$. The
  canonical morphism $\hat{f}: \Coim{f} \to \Im{f}$ is monic and epic.
\end{Prop}
\begin{proof}
  By duality it suffices to check that the morphism
  $\bar{f}$ in the diagram
  \[
  \xymatrix@R=0.5pc{
    A \ar[rr]^-{f} \ar@{->>}[dr]_-{j} & & B \\
    & \Coim{f} \ar[ur]_-{\bar{f}}
  }
  \]
  is monic. Let $x: X \to \Coim{f}$ be a morphism such that
  $\bar{f}x = 0$. The pull-back $y: Y \to A$ of $x$ along $j$
  satisfies $fy = 0$, so $y$ factors over $\Ker{f}$ and hence 
  $jy = 0$. But then the map $Y \epi X \to \Coim{f}$ is zero as well, 
  so $x = 0$.
\end{proof}

\begin{Rem}
  Every morphism $f$ in a quasi-abelian category $\scrA$
  has two epic-monic factorizations, one over $\Coim{f}$ and one
  over $\Im{f}$. The quasi-abelian category $\scrA$ is abelian if and
  only if the two factorizations coincide for all
  morphisms $f$.
\end{Rem}

\begin{Rem}
  An additive category with kernels and cokernels is called
  \emph{semi-abelian} if the canonical morphism $\Coim{f} \to \Im{f}$
  is always monic and epic. We have just proved that quasi-abelian
  categories are semi-abelian. It may seem obvious that the concept of
  semi-abelian categories is strictly weaker than the concept of a
  quasi-abelian category. However, it is surprisingly delicate to come up
  with an explicit example. This led Ra\u\i{}kov to conjecture that
  every semi-abelian category is quasi-abelian. A counterexample to
  this conjecture was recently found by
  Rump~\cite{rump-counterexample}.
\end{Rem}

\begin{Rem}
  We do not develop the theory of quasi-abelian categories any
  further. The interested reader may consult
  Schneiders~\cite{MR1779315}, Rump~\cite{MR1856638} and the
  references therein.
\end{Rem}

\section{Exact Functors}
\label{sec:exact-functors}

\begin{Def}
  Let $(\scrA,\scrE)$ and $(\scrA',\scrE')$ be exact categories. An
  (additive) functor $F: \scrA \to \scrA'$ is called \emph{exact} if
  $F(\scrE) \subset \scrE'$. The functor $F$ \emph{reflects exactness}
  if $F(\sigma) \in \scrE'$ implies $\sigma \in \scrE$ for all
  $\sigma \in \scrA^{\to\to}$.
\end{Def}

\begin{Prop}
  \label{prop:exact-functors-push-out-pull-back}
  An exact functor preserves push-outs along admissible monics and
  pull-backs along admissible epics.
\end{Prop}

\begin{proof}
  An exact functor preserves admissible monics and admissible epics,
  in particular it preserves diagrams of type
  \[
  \vcenter{
    \xymatrix{
      \ar[d] \ar@{ >->}[r] & \ar[d] \\
      \ar@{ >->}[r] &
    }
  }
  \qquad \text{and} \qquad
  \vcenter{
    \xymatrix{
      \ar[d] \ar@{->>}[r] & \ar[d] \\
      \ar@{->>}[r] &
    }
  }
  \]
  so the result follows immediately from
  Proposition~\ref{prop:pushout-exact} and its dual.
\end{proof}

The following exercises show how one can induce new exact structures
using functors satisfying certain exactness properties.

\begin{Exer}[{Heller~\cite[7.3]{MR0100622}}]
  \label{exer:pull-back-ex-str}
  Let $F: (\scrA,\scrE) \to (\scrA',\scrE')$ be an exact
  functor and let $\scrF'$ be another exact structure on
  $\scrA'$. Then $\scrF = \{\sigma \in \scrE\,:\,F(\sigma) \in
  \scrF'\}$ is an exact structure on $\scrA$.
\end{Exer}

\begin{Rem}[Heller]
  The prototypical application of the previous exercise is the following: 
  A (unital) ring homomorphism $\varphi: R' \to R$ yields an exact functor
  $\varphi^{\ast}: \leftidx_{{R}}{\Mod} \to \leftidx_{{R'}}{\Mod}$ of the
  associated module categories. Let $\scrF'$ be the class of split
  exact sequences on $\leftidx_{{R'}}{\Mod}$. The induced structure
  $\scrF$ on $\leftidx_{R}{\Mod}$
  consisting of sequences $\sigma$ such that $\varphi^{\ast}(\sigma)$
  is split exact is the \emph{relative exact structure} with respect
  to $\varphi$. This structure is used in particular to define the
  relative derived functors such as the relative $\Tor$ and $\Ext$ functors.
\end{Rem}

\begin{Exer}[K\"unzer]
  \label{exer:largest-ex-structure-fct-ex}
  Let $F: (\scrA, \scrE) \to (\scrA',\scrE')$ be a functor which preserves
  admissible kernels, i.e., for every morphism
  $f:B \to C$ with an admissible monic $k: A \mono B$
  as kernel, the morphism $F(k)$ is an admissible monic and a kernel of
  $F(f)$. Let $\scrF = \{\sigma \in \scrE\,:\,F(\sigma) \in \scrE'\}$
  be the largest subclass of $\scrE$ on which $F$ is exact. Prove that
  $\scrF$ is an exact structure.

  \emph{Hint:} Axioms [E$0$], [E$0^{\opp}$] and [E$1^{\opp}$] are easy.
  To check axiom~[E$1$] use the 
  obscure axiom~\ref{prop:obscure-axiom} and the
  $3 \times 3$-lemma~\ref{cor:3x3-lemma}. 
  Axiom~[E$2$] follows from the obscure axiom~\ref{prop:obscure-axiom} and
  Proposition~\ref{prop:pushout-exact}~(iv),
  while axiom~[E$2^{\opp}$] follows from the fact that $F$ preserves
  certain pull-back squares. 
\end{Exer}

\begin{Exer}
  Let $\scrP$ be a set of objects in the exact category
  $(\scrA,\scrE)$. Consider the class $\scrE_{\scrP}$ consisting of the
  sequences $A' \mono A \epi A''$ in $\scrE$ such that
  \[
  \Hom_{\scrA}{(P,A')} \mono \Hom_{\scrA}{(P,A)} \epi \Hom_{\scrA}{(P,A'')}
  \]
  is an exact sequence of abelian groups for all $P \in \scrP$.
  Prove that $\scrE_{\scrP}$
  is an exact structure on $\scrA$.

  \emph{Hint:} Use Exercise~\ref{exer:largest-ex-structure-fct-ex}.
\end{Exer}

\section{Idempotent Completion}
\label{sec:idempotent-completion}

In this section we discuss Karoubi's construction of `the' idempotent 
completion of an additive category, see~\cite[1.2]{MR0238927} and show
how to extend it to exact categories. 
Admittedly, the constructions and arguments presented here
are quite obvious (once the definitions are given) and thus rather
boring, but as the author is unaware of a reasonably complete
exposition it seems worthwhile to outline the details. A different
account for small categories (not necessarily additive)
is given in Borceux~\cite[Proposition~6.5.9, p.~274]{MR1291599}. It appears
that the latter approach is due to M.~Bunge~\cite{bungesthesis}.

\begin{Def}
  An additive category $\scrA$ is 
  \emph{idempotent complete}~\cite[1.2.1, 1.2.2]{MR0238927} if for every
  idempotent $p: A \to A$, \emph{i.e.} $p^{2} = p$,
  there is a decomposition $A \cong K \oplus I$ of $A$
  such that $p \cong \mat{0 & 0 \\ 0 & 1}$. 
\end{Def}

\begin{Rem}
  \label{rem:ic-iff-p-has-kernel}
  The additive category $\scrA$ is idempotent complete if
  and only if every idempotent has a kernel.
  
  Indeed, suppose that every idempotent has a kernel. Let $k: K \to A$
  be a kernel of the idempotent $p: A \to A$ and let $i:I \to A$ be a
  kernel of the idempotent $1 - p$. Because $p(1-p) = 0$,
  we have $(1-p) = kl$ for a unique $l: A \to K$ and because
  $(1-p)p = 0$ we have $p = ij$ for a unique $j: A \to
  I$. Since $k$ is monic and $kli = (1-p)i = 0$ we have $li = 0$ and
  because $klk = (1-p)k = pk + (1-p)k = k$ we have $lk =
  1_{K}$. Similarly, $jk = 0$ and $ji = 1_{I}$. Therefore
  $\mat{k & i}: K \oplus I \to A$ and $\mat{l \\ j}: A \to K \oplus I$
  are inverse to each other and
  $\mat{l \\ j}p \mat{k & i} = \mat{l \\ j} ij \mat{k & i} = 
  \mat{0 & 0 \\ 0 & 1_{I}}$ as desired. Notice that we have
  constructed an analysis of $p$:
  \[
  \xymatrix@R=0.5pc{
    & A \ar[rr]^{p} \ar@{->>}[dr]_{j} 
    & & A \ar@{->>}[dr]^{l} \\
    K \ar@{ >->}[ur]^{k} & & I \ar@{ >->}[ur]_{i} & & K,
  }
  \]
  in particular $k:K\mono A$ is a kernel of $p$ and 
  $i: I \mono A$ is an image of $p$.
  The converse direction is even more obvious.
\end{Rem}

\begin{Rem}
  \label{rem:constr-idemp-compl}
  Every additive category $\scrA$ can be fully faithfully
  embedded into an idempotent complete
  additive category $\scrA^{\wedge}$.
  
  The objects of $\scrA^{\wedge}$ are the pairs
  $(A, p)$ consisting of an object $A$ of $\scrA$ and an
  idempotent $p: A \to A$ while the sets of morphisms are 
  \[
  \Hom_{\scrA^{\wedge}}{((A,p), (B,q))} = q \circ \Hom_{\scrA} {(A, B)}
  \circ p
  \]
  with the composition induced by the composition in $\scrA$.
  It is easy to see that $\scrA^{\wedge}$ is additive with biproduct
  $(A, p) \oplus (A',p') = (A \oplus A', p \oplus p')$ and obviously the
  functor $i_{\scrA}:\scrA \to \scrA^{\wedge}$ given by
  $i_{\scrA}{(A)} = (A,1_{A})$ and $i_{\scrA}(f) = f$ is fully faithful. 
  In order to see that $\scrA^{\wedge}$ is idempotent 
  complete, suppose $pfp$ is an idempotent of $(A,p)$ in
  $\scrA^{\wedge}$. \emph{A fortiori} $pfp$ is an idempotent of $\scrA$
  and the object $(A,p)$ is isomorphic to the direct sum
  $(A,p - pfp) \oplus (A,pfp)$
  via the morphisms $\mat{p - pfp \\ pfp}$ and $\mat{p-pfp & & pfp}$. 
  The equation $\mat{p - pfp \\ pfp} pfp \mat{p - pfp & & pfp} = 
  \mat{0 & 0 \\ 0 & pfp}$ proves $\scrA^{\wedge}$ to be
  idempotent complete.
\end{Rem}

\begin{Def}
  The pair $(\scrA^{\wedge}, i_{\scrA})$ constructed in
  Remark~\ref{rem:constr-idemp-compl}  
  is called the \emph{idempotent completion} of $\scrA$.
\end{Def}

The next goal is to characterize the pair
$(\scrA^{\wedge}, i_{\scrA})$ by a universal property
(Proposition~\ref{prop:idempotent-completion-additive}).
We first work out some nice properties of the explicit
construction.

\begin{Rem}
  \label{rem:ic-of-ic-cat-is-ic}
  If $\scrA$ is idempotent
  complete then $i_{\scrA}: \scrA \to \scrA^{\wedge}$ is an
  equivalence of categories. In order to construct a quasi-inverse
  functor of $i_{\scrA}: \scrA \to \scrA^{\wedge}$, choose for each
  idempotent $p: A \to A$ a kernel $K_{p}$, an image $I_{p}$ and
  morphisms $i_{p}, j_{p}, k_{p}, l_{p}$ as in 
  Remark~\ref{rem:ic-iff-p-has-kernel} and map the
  object $(A,p)$ of $\scrA^{\wedge}$ to $I_{p}$. 
  A morphism $(A,p) \to (B,q)$ of $\scrA^{\wedge}$ can be written
  as $qfp$ and map it to $j_{q}qfpi_{p} = j_{q}fi_{p}$. Obviously,
  this yields a quasi-inverse functor of 
  $i_{\scrA}: \scrA \to \scrA^{\wedge}$.
\end{Rem}

\begin{Rem}
  \label{rem:functoriality-of-ic}
  A functor $F: \scrA \to \scrB$ yields a functor
  $F^{\wedge}: \scrA^{\wedge} \to \scrB^{\wedge}$, simply by setting
  $F^{\wedge}(A,p) = (F(A),F(p))$ and 
  $F^{\wedge}(qfp)= F(q) F(f) F(p)$. 
  Obviously, $F^{\wedge}i_{\scrA} = i_{\scrB}F$ and
  $(GF)^{\wedge} = G^{\wedge} F^{\wedge}$. 
\end{Rem}

\begin{Rem}
  \label{rem:2-functoriality-of-ic}
  A natural transformation
  $\alpha: F \Rightarrow G$ of functors $\scrA \to \scrB$ yields a
  unique natural transformation $\alpha^{\wedge}: F^{\wedge} \Rightarrow
  G^{\wedge}$.

  Observe first that a natural transformation $\alpha': F'
  \Rightarrow G'$ of functors $\scrA^{\wedge} \to \scrB^{\wedge}$ is
  completely determined by its values on $i_{\scrA}{(\scrA)}$ by the
  following argument.
  Every object $(A,p)$ of $\scrA^{\wedge}$ is canonically a
  retract of $(A,1_{A})$ via the morphisms $s: (A,p) \to (A,1_{A})$ and
  $r: (A, 1_{A}) \to (A,p)$ given by $p \in
  \Hom_{\scrA}{(A,A)}$. Therefore, by naturality, we must have
  \[
  \alpha'_{(A,p)} = \alpha'_{(A,p)} F'(r) F'(s) = G'(r)
  \alpha'_{(A,1_{A})} F'(s),
  \]
  so $\alpha'$ is completely determined by its values on
  $i_{\scrA}{(\scrA)}$. Now given a natural transformation
  $\alpha: F \Rightarrow G$ of functors $\scrA \to \scrB$ put
  \[
  \alpha^{\wedge}_{(A,p)} = 
  G^{\wedge}(r) i_{\scrB}(\alpha_{A}) F^{\wedge}(s)
  \]
  which is simply the element $G(p) \alpha_{A} F(p)$ in
  \[
  \Hom_{\scrB^{\wedge}} {(F^{\wedge}(A,p), G^{\wedge}(A,p))} = 
  G(p)\circ \Hom_{\scrB}(F(A),G(A)) \circ F(p).
  \]
  It is easily checked that this definition of $\alpha^{\wedge}$ indeed yields
  a natural transformation $F^{\wedge} \Rightarrow G^{\wedge}$ as
  desired.
\end{Rem}

\begin{Rem}
  \label{rem:strict-2-functoriality-of-ic}
  The assignment $\alpha \mapsto \alpha^{\wedge}$ is
  compatible with vertical and horizontal composition 
  (see e.g.~\cite[II.5, p.~42f]{MR1712872}): For functors
  $F,G,H: \scrA \to \scrB$ and  natural transformations
  $\alpha:F \Rightarrow G$ and $\beta: G \Rightarrow H$ we have 
  $(\beta \circ \alpha)^{\wedge} = \beta^{\wedge} \circ \alpha^{\wedge}$ 
  while for functors $F,G: \scrA \to \scrB$ and
  $H,K:\scrB \to \scrC$ with natural transformations $\alpha:F
  \Rightarrow G$ and $\beta: H \Rightarrow K$ 
  we have $(\beta \ast \alpha)^{\wedge} = \beta^{\wedge} \ast
  \alpha^{\wedge}$.
\end{Rem}

\begin{Rem}
  \label{rem:uniqueness-of-fctrs-to-ic-cat}
  A functor $F: \scrA^{\wedge} \to \scrI$ to an idempotent complete
  category is determined up to unique isomorphism by its values on
  $i_{\scrA}{(\scrA)}$. A natural transformation $\alpha: F
  \Rightarrow G$ of functors $\scrA^{\wedge} \to \scrI$ is determined
  by its values on $i_{\scrA}{(\scrA)}$.
  
  Indeed, exhibit each $(A,p)$ as a retract of
  $(A,1)$ as in Remark~\ref{rem:2-functoriality-of-ic}. 
  Consider $p$ as an
  idempotent of $(A,1)$, so $F(p)$ is an idempotent of
  $F(A,1)$. Choosing an image $I_{F(p)}$ of $F(p)$ as in
  Remark~\ref{rem:ic-iff-p-has-kernel}, it is clear that
  the functor $F$ must
  map $(A,p)$ to $I_{F(p)}$ and is thus determined up to a
  unique isomorphism. The claim about natural transformations is
  analogous to the argument in Remark~\ref{rem:2-functoriality-of-ic}.
\end{Rem}

\begin{Prop}
  \label{prop:idempotent-completion-additive}
  The functor $i_{\scrA}: \scrA \to \scrA^{\wedge}$ is $2$-universal among
  functors from $\scrA$ to idempotent complete categories:
  \begin{enumerate}[(i)]
    \item
      For every functor $F: \scrA \to \scrI$ to an idempotent complete
      category $\scrI$ there exists a functor
      $\widetilde{F}: \scrA^{\wedge} \to \scrI$ and a
      natural isomorphism 
      $\alpha: F \Rightarrow \widetilde{F}i_{\scrA}$.

    \item
      Given a functor $G: \scrA^{\wedge} \to \scrI$
      and a natural transformation $\gamma: F \Rightarrow Gi_{\scrA}$
      there is a unique natural transformation $\beta:\widetilde{F}
      \Rightarrow G$ such that $\gamma = \beta \ast \alpha$.
  \end{enumerate}
\end{Prop}

\begin{proof}[Sketch of the Proof]
  Given a functor $F: \scrA \to \scrI$, the construction outlined in
  Remark~\ref{rem:uniqueness-of-fctrs-to-ic-cat} yields candidates
  for $\widetilde{F}: \scrA^{\wedge} \to \scrI$ and 
  $\alpha: F \Rightarrow \widetilde{F}i_{\scrA}$ and we leave it to the
  reader to check that this works. Once $\widetilde{F}$ and $\alpha$
  are fixed, $\widetilde{\beta} := \gamma \ast \alpha^{-1}$ 
  yields a natural transformation
  $\widetilde{F}i_{\scrA} \Rightarrow G i_{\scrA}$ and the procedure
  in Remark~\ref{rem:uniqueness-of-fctrs-to-ic-cat} shows what an
  extension $\beta: \widetilde{F} \Rightarrow G$ of
  $\widetilde{\beta}$ 
  must look like and
  again we leave it to the reader to check that this works.
\end{proof}

\begin{Cor}
  \label{cor:ic-equiv-fctr-cat}
  Let $\scrA$ be a small additive category. The functor $i_{\scrA}:
  \scrA \to \scrA^{\wedge}$ induces an equivalence of functor
  categories
  \[
  (i_{\scrA})^{\ast}: \Hom{(\scrA^{\wedge},\scrI)} 
  \xrightarrow{\simeq}
  \Hom{(\scrA,\scrI)}
  \]
  for every idempotent complete category $\scrI$.
\end{Cor}
\begin{proof}
  Point~(i) of Proposition~\ref{prop:idempotent-completion-additive}
  states that $(i_{\scrA})^{\ast}$ is essentially surjective and
  it follows from point~(ii) that it is fully faithful, hence it is an
  equivalence of categories.
\end{proof}
\if{0}
\begin{proof}
  There an obvious way to extend $F$ to a functor
  $F^{\wedge}: \scrA^{\wedge} \to \scrI^{\wedge}$ which is given on
  objects by $F^{\wedge}(A,p) = (F(A),F(p))$.
  Since $\scrI$ is
  idempotent complete, the functor $i_{\scrI}:\scrI \to \scrI^{\wedge}$ is an
  equivalence, so we can choose a quasi-inverse
  $j_{\scrI}:\scrI^{\wedge} \to \scrI$ of $i_{\scrI}$ 
  and we obtain the desired functor
  by setting $\widetilde{F} = j_{\scrI} F^{\wedge}$.
  
  Next observe that given two functors $G,G': \scrA^{\wedge} \to
  \scrI$, a natural transformation $\gamma: G \Rightarrow G'$ is
  completely determined by its restriction to
  $i_{\scrA}{(\scrA)}$. Indeed, every object $(A,p)$ in
  $\scrA^{\wedge}$ is a retract of $(A, 1_{A})$ via the morphisms
  $s:(A,p) \to (A,1_{A})$ and $r:(A,1_{A}) \to (A,p)$ given by 
  $p \in \Hom_{\scrA}{(A,A)}$. Therefore, by naturality, we must have
  \[
  \gamma_{(A,p)} = \gamma_{(A,p)} G(r) G(s) = G'(r) \gamma_{(A,1_{A})} G(s).
  \]
  Conversely, given a natural transformation $\gamma: Gi_{\scrA}
  \Rightarrow G'i_{\scrA}$ the expression
  \[
  \gamma_{(A,p)} := G'(r) \gamma_{A} G(s)
  \]
  is easily checked to give a natural transformation $G \Rightarrow G'$.
  isomorphism $j_{\scrI} i_{\scrI} \Rightarrow \id_{\scrI}$ yields the
  natural isomorphism 
  $\widetilde{\alpha}: F  \Rightarrow  \widetilde{F}i_{\scrA}$. 
  This settles point~(i), we leave point~(ii) as an exercise for the reader.
\end{proof}

\begin{Rem}
  \label{rem:idemp-2-adjunction}
  Another way of phrasing the proposition is: Let $\scrA$ be a small
  additive category and let  $\scrI$ be an idempotent complete
  category. The inclusion functor $i_{\scrA}: \scrA \to
  \scrA^{\wedge}$ induces an equivalence of functor categories 
  \[
  (i_{\scrA})^{\ast} : \Hom{(\scrA^{\wedge}, \scrI)} \xrightarrow{\simeq}
  \Hom{(\scrA, \scrI)}.
  \]
\end{Rem}
\fi

\begin{Exm}
  Let $\scrF$ be the category of (finitely generated) 
  free modules over a ring $R$. Its
  idempotent completion $\scrF^{\wedge}$ is equivalent to the category of
  (finitely generated) projective modules over $R$.
\end{Exm}

Let now $(\scrA, \scrE)$ be an exact category. Call a sequence in
$\scrA^{\wedge}$ short exact if it is a direct summand in
$(\scrA^{\wedge})^{\to\to}$ of a sequence in $\scrE$ and denote the class of
short exact sequences in $\scrA^{\wedge}$ by $\scrE^{\wedge}$.

\begin{Prop}
  \label{prop:idempotent-completion-exact}
  The class $\scrE^{\wedge}$ is an exact structure on $\scrA^{\wedge}$.
  The inclusion functor
  $i_{\scrA}:(\scrA, \scrE) \to (\scrA^{\wedge}, \scrE^{\wedge})$ preserves
  and reflects exactness and is $2$-universal among exact functors
  to idempotent complete exact categories:
  \begin{enumerate}[(i)]
    \item
      Let $F: \scrA \to \scrI$ be an exact functor to an idempotent
      complete exact category $\scrI$. There exists an exact functor
      $\widetilde{F}:\scrA^{\wedge} \to \scrI$ together with a natural
      isomorphism
      $\alpha: F \Rightarrow \widetilde{F} i_{\scrA}$.

    \item
      Given another exact functor $G: \scrA \to \scrI$ together with a
      natural transformation $\gamma: F \Rightarrow G i_{\scrA}$,
      there is a unique natural transformation 
      $\beta: \widetilde{F} \Rightarrow G$ 
      such that $\gamma = \beta \ast \alpha$.
  \end{enumerate}
\end{Prop}
\begin{proof}
  To prove that $\scrE^{\wedge}$ is an exact structure is
  straightforward but rather tedious, so we skip it.\footnote{%
    Thomason~\cite[A.9.1~(b)]{MR1106918} gives a short argument
    relying on the embedding into an abelian category, but it can be
    done by completely elementary means as well.} Given this, it is
  clear that the functor $\scrA \to \scrA^{\wedge}$ is exact and
  reflects exactness.
  If $F: \scrA \to \scrI$ is an exact functor to an idempotent
  complete exact category then, as every sequence in $\scrE^{\wedge}$
  is a direct summand of a sequence in $\scrE$, an extension
  $\widetilde{F}$ of $F$ as provided by
  Proposition~\ref{prop:idempotent-completion-additive} 
  must carry exact sequences in $\scrA^{\wedge}$ to exact sequences in
  $\scrI$. Thus statements~(i) and~(ii) follow from the corresponding
  statements in Proposition~\ref{prop:idempotent-completion-additive}.
\end{proof}

\begin{Cor}
  For a small exact category $(\scrA, \scrE)$, the exact functor
  $i_{\scrA}$ induces an equivalence of the
  categories of exact functors
  \[
  (i_{\scrA})^{\ast}: \Hom_{\Ex}{((\scrA^{\wedge}, \scrE^{\wedge}),
    \scrI)} \xrightarrow{\simeq} \Hom_{\Ex}{((\scrA,\scrE), \scrI)}
  \]
  to every idempotent complete exact category $\scrI$. \qed
\end{Cor}
\if{0}
\begin{Rem}
  One can interpret Proposition~\ref{prop:idempotent-completion-exact}
  by saying that the equivalence of categories of
  restricts to an equivalence of
  the full subcategories of exact functors
  $(i_{\scrA})^{\ast}: \Hom_{\Ex}{((\scrA^{\wedge}, \scrE^{\wedge}), \scrI)} 
  \xrightarrow{\simeq} \Hom_{\Ex}{((\scrA,\scrE), \scrI)}$.
\end{Rem}
\fi
\section{Weak Idempotent Completeness}
\label{sec:weak-idempotent-compl}

Thomason introduced in~\cite[A.5.1]{MR1106918} the notion of an exact
category with ``weakly split idempotents''.
It turns out that this is a property of the underlying additive
category rather than the exact structure.

Recall that in an arbitrary category a morphism $r: B \to C$ is called
a \emph{retraction} if there exists a \emph{section} $s: C \to B$ of
$r$ in the sense that $rs = 1_{C}$. 
Dually, a morphism $c: A \to B$ is a
\emph{coretraction} if it admits a section $s: B \to A$, i.e.,
$sc = 1_{A}$. Observe that retractions are epics and coretractions
are monics. Moreover, a section of a retraction is a coretraction and
a section of a coretraction is a retraction.

\begin{Lem}
  \label{lem:weakly-idempotent-complete}
  In an additive category $\scrA$ the following are equivalent:
  \begin{enumerate}[(i)]
    \item
      Every coretraction has a cokernel.
    \item
      Every retraction has a kernel.
  \end{enumerate}
\end{Lem}

\begin{Def}
  \label{def:weakly-idempotent-complete}
  If the conditions of the previous lemma hold then $\scrA$ is said to
  be \emph{weakly idempotent complete}.
\end{Def}

\begin{Rem}
  Freyd~\cite{MR0206069} uses the more descriptive terminology
  \emph{retracts have complements} for weakly idempotent complete
  categories. He proves in particular that a weakly idempotent
  complete category with countable coproducts is idempotent complete.
\end{Rem}

\begin{Rem}
  \label{rem:split-special-idempotents}
  Assume that $r: B \to C$ is a retraction with section $s: C \to
  B$. Then
  $sr: B \to B$ is an idempotent. Let us prove that this idempotent
  gives rise to a splitting of $B$ if $r$ admits a kernel
  $k: A \to B$.

  Indeed, since $r(1_{B} - sr) = 0$, there is
  a unique morphism $t: B \to A$ such that $kt = 1_{B} - sr$. It
  follows that $k$ is a coretraction because $ktk = (1_{B} - sr)k = k$
  implies that $tk = 1_{A}$. Moreover $kts = 0$, so $ts = 0$, hence
  $\mat{k & s} : A \oplus C \to B$
  is an isomorphism with inverse $\mat{t \\ r}$. In particular, the
  sequences $A \to B \to C$ and $A \to A \oplus C \to C$ are
  isomorphic.
\end{Rem}

\begin{proof}[Proof of Lemma~\ref{lem:weakly-idempotent-complete}]
  By duality it suffices to prove that~(ii) implies~(i).
  
  Let $c: C \to B$ be a coretraction with section $s$. Then $s$ is
  a retraction and, assuming~(ii), it admits a kernel $k: A \to B$. By
  the discussion in Remark~\ref{rem:split-special-idempotents},
  $k$ is a coretraction with 
  section $t: B \to A$ and it is obvious that $t$ is a cokernel of $c$.
\end{proof}

\begin{Cor}
  Let $(\scrA, \scrE)$ be an exact category. The following are
  equivalent:
  \begin{enumerate}[(i)]
    \item
      The additive category $\scrA$ is weakly idempotent complete.

    \item
      Every coretraction is an admissible monic.
      
    \item
      Every retraction is an admissible epic.
  \end{enumerate}
\end{Cor}
\begin{proof}
  It follows from Remark~\ref{rem:split-special-idempotents} that every
  retraction $r: B \to C$ admitting a kernel gives rise to a
  sequence $A \to B \to C$ which is isomorphic to the split exact sequence
  $A \mono A \oplus C \epi C$, hence $r$ is an admissible epic by
  Lemma~\ref{lem:split-sequences-exact}, whence~(i) implies~(iii).
  By duality~(i) implies~(ii) as well.
  Conversely, every admissible monic has a cokernel and every
  admissible epic has a kernel, hence~(ii)
  and~(iii) both imply~(i).
\end{proof}

In a weakly idempotent complete exact category the obscure axiom
(Proposition~\ref{prop:obscure-axiom}) has an easier statement---this
is Heller's cancellation axiom~\cite[(P2), p.~492]{MR0100622}:

\begin{Prop}
  \label{prop:weakly-split-obscure-axiom}
  Let $(\scrA,\scrE)$ be an exact category. The following are
  equivalent:
  \begin{enumerate}[(i)]
    \item
      The additive category $\scrA$ is weakly idempotent complete.
    \item
      Consider two morphisms $g: B \to C$ and $f: A \to B$.
      If $gf: A \epi C$ is an admissible
      epic then $g$ is an admissible epic.
  \end{enumerate}
\end{Prop}
\begin{proof}
  (i) $\Rightarrow$ (ii):
  Form the pull-back over $g$ and $gf$ and consider the diagram 
  \[ 
  \xymatrix{
    A  \ar@/_1.2pc/[ddr]_-{1_{A}} \ar@/^1.2pc/[drr]^-{f} 
    \ar@{.>}[dr]^{\exists!} \\
    & B' \ar@{->>}[r] \ar[d]_-{g'}
    \ar@{}[dr]|{\text{PB}} & B \ar[d]^{g}
    \\
    & A \ar@{->>}[r]^{gf} & C   
  }
  \]
  which proves $g'$ to be a retraction, so
  $g'$ has a kernel $K' \to B'$. Because the diagram is a pull-back, the
  composite $K' \to B' \to B$ is a kernel of $g$ and now the dual of
  Proposition~\ref{prop:obscure-axiom} applies to yield that
  $g$ is an admissible epic.
  
  For the implication (ii) $\Rightarrow$ (i) simply observe that (ii)
  implies that retractions are admissible epics.
\end{proof}

\begin{Cor}
  \label{cor:cancellation}
  An exact category is weakly idempotent complete if and only if it
  has the following property: all morphisms $g: B \to C$ for which
  there is a commutative diagram
  \[
  \xymatrix{
    A \ar@{->>}[r]^{f} \ar@{->>}[dr]_{gf} & B \ar[d]^{g} \\
    & C
  }
  \]
  are admissible epics.
\end{Cor}
\begin{proof}
  A weakly idempotent complete exact category enjoys the stated
  property by Proposition~\ref{prop:weakly-split-obscure-axiom}.
  
  Conversely, let $r: B \to C$ and $s: C \to B$ be such that $rs =
  1_{C}$. We want to show that $r$ is an admissible epic.
  The sequences
  \[
  \vcenter{
    \xymatrix{
      B \ar@{ >->}[r]^-{\mat{1 \\ -r}} &
      B \oplus C \ar@{->>}[r]^-{\mat{r & 1}} & C
    }
  } 
  \qquad \text{and} \qquad
  \vcenter{
    \xymatrix{
      C \ar@{ >->}[r]^-{\mat{-s \\ 1}} &
      B \oplus C \ar@{->>}[r]^-{\mat{1 & s}} & B
    }
  } 
  \]
  are split exact, so $\mat{r & 1}$ and $\mat{1 & s}$ are admissible 
  epics. But the diagram
  \[
  \xymatrix{
    B \oplus C \ar@{->>}[r]^-{\mat{1 & s}} 
    \ar@{->>}[dr]_{\mat{r & 1}} & 
    B \ar[d]^{r} \\
    & C
  }
  \]
  is commutative, hence $r$ is an admissible epic.
\end{proof}

\begin{Rem}[{\cite[1.12]{MR1080854}}]
  Every small additive category $\scrA$ has a
  \emph{weak idempotent completion}
  $\scrA'$. Objects of $\scrA'$ are the pairs $(A,p)$, 
  where $p: A \to A$ is an idempotent factoring as $p = cr$ for some
  retraction $r: A \to X$ and coretraction $c: X \to A$ with
  $rc = 1_{B}$, while the morphisms are given by
  \[
  \Hom_{\scrA'}{((A,p), (B,q))} = q \circ \Hom_{\scrA} {(A, B)}
  \circ p.
  \]
  It is
  easy to see that the functor $\scrA \to \scrA'$ given on objects by
  $A \mapsto (A, 1_{A})$ is $2$-universal among functors from $\scrA$ to a
  weakly idempotent complete category. Moreover, if $(\scrA,\scrE)$ 
  is exact then so is $(\scrA', \scrE')$, where the sequences in
  $\scrE'$ are the direct summands in $\scrA'$ of sequences in
  $\scrE$, and the functor $\scrA \to \scrA'$ preserves and reflects
  exactness and is $2$-universal among exact functors to weakly
  idempotent complete categories.
\end{Rem}

\begin{Rem}
  Contrary to the construction of the idempotent completion, there is
  the set-theoretic subtlety that the weak idempotent completion might 
  not be well-defined if $\scrA$ is not small:
  it is \emph{not} clear a priori that the
  objects $(A,p)$ form a class---essentially for the same
  reason that the monics in a category need not form a class, see
  e.g. the discussion in Borceux~\cite[p.~373f]{MR1313497}.
\end{Rem}

\section{Admissible Morphisms and the Snake Lemma}
\label{sec:adm-morph-snake-lemma}


\begin{Def}
  \label{def:admissible-morphism}
  A morphism $f:A \to B$ in an exact category 
  \[
  \xymatrix@R=0.5pc{
    A \ar[rr]^{f}|{\object@{o}} \ar@{->>}[dr]_{e} & & B \\
    & I \ar@{ >->}[ur]_{m}
  }
  \]
  is called \emph{admissible} if it factors
  as a composition of an admissible monic with an admissible epic.
  Admissible morphisms will sometimes be displayed as 
  $\xymatrix{ \ar[r]|{\object@{o}} &  }$
  in diagrams.
\end{Def}

\begin{Rem}
  \label{rem:admissible-morphisms-dont-compse}
  Let $f$ be an admissible morphism. If $e'$ is an admissible epic and
  $m'$ is an admissible monic then $m'fe'$ is admissible if the
  composition is defined. However, admissible morphisms are \emph{not} closed
  under composition in general. Indeed, consider a morphism $g: A \to B$ 
  which is not admissible. The morphisms 
  $\mat{1 \\ g}: A \to A \oplus B$ and 
  $\mat{0 & 1}: A \oplus B \to B$ are admissible since they are
  part of split exact sequences. But $g = \mat{0 & 1} \mat{1 \\ g}$ is
  not admissible by hypothesis.
\end{Rem}

\begin{Rem}
  \label{rem:adm-mor=strict-mor}
  We choose the terminology \emph{admissible morphism} even though
  \emph{strict morphism} seems to be more standard
  (see e.g.~\cite{MR1856638, MR1779315}). By
  Exercise~\ref{exer:adm-epic+monic=iso} an admissible monic is the
  same thing as an admissible morphism which happens to be monic.
\end{Rem}

\begin{Lem}[{\cite[3.4]{MR0100622}}]
  \label{lem:adm-mor-factors-uniquely}
  The factorization of an admissible
  morphism is unique up to unique isomorphism. More precisely: In a
  commutative diagram of the form
  \[
  \xymatrix{
    A \ar@{->>}[r]^{e} \ar@{->>}[d]_{e'} & 
    I \ar@{ >->}[d]^{m} \ar@<2pt>@{.>}[dl]^{i} \\
    I' \ar@{ >->}_{m'}[r] \ar@<2pt>@{.>}[ur]^{i'} & B
  }
  \]
  there exist unique morphisms $i$, $i'$ making the diagram
  commutative. In particular, $i$ and $i'$ are mutually inverse 
  isomorphisms.
\end{Lem}
\begin{proof}
  Let $k$ be a kernel of  $e$. Since $m'e'k = mek = 0$
  and $m'$ is monic we have $e'k = 0$, hence there exists a unique
  morphism $i: I \to I'$ such that $e' = ie$. Moreover, $m'ie =
  m'e' = me $ and $e$ epic imply $m'i = m$. Dually for $i'$.
\end{proof}

\begin{Rem}
  \label{rem:adm-mor-analysis}
  An admissible morphism has an \emph{analysis} 
  (cf. \cite[IX.2]{MR0156879})
  \[
  \xymatrix@R=0.5pc{
    & A \ar[rr]^{f}|{\object@{o}} \ar@{->>}[dr]_{e} 
    & & B \ar@{->>}[dr]^{c} \\
    K \ar@{ >->}[ur]^{k} & & I \ar@{ >->}[ur]_{m} & & C
  }
  \]
  where $k$ is a kernel, $c$ is a cokernel, $e$ is a
  coimage and $m$ is an image of $f$ and the isomorphism classes of
  $K$, $I$ and $C$ are well-defined by Lemma~\ref{lem:adm-mor-factors-uniquely}.
\end{Rem}

\begin{Exer}
  If $\scrA$ is an exact category in which every morphism is
  admissible then $\scrA$ is abelian.
  [A solution is given by Freyd in~\cite[Proposition~3.1]{MR0209333}].
\end{Exer}

\begin{Lem}
  \label{lem:adm-mor-stable-under-push-pull}
  Admissible morphisms are stable under push-out along admissible monics
  and pull-back along admissible epics.
\end{Lem}
\begin{proof}
  Let $A \epi I \mono B$ be an admissible epic-admissible
  monic factorization of an
  admissible morphism. To prove the claim about push-outs construct
  the diagram
  \[
  \xymatrix{
    A \ar@{->>}[r] \ar@{ >->}[d] \ar@{}[dr]|{\text{PO}} &
    I \ar@{ >->}[r] \ar@{ >->}[d] \ar@{}[dr]|{\text{PO}} &
    B \ar@{ >->}[d] \\
    A' \ar@{->>}[r] & I' \ar@{ >->}[r] & B'.
  }
  \]
  Proposition~\ref{prop:pb-adm-monic-adm-monic} yields that $A'
  \to I'$ is an admissible epic and the rest is clear.
\end{proof}

\begin{Def}
  \label{def:ex-seq-adm-mor}
  A sequence of admissible morphisms 
  \[  
  \xymatrix@R=0.5pc{
    A' \ar[rr]^{f}|{\object@{o}} \ar@{->>}[dr]_{e} & & 
    A \ar[rr]^{f'}|{\object@{o}} \ar@{->>}[dr]_{e'} & & A''
    \\
    & I \ar@{ >->}[ur]_{m} & &  I' \ar@{ >->}[ur]_{m'}
  }
  \]
  is \emph{exact} if $I \mono A \epi I'$ is short exact.
  Longer sequences of
  admissible morphisms are exact if the sequence given by any
  two consecutive morphisms is exact. Since the term ``exact'' is
  heavily overloaded, we also use the synonym \emph{``acyclic''}, 
  in particular in connection with chain complexes.
\end{Def}

\begin{Lem}[Five Lemma, II]
  \label{lem:long-five-lemma}
  If the commutative diagram 
  \[
  \xymatrix{
    A_{1} \ar[d]^{\cong} \ar[r]|-{\object@{o}} &
    A_{2} \ar[d]^{\cong} \ar[r]|{\object@{o}} &
    A_{3} \ar[d]^{f} \ar[r]|{\object@{o}} &
    A_{4} \ar[d]^{\cong} \ar[r]|{\object@{o}} &
    A_{5} \ar[d]^{\cong} \\
    B_{1} \ar[r]|{\object@{o}} &
    B_{2} \ar[r]|{\object@{o}} &
    B_{3} \ar[r]|{\object@{o}} &
    B_{4} \ar[r]|{\object@{o}} &
    B_{5}
  }
  \]
  has exact rows then $f$ is an isomorphism.
\end{Lem}

\begin{proof}[Sketch of the Proof]
  By hypothesis we may choose factorizations 
  $A_{i} \epi I_{i} \mono A_{i+1}$ of $A_{i} \to
  A_{i+1}$ and $B_{i} \epi J_{i} \mono B_{i+1}$ of 
  $B_{i} \to B_{i+1}$ for $i = 1,\ldots,4$. Using
  Lemma~\ref{lem:adm-mor-factors-uniquely} and
  Exercise~\ref{exer:two-out-of-three-five-lemma} there are
  isomorphisms $I_{1} \cong J_{1}$ and $I_{2} \cong J_{2}$ which
  one may insert into the diagram without destroying 
  its commutativity. Dually for $I_{4} \cong J_{4}$ and $I_{3} \cong
  J_{3}$. The five lemma~\ref{cor:five-lemma} 
  then implies that $f$ is an isomorphism.
\end{proof}

\begin{Exer}
  \label{exer:sharp-four-five-lemma}
  Assume that $\scrA$ is weakly idempotent complete
  (Definition~\ref{def:weakly-idempotent-complete}). 
  \begin{enumerate}[(i)]
    \item (Sharp Four Lemma)
      Consider a
      commutative diagram
      \[
      \xymatrix{
        A_{1} \ar@{->>}[d] \ar[r]|{\object@{o}} &
        A_{2} \ar[d]^{\cong} \ar[r]|{\object@{o}} &
        A_{3} \ar[d]^{f} \ar[r]|{\object@{o}} &
        A_{4} \ar@{ >->}[d] \\
        B_{1} \ar[r]|{\object@{o}} &
        B_{2} \ar[r]|{\object@{o}} &
        B_{3} \ar[r]|{\object@{o}} &
        B_{4}
      }
      \]
      with exact rows. Prove that $f$ is an admissible monic. Dualize.
      
    \item
      (Sharp Five Lemma)
      If the commutative diagram 
      \[
      \xymatrix{
        A_{1} \ar@{->>}[d] \ar[r]|{\object@{o}} &
        A_{2} \ar[d]^{\cong} \ar[r]|{\object@{o}} &
        A_{3} \ar[d]^{f} \ar[r]|{\object@{o}} &
        A_{4} \ar[d]^{\cong} \ar[r]|{\object@{o}} &
        A_{5} \ar@{ >->}[d] \\
        B_{1} \ar[r]|{\object@{o}} &
        B_{2} \ar[r]|{\object@{o}} &
        B_{3} \ar[r]|{\object@{o}} &
        B_{4} \ar[r]|{\object@{o}} &
        B_{5}
      }
      \]
      has exact rows then $f$ is an isomorphism.
  \end{enumerate}
  \emph{Hint:}
  Use Proposition~\ref{prop:weakly-split-obscure-axiom},
  Exercise~\ref{exer:adm-epic+monic=iso},
  Exercise~\ref{exer:two-out-of-three-five-lemma} as well as
  Corollary~\ref{cor:five-lemma}.
\end{Exer}

\begin{Prop}[$\Ker$--$\Coker$--Sequence]
  \label{prop:ker-coker-sequence}
  Assume that $\scrA$ is a weakly idempotent complete exact category.
  Let $f: A \to B$ and $g: B \to C$ be admissible morphisms such that
  $h = gf : A \to C$ is admissible as well. There is an exact sequence
  \[
  \xymatrix{
    \Ker{f} \ar@{ >->}[r] &
    \Ker{h} \ar[r]|-{\object@{o}} &
    \Ker{g} \ar[r]|-{\object@{o}} &
    \Coker{f} \ar[r]|-{\object@{o}} &
    \Coker{h} \ar@{->>}[r] &
    \Coker{g} 
  }
  \]
  depending naturally on the diagram $h = gf$.
\end{Prop}
\begin{proof} 
  Observe that the morphism 
  $\Ker{f} \mono A$ factors over $\Ker{h} \mono A$ via 
  a unique morphism $\Ker{f} \to \Ker{h}$ which is an admissible monic
  by Proposition~\ref{prop:weakly-split-obscure-axiom}.
  Let $\Ker{h} \epi X$ be a
  cokernel of $\Ker{f} \mono \Ker{h}$. Dually, there is an admissible
  epic $\Coker{h} \epi \Coker{g}$ with $Z \mono \Coker{h}$ as
  kernel.
  The Noether isomorphism~\ref{lem:c/b=(c/a)/(b/a)} implies that
  there are two commutative diagrams with exact rows and columns
  \[
  \vcenter{
    \xymatrix{
      \Ker{f} \ar@{=}[d] \ar@{ >->}[r] &
      \Ker{h} \ar@{ >->}[d] \ar@{->>}[r] & 
      X \ar@{ >->}[d] \\
      \Ker{f} \ar@{ >->}[r] &
      A \ar@{->>}[r] \ar@{->>}[d] &
      \Im{f} \ar@{->>}[d] \\
      & \Im{h} \ar@{=}[r] & \Im{h}
    }
  }
  \qquad \text{and} \qquad
  \vcenter{
    \xymatrix{
      \Im{h} \ar@{ >->}[r] \ar@{=}[d] &
      \Im{g} \ar@{->>}[r] \ar@{ >->}[d] &
      Z \ar@{ >->}[d] \\
      \Im{h} \ar@{ >->}[r] &
      C \ar@{->>}[r] \ar@{->>}[d] &
      \Coker{h} \ar@{->>}[d] \\
      & \Coker{g} \ar@{=}[r] & \Coker{g}.
    }
  }
  \]
  It is easy to see that there is an admissible monic $X \mono
  \Ker{g}$ whose cokernel we denote by $\Ker{g} \epi Y$. Therefore the
  $3 \times 3$-lemma yields a commutative diagram with exact rows and columns
  \[
  \vcenter{
    \xymatrix{
      X \ar@{ >->}[d] \ar@{ >->}[r] &
      \Ker{g} \ar@{ >->}[d] \ar@{->>}[r] &
      Y \ar@{ >->}[d] \\
      \Im{f} \ar@{->>}[d] \ar@{ >->}[r] & 
      B \ar@{->>}[d] \ar@{->>}[r] &
      \Coker{f} \ar@{->>}[d] \\
      \Im{h} \ar@{ >->}[r] &
      \Im{g} \ar@{->>}[r] & Z.
    }
  }
  \]
  The desired $\Ker$--$\Coker$--sequence is now obtained by splicing:
  Start with the first row of the first diagram, splice it
  with the first row of the third diagram, and continue with the third row
  of the third diagram and the third row of the second diagram.
  The naturality assertion is obvious.
\end{proof}

\begin{Lem}
  \label{lem:ker-coker-wic-nec}
  Let $\scrA$ be an exact category in which each commutative
  triangle of admissible morphisms yields an exact 
  $\Ker$--$\Coker$--sequence where exactness is understood in the
  sense of admissible morphisms. 
  Then $\scrA$ is weakly idempotent complete.
\end{Lem}
\begin{proof}
  We check the criterion in Corollary~\ref{cor:cancellation}.
  We need to show that in every commutative diagram of the form
  \[
  \xymatrix{
    A \ar@{->>}[r]^{f} \ar@{->>}[dr]_{h} & B \ar[d]^{g} \\
    & C
  }
  \]
  the morphism $g$ is an admissible epic. Given such a diagram,
  consider the diagram
  \[
  \xymatrix{
    \Ker{f} \ar@{ >->}[r] \ar[dr]|-{\object@{o}}_{0} & 
    A \ar@{->>}[d]^{h} \\
    & C
  }
  \]
  whose associated $\Ker$--$\Coker$--sequence is 
  \[
  \xymatrix{
    0 \ar[r]|-{\object@{o}} & 
    \Ker{f} \ar[r]|-{\object@{o}} & 
    \Ker{h} \ar[r]|-{\object@{o}} &
    B \ar[r]|-{\object@{o}}^-{g} & 
    C \ar[r]|-{\object@{o}} &
    0
  }
  \]
  so that $g$ is an admissible epic.
\end{proof}

The $\Ker$--$\Coker$--sequence immediately yields the following
version of the snake lemma, the neat proof of which was pointed
out to the author by Matthias K\"unzer.

\begin{Cor}[Snake Lemma, I]
  \label{cor:snake-lemma}
  Let $\scrA$ be weakly idempotent complete.
  Consider a morphism of short exact sequences $A' \mono A
  \epi A''$ and $B' \mono B \epi B''$ with admissible components. 
  There is a commutative diagram
  \[
  \xymatrix{
    K'  \ar@{ >->}[d] \ar@{.>}[r]^{k} &
    K   \ar@{ >->}[d] \ar@{.>}[r]^{k'} &
    K'' \ar@{ >->}[d] \\
    A'  \ar@{ >->}[r] \ar[d]|{\object@{o}} &
    A   \ar@{->>}[r] \ar[d]|{\object@{o}} &
    A'' \ar[d]|{\object@{o}} \\
    B'  \ar@{ >->}[r] \ar@{->>}[d] &
    B   \ar@{->>}[r] \ar@{->>}[d] &
    B'' \ar@{->>}[d] \\
    C' \ar@{.>}[r]^{c} & C \ar@{.>}[r]^{c'}  & C''
  }
  \]
  with exact rows and columns and there is a connecting morphism
  $\delta:K'' \to C'$ fitting into an exact sequence
  \[
  \xymatrix{
    K' \ar@{ >->}[r]^{k} &
    K  \ar[r]|{\object@{o}}^{k'} &
    K'' \ar[r]^{\delta}|{\object@{o}} &
    C' \ar[r]|{\object@{o}}^{c} &
    C  \ar@{->>}[r]^{c'} &
    C''
  }
  \]
  depending naturally on the morphism of short exact sequences.
\end{Cor}

\begin{Rem}
  Using the notations of the proof of
  Lemma~\ref{lem:ker-coker-wic-nec}
  consider the diagram
  \[
  \xymatrix{
    \Ker{f} \ar[d]^{0}|-{\object@{o}} \ar@{ >->}[r] &
    A \ar[d]^{h}|-{\object@{o}}  \ar@{->>}[r] &
    B \ar[d]|-{\object@{o}} \\
    C \ar@{=}[r] & C \ar[r] & 0.
  }
  \]
  The sequence
  $\xymatrix{
    \Ker{f} \ar@{ >->}[r] & 
    \Ker{h} \ar[r]|-{\object@{o}} &
    B \ar[r]|-{\object@{o}}^-{g} & 
    C \ar[r]|-{\object@{o}} &
    0 \ar@{->>}[r] & 0
  }$
  provided by the snake lemma shows that $g$ is an admissible epic. It
  follows that the snake lemma can only hold if the category is weakly
  idempotent complete.
\end{Rem}

\begin{proof}[Proof of Corollary~\ref{cor:snake-lemma}]
  By Proposition~\ref{prop:decompose-morphisms-of-exact-sequences}
  and Lemma~\ref{lem:adm-mor-stable-under-push-pull} we get the
  commutative diagram
  \[
  \xymatrix{
    A'  \ar[d]|{\object@{o}} \ar@{ >->}[r] \ar@{}[dr]|{\text{BC}} &
    A   \ar[d]|{\object@{o}} \ar@{->>}[r] &
    A'' \ar@{=}[d] \\
    B'  \ar@{=}[d] \ar@{ >->}[r] &
    D   \ar[d]|{\object@{o}} \ar@{->>}[r] \ar@{}[dr]|{\text{BC}} &
    A'' \ar[d]|{\object@{o}} \\
    B' \ar@{ >->}[r] & B \ar@{->>}[r] & B''
  }
  \]
  [more explicitly, the diagram is obtained by forming the push-out 
  $A'AB'D$].
  The $\Ker$--$\Coker$--sequence of the commutative
  triangle of admissible morphisms 
  \[
  \xymatrix{
    A \ar[r]|{\object@{o}}  \ar[dr]|{\object@{o}} & 
    D \ar[d]|{\object@{o}} \\
    & B
  }
  \]
  yields the desired result by Remark~\ref{rem:cokernels-in-push-outs}.
\end{proof}

\begin{Exer}[{Snake Lemma, II, cf. \cite[4.3]{MR0100622}}]
  Formulate and prove a snake lemma for a diagram of the form
  \[
  \xymatrix{
    &
    A'  \ar[r]|{\object@{o}} \ar[d]|{\object@{o}} &
    A   \ar[r]|{\object@{o}} \ar[d]|{\object@{o}} &
    A'' \ar[r] \ar[d]|{\object@{o}} &
    0 & \\
    0  \ar[r] &
    B'  \ar[r]|{\object@{o}}  &
    B   \ar[r]|{\object@{o}}  &
    B''
  }
  \]
  in weakly idempotent complete categories. Prove 
  $\Ker{(A' \to A)} = \Ker{(K' \to K)}$ and
  $\Coker{(B \to B'')} = \Coker{(C \to C'')}$.

  \emph{Hint:}
  Reduce to Corollary~\ref{cor:snake-lemma} by using 
  Proposition~\ref{prop:weakly-split-obscure-axiom} and
  the Noether isomorphism~\ref{lem:c/b=(c/a)/(b/a)}.
\end{Exer}

\begin{Rem}
  Heller \cite[4.3]{MR0100622} gives a direct proof of the snake
  lemma starting from his
  axioms. Using the Noether isomorphism~\ref{lem:c/b=(c/a)/(b/a)} and
  the $3 \times 3$-lemma~\ref{cor:3x3-lemma} as well as
  Proposition~\ref{prop:weakly-split-obscure-axiom}, 
  Heller's proof is easily
  adapted to a proof from Quillen's axioms.
\end{Rem}

\section{Chain Complexes and Chain Homotopy}
\label{sec:ch-cxes-ch-htpy}

The notion of chain complexes makes sense in every additive category 
$\scrA$. A \emph{(chain) complex} is a diagram
$(A^{\bullet}, d_{A}^{\bullet})$ 
\[
\cdots \xrightarrow{} A^{n-1} \xrightarrow{d_{A}^{n-1}} A^{n}
\xrightarrow{d_{A}^{n}} A^{n+1} \xrightarrow{} \cdots
\]
subject to the condition that $d^{n} d^{n-1} = 0$ for all $n$ and a
\emph{chain map} is a morphism of such diagrams. The category of
complexes and chain maps is denoted by $\Ch{(\scrA)}$. 
Obviously, the category $\Ch{(\scrA)}$ is additive.

\begin{Lem}
  \label{lem:cxes-exact}
  If $(\scrA,\scrE)$ is an exact category then $\Ch{(\scrA)}$ is an
  exact category with respect to the class $\Ch{(\scrE)}$
  of short sequences of chain maps which are exact in each degree. 
  If $\scrA$ is abelian then so is $\Ch{(\scrA)}$.
\end{Lem}

\begin{proof}
  The point is that (as in every functor category)
  limits and colimits of diagrams in $\Ch{(\scrA)}$
  are obtained by taking the limits and colimits pointwise 
  (in each degree), in
  particular push-outs under admissible monics and pull-backs over
  admissible epics exist and yield admissible monics and epics. The
  rest is obvious.
\end{proof}

\begin{Def}
  The \emph{mapping cone} of a chain map $f: A \to B$ is
  the complex
  \[
  \cone{(f)}^{n} = A^{n+1} \oplus B^{n}
  \qquad \text{with differential} \qquad 
  d_{f}^{n} = \mat{-d_{A}^{n+1} & 0 \\ f^{n+1} & d_{B}^{n}}.
  \]
  Notice that $d_{f}^{n+1} d_{f}^{n} = 0$ precisely because $f$ is a chain
  map. It is plain that the mapping cone 
  defines a functor from the category of
  morphisms in $\Ch{(\scrA)}$ to $\Ch{(\scrA)}$.

  The \emph{translation functor} on $\Ch{(\scrA)}$ is defined to be
  $\Sigma A = \cone{(A \to 0)}$. 
  More explicitly, $\Sigma A$ is the complex with components
  $(\Sigma A)^{n} = A^{n+1}$ and differentials
  $d_{\Sigma A}^{n} = - d_{A}^{n+1}$. If $f$ is a chain map, its
  translate is given by $(\Sigma f)^{n} = f^{n+1}$. Clearly, $\Sigma$
  is an additive automorphism of $\Ch{(\scrA)}$.

  The \emph{strict triangle} over the chain map $f: A \to B$ is the
  $3$-periodic (or rather $3$-helicoidal, if you insist) sequence
  \[
  A \xrightarrow{f} B \xrightarrow{i_{f}} \cone{(f)} \xrightarrow{j_{f}}
  \Sigma A \xrightarrow{\Sigma f} \Sigma B \xrightarrow{\Sigma i_{f}}
  \Sigma \cone{(f)} \xrightarrow{\Sigma j_{f}} \cdots,
  \]
  where the chain map $i_{f}$ has components $\mat{0 \\ 1}$ and $j_{f}$ has
  components $\mat{1 & 0}$.
\end{Def}

\begin{Rem}
  \label{rem:strict-triangle}
  Let $f: A \to B$ be a chain map.
  Observe that the sequence of chain maps
  \[
  B \xrightarrow{i_{f}} \cone{(f)} \xrightarrow{j_{f}} \Sigma A
  \]
  splits in each degree, however it need not be a split exact sequence
  in $\Ch{(\scrA)}$, because the degreewise 
  splitting maps need not assemble to chain maps. In fact, it is 
  straightforward to verify that the above sequence is split exact in
  $\Ch{(\scrA)}$ if and only if $f$ is chain homotopic to zero in the
  sense of Definition~\ref{def:chain-homotopy}.
\end{Rem}

\begin{Exer}
  \label{exer:str-tr-long-ex-seq}
  Assume that $\scrA$ is an abelian category. Prove that the strict
  triangle over the chain map $f:A \to B$ 
  gives rise to a long exact homology sequence
  \[
  \cdots  \xrightarrow{} 
  H^{n}(A) \xrightarrow{H^{n}(f)} 
  H^{n}(B) \xrightarrow{H^{n}(i_{f})}
  H^{n}{(\cone{(f)})} \xrightarrow{H^{n}{(j_{f})}} H^{n+1}{(A)}
  \xrightarrow{}\cdots.
  \]
  Deduce that $f$ induces an isomorphism of $H^{\ast}(A)$ with
  $H^{\ast}(B)$ if and only if $\cone{(f)}$ is acyclic.
\end{Exer}

\begin{Def}
  \label{def:chain-homotopy}
  A chain map $f:A \to B$ is \emph{chain homotopic to zero} if there exist
  morphisms $h^{n}:A^{n} \to B^{n-1}$ such that 
  $f^{n} = d_{B}^{n-1}h^{n} + h^{n+1}d_{A}^{n}$. A chain complex $A$ is
  called \emph{null-homotopic} if $1_{A}$ is chain homotopic to zero.
\end{Def}

\begin{Rem}
The maps which are chain homotopic to zero form an ideal in
$\Ch{(\scrA)}$, that is to say if $h: B \to C$ is chain homotopic to zero then
so are $hf$ and $gh$ for all morphisms $f: A \to B$ and $g: C \to D$,
if $h_{1}$ and $h_{2}$ are chain homotopic to zero then so is 
$h_{1} \oplus h_{2}$. The set $N{(A,B)}$ 
of chain maps $A \to B$ which are chain homotopic to zero is a
subgroup of the abelian group $\Hom_{\Ch{(\scrA)}}{(A,B)}$. 
\end{Rem}

\begin{Def}
  The \emph{homotopy category} $\K{(\scrA)}$ is the category with the
  chain complexes over $\scrA$ as  objects and 
  $\Hom_{\K{(\scrA)}}{(A,B)} := \Hom_{\Ch{(\scrA)}}{(A,B)} /
  N{(A,B)}$ as morphisms.
\end{Def}

\begin{Rem}
  \label{rem:htpy-cat-triangulated}
  Notice that every null-homotopic complex is isomorphic 
  to the zero object in $\K{(\scrA)}$.
  It turns out that $\K{(\scrA)}$ is additive, but it is very rarely
  abelian or exact with respect to a
  non-trivial exact structure (see Verdier~\cite[Ch.II,
  1.3.6]{MR1453167}). However, $\K{(\scrA)}$
  has the structure of a \emph{triangulated category} induced by the
  \emph{strict triangles} in $\Ch{(\scrA)}$, see
  e.g. Verdier~\cite{MR1453167},
  Be\u\i{}linson-Bernstein-Deligne~\cite{MR751966}, 
  Gelfand-Manin~\cite{MR1950475}, 
  Grivel~\cite[Chapter~I]{MR882000},
  Kashiwara-Schapira~\cite{MR2182076}, 
  Keller~\cite{MR1421815},
  Neeman~\cite{MR1812507} or Weibel~\cite{MR1269324}.
\end{Rem}

\begin{Rem}
  For each object $A \in \scrA$, define $\cone{(A)} =
  \cone{(1_{A})}$. Notice that $\cone{(A)}$ is null-homotopic with
  $\mat{0 & 1 \\ 0 & 0}$ as contracting homotopy.
\end{Rem}

\begin{Rem}
  If $f$ and $g$ are chain homotopy equivalent, i.e., $f-g$ is chain
  homotopic to zero, then $\cone{(f)}$ and $\cone{(g)}$ are 
  isomorphic in $\Ch{(\scrA)}$ but the isomorphism and its homotopy
  class will generally depend on the
  choice of a chain homotopy. In particular, the mapping
  cone construction does not yield a functor defined on morphisms of
  $\K{(\scrA)}$. 
\end{Rem}

\begin{Rem}
  \label{rem:cone-factors}
  A chain map $f: A \to B$ is chain homotopic to zero if and only if
  it factors  as $h i_{A} = f$ over $h: \cone{(A)} \to
  B$, where $i_{A} = i_{1_{A}}: A \to
  \cone{(A)}$. Moreover, $h$ has components $\mat{h^{n+1} & f^{n}}$,
  where the family of morphisms $\{h^{n}\}$ is a chain homotopy of $f$
  to zero. Similarly, $f$ is chain
  homotopic to zero if and only if $f$ factors through $j_{\Sigma^{-1}B} =
  j_{1_{\Sigma^{-1}B}}: \cone{(\Sigma^{-1}B)} \to B$.
\end{Rem}

\begin{Rem}
  \label{rem:cone-weak-cokernel}
  The mapping cone construction yields the push-out diagram
  \[
  \xymatrix{
    A \ar[r]^{f} \ar[d]_{i_A} \ar@{}[dr]|{\text{PO}} &
    B \ar[d]^{i_{f}} \\
    \cone{(A)} \ar[r]_{\mat{1 & 0 \\ 0 & f}} &
    \cone{(f)}
  }
  \]
  in $\Ch{(\scrA)}$. Now suppose that $g: B \to C$ is a chain map such
  that $gf$ is chain homotopic to zero. By
  Remark~\ref{rem:cone-factors}, $gf$ factors over $i_{A}$ and using
  the push-out property of the above diagram it follows that $g$
  factors over $i_{f}$. This construction will depend on the choice of
  an explicit chain homotopy $gf \simeq 0$ in general. 
  In particular, $\cone(f)$ is a \emph{weak cokernel} in $\K{(\scrA)}$
  of the homotopy class of $f$ in that it has the
  factorization property of a cokernel but without uniqueness.
  Similarly, $\Sigma^{-1}\cone{(f)}$ is a \emph{weak kernel}
  of $f$ in $\K{(\scrA)}$.
\end{Rem}

\section{Acyclic Complexes and Quasi-Isomorphisms}
\label{sec:ac-cxes-qis}

The present section is probably only of interest to readers acquainted
with triangulated categories or at least with the construction of
the derived category of an abelian category. After giving the
fundamental definition of acyclicity of a complex over an
exact category, we may formulate the intimately connected 
notion of quasi-isomorphisms.

We will give an elementary proof of the fact that the homotopy
category $\Ac{(\scrA)}$ of acyclic complexes over an
exact category $\scrA$ is a triangulated category. It turns out that
$\Ac{(\scrA)}$ is a strictly full subcategory of the homotopy category
of chain complexes $\K{(\scrA)}$ if and
only if $\scrA$ is idempotent complete, and in this case
$\Ac{(\scrA)}$ is even thick in $\K{(\scrA)}$. Since thick
subcategories are strictly full by definition, 
$\Ac{(\scrA)}$ is thick if and only if $\scrA$ is idempotent
complete.

By \cite[Chapter~2]{MR1812507}, the Verdier quotient $\K/\scrT$ is
defined for any (strictly full)
triangulated subcategory $\scrT$ of a triangulated
category $\K$ and it coincides with the Verdier quotient 
$\K / \bar{\scrT}$, where $\bar{\scrT}$ is the \emph{thick closure} of
$\scrT$. The case we are interested in is $\K = \K{(\scrA)}$
and $\scrT = \Ac{(\scrA)}$. The Verdier quotient 
$\D{(\scrA)} = \K{(\scrA)} / \Ac{(\scrA)}$ is the
\emph{derived category} of $\scrA$. If $\scrA$ is idempotent complete then
$\overline{\Ac{(\scrA)}} = \Ac{(\scrA)}$ and it is clear that
quasi-isomorphisms are then precisely the chain maps with acyclic
mapping cone. If $\scrA$ fails to be idempotent complete, it turns out
that the thick closure $\overline{\Ac{(\scrA)}}$ of $\Ac{(\scrA)}$ is
the same as the closure of $\Ac{(\scrA)}$ under isomorphisms in
$\K{(\scrA)}$, so a chain map $f$ is a
quasi-isomorphism if and only if $\cone{(f)}$ is homotopy equivalent
to an acyclic complex.

Similarly, the derived categories of bounded, left bounded or right bounded
complexes are constructed as in the abelian setting. It is useful to
notice that for $\ast \in \{+,-,b\}$ the category
$\Ac^{\ast}{(\scrA)}$ is thick in $\K^{\ast}{(\scrA)}$ if
and only if $\scrA$ is weakly idempotent complete, which leads to an easier
description of quasi-isomorphisms.

If $\scrB$ is a fully exact subcategory of $\scrA$, the inclusion
$\scrB \to \scrA$ yields a canonical functor $\D{(\scrB)} \to
\D{(\scrA)}$ and we state conditions which ensure that this functor is
essentially surjective or fully faithful.

We end the section with a short discussion of Deligne's approach to
total derived functors.

\subsection{The Homotopy Category of Acyclic Complexes}

\begin{Def}
  A chain complex $A$ over an exact category
  is called \emph{acyclic} (or \emph{exact}) if each differential
  factors as $A^{n} \epi Z^{n+1}A \mono  A^{n+1}$ in such a way 
  that each sequence
  $Z^{n}A \mono A^{n} \epi Z^{n+1}A$ is exact.
\end{Def}

\begin{Rem}
  An acyclic complex is a complex with admissible differentials
  (Definition~\ref{def:admissible-morphism}) which is exact in the
  sense of Definition~\ref{def:ex-seq-adm-mor}. In particular,
  $Z^{n}A$ is a kernel of $A^{n} \to A^{n+1}$, an image and coimage
  of $A^{n-1} \to A^{n}$ and a cokernel of $A^{n-2} \to A^{n-1}$. 
\end{Rem}


\begin{Lem}[{Neeman~\cite[1.1]{MR1080854}}]
  \label{lem:cone-of-acyclics-is-acyclic}
  The mapping cone of a chain map $f: A \to B$ between acyclic
  complexes is acyclic.
\end{Lem}

\begin{proof}
  An easy diagram chase shows that the dotted morphisms in the diagram
  \[
  \xymatrix@R=0.8pc{
    A^{n-1} \ar[rr]^-{d_{A}^{n-1}} \ar@{->>}[dr]_-{j_{A}^{n-1}} 
    \ar[ddd]_-{f^{n-1}} & &
    A^{n} \ar[rr]^-{d_{A}^{n}} \ar@{->>}[dr]_-{j_{A}^{n}} \ar[ddd]^-{f^{n}} & & 
    A^{n+1} \ar[ddd]^-{f^{n+1}} \\
    & Z^{n}A \ar@{ >->}[ur]_-{i_{A}^{n}} \ar@{.>}[d]^-{\exists! g^{n}} & 
    & Z^{n+1}A \ar@{ >->}[ur]_-{i_{A}^{n+1}} \ar@{.>}[d]^-{\exists! g^{n+1}} & \\
    & Z^{n}B \ar@{ >->}[dr]^-{i_{B}^{n}} &
    & Z^{n+1}B \ar@{ >->}[dr]^-{i_{B}^{n+1}} & \\
    B^{n-1} \ar[rr]_-{d_{B}^{n-1}} \ar@{->>}[ur]^-{j_{B}^{n-1}} & &
    B^{n} \ar[rr]_-{d_{B}^{n}} \ar@{->>}[ur]^-{j_{B}^{n}}  & &
    B^{n+1} 
  }
  \]
  exist and are the unique morphisms $g^{n}$ making the
  diagram commutative.
  
  By Proposition~\ref{prop:decompose-morphisms-of-exact-sequences}
  we find objects $Z^{n}C$
  fitting into a commutative diagram
  \[
  \xymatrix@R=0.9pc{
    A^{n-1} \ar[rr]^-{d_{A}^{n-1}} \ar@{->>}[dr]_-{j_{A}^{n-1}} 
    \ar[dd]_{f'^{n-1}} & &
    A^{n} \ar[rr]^-{d_{A}^{n}} \ar@{->>}[dr]_-{j_{A}^{n}} 
    \ar[dd]^-{f'^{n}} & & 
    A^{n+1} \ar[dd]^-{f'^{n+1}} \\
    & Z^{n}A \ar@{ >->}[ur]_-{i_{A}^{n}} \ar[dd]^-{g^{n}} 
    \ar@{}[dr]|{\text{BC}} & 
    & Z^{n+1}A \ar@{ >->}[ur]_-{i_{A}^{n+1}} \ar[dd]^-{g^{n+1}}
    \ar@{}[dr]|{\text{BC}} & \\
    Z^{n-1}C \ar@{->>}[ur]_-{h^{n-1}} \ar[dd]_-{f''^{n-1}}
    \ar@{}[dr]|{\text{BC}}  & & 
    Z^{n}C \ar@{->>}[ur]_-{h^{n}} \ar[dd]^-{f''^{n}} 
    \ar@{}[dr]|{\text{BC}}  & & 
    Z^{n+1}C \ar[dd]^-{f''^{n+1}}  \\
    & Z^{n}B \ar@{ >->}[ur]^-{k^{n}} \ar@{ >->}[dr]^-{i_{B}^{n}} &
    & Z^{n+1}B \ar@{ >->}[ur]^-{k^{n+1}} \ar@{ >->}[dr]^-{i_{B}^{n+1}} & \\
    B^{n-1} \ar[rr]_-{d_{B}^{n-1}} \ar@{->>}[ur]^-{j_{B}^{n-1}} & &
    B^{n} \ar[rr]_-{d_{B}^{n}} \ar@{->>}[ur]^{j_{B}^{n}}  & &
    B^{n+1} 
  }
  \]
  where $f^{n} = f''^{n}f'^{n}$ and
  the quadrilaterals marked $\text{BC}$ are bicartesian. Recall 
  that the objects $Z^{n}C$ are obtained by forming the push-outs under
  $i_{A}^{n}$ and $g^{n}$ (or the pull-backs over $j_{B}^{n}$ and
  $g^{n+1}$) and that $Z^{n}B \mono Z^{n}C \epi Z^{n+1}A$ is short exact.


  It follows from Corollary~\ref{cor:gluing-pb-po} that for each $n$
  the sequence
  \[
  \xymatrix{
    Z^{n-1}C \ar@{ >->}[rr]^-{\mat{-i_{A}^{n}h^{n-1} \\ f''^{n-1}}}
    & & 
    A^{n} \oplus B^{n-1} \ar@{->>}[rr]^-{\mat{f'^{n} && k^{n}j_{B}^{n-1}}}
    & & 
    Z^{n}C
  }
  \]
  is short exact and the commutative diagram
  \[
  \xymatrix{
    A^{n} \oplus B^{n-1} 
    \ar[rr]^-{\mat{-d_{A}^{n} & 0 \\ f^{n} & d_{B}^{n-1}}} 
    \ar@{->>}[dr]|-{\mat{f'^{n} & k^{n}j_{B}^{n-1}}}
    & &
    A^{n+1} \oplus B^{n}
    \ar[rr]^-{\mat{-d_{A}^{n+1} & 0 \\ f^{n+1} & d_{B}^{n}}}
    \ar@{->>}[dr]|-{\mat{f'^{n+1} & k^{n+1}j_{B}^{n}}}
    & &
    A^{n+2} \oplus B^{n+1} \\
    &
    Z^{n}C \ar@{ >->}[ur]_-{\;\;\mat{-i_{A}^{n+1}h^{n} \\ f''^{n}}} & &
    Z^{n+1}C \ar@{ >->}[ur]_-{\;\;\;\mat{-i_{A}^{n+2}h^{n+1} \\ f''^{n+1}}} 
  }
  \]
  proves that $\cone{(f)}$ is acyclic.
\end{proof}

\begin{Rem}
  Retaining the notations of the proof we have a short exact sequence 
  \[
  Z^{n}B \mono Z^{n}C \epi Z^{n+1} A.
  \]
  This sequence exhibits $Z^{n}C = 
  \Ker{\mat{-d_{A}^{n+1} & 0 \\ f^{n+1} & d_{B}^{n}}}$ as an extension of
  $Z^{n+1}A = \Ker{d_{A}^{n+1}}$
  by $Z^{n}B = \Ker{d_{B}^{n}}$.
\end{Rem}

Let $\Ac{(\scrA)}$ be the full subcategory of the homotopy
category $\K{(\scrA)}$ consisting of acyclic complexes over the exact
category $\scrA$. It follows from
Proposition~\ref{prop:sum-exact} that the direct sum of two acyclic
complexes is acyclic. Thus $\Ac{(\scrA)}$ is a full additive
subcategory of $\K{(\scrA)}$. The previous lemma implies that even
more is true:

\begin{Cor}
  The homotopy category of acyclic complexes $\Ac{(\scrA)}$
  is a triangulated subcategory of $\K{(\scrA)}$. \qed
\end{Cor}

\begin{Rem}
  For reasons of convenience, many authors assume that triangulated
  subcategories are not only full but \emph{strictly full}. We do not
  do so because $\Ac{(\scrA)}$ is closed under isomorphisms in
  $\K{(\scrA)}$ if and only if $\scrA$ is idempotent complete, see
  Proposition~\ref{prop:ac-cxes-strictly-full-iff-karoubian}.
\end{Rem}

\begin{Lem}
  \label{lem:acyclic-complexes-thick}
  Assume that $(\scrA,\scrE)$ is idempotent complete. Every retract
  in $\K{(\scrA)}$ of an acyclic complex $A$ is acyclic.
\end{Lem}
\begin{proof}[Proof (cf. {\cite[2.3 a)]{MR1052551}})]
  Let the chain map $f: X \to A$ be a coretraction, i.e., 
  there is a chain map $s: A \to X$ such that 
  $s^{n}f^{n} - 1_{X^{n}} = d_{X}^{n-1}h^{n} + h^{n+1}d_{X}^{n}$ for
  some morphisms $h^{n}: X^{n} \to X^{n-1}$. Obviously, the complex
  $IX$ with components
  \[
  (IX)^{n} = X^{n} \oplus X^{n+1} \qquad \text{and differential}
  \qquad
  \mat{0 & 1 \\ 0 & 0}
  \]
  is acyclic. There is a chain map $i_{X}: X \to IX$ given by
  \[
  i_{X}^{n} = \mat{1_{X^{n}} \\ d_{X}^{n}}:
  X^{n} \to X^{n} \oplus X^{n+1}
  \]
  and the chain map
  \[
  \mat{f \\ i_{X}} : X \to A \oplus IX
  \]
  has the chain map
  \[
  \mat{s^{n} & -d_{X}^{n-1}h^{n} & -h^{n+1}}: 
  A^{n} \oplus X^{n} \oplus X^{n+1} \to X^{n}
  \]
  as a left inverse. Hence, on replacing the acyclic complex $A$ by the
  acyclic complex $A \oplus IX$, we may assume that $f: X \to A$ has
  $s$ as a left inverse in $\Ch{(\scrA)}$. But then $e = fs: A \to A$ is
  an idempotent in $\Ch{(\scrA)}$ 
  and it induces an idempotent on the exact sequences
  $Z^{n}A \mono A^{n} \epi Z^{n+1}A$ witnessing that $A$ is acyclic
  as in the first diagram of the proof of 
  Lemma~\ref{lem:cone-of-acyclics-is-acyclic}. This means that
  $Z^{n}A \mono A^{n} \epi Z^{n+1}A$ decomposes as a direct sum of
  two short exact sequences (Corollary~\ref{cor:summands-exact}) since
  $\scrA$ is idempotent complete. Therefore the acyclic complex
  $A = X' \oplus Y'$ is a direct sum of the acyclic complexes $X'$ and
  $Y'$, and $f$ induces an isomorphism from $X$ to $X'$ in
  $\Ch{(\scrA)}$. The details are left to the reader.
\end{proof}

\begin{Exer}
  \label{exer:comparing-dist-triangles}
  Prove that the sequence $X \to \cone{(X)} \to \Sigma X$ from
  Remark~\ref{rem:strict-triangle} is
  isomorphic to a sequence $X \to IX \to \Sigma X$ in
  $\Ch{(\scrA)}$. 
\end{Exer}

\begin{Prop}[{\cite[11.2]{MR1421815}}]
  \label{prop:ac-cxes-strictly-full-iff-karoubian}
  The following are equivalent:
  \begin{enumerate}[(i)]
    \item
      Every null-homotopic complex in $\Ch{(\scrA)}$ is acyclic.
    \item
      The category $\scrA$ is idempotent complete.
      
    \item
      The class of acyclic complexes is closed under isomorphisms in
      $\K{(\scrA)}$.
  \end{enumerate}
\end{Prop}
\begin{proof}[Proof (Keller)]
  Let us prove that (i) implies (ii).
  Let $e:A \to A$ be an idempotent of $\scrA$. 
  Consider the complex
  \[
  \cdots \xrightarrow{1-e} A \xrightarrow{e} A \xrightarrow{1-e} A
  \xrightarrow{e} \cdots
  \]
  which is null-homotopic. By~(i) this complex is acyclic.
  This means by definition that $e$ has a kernel and hence $\scrA$ is
  idempotent complete.

  Let us prove that (ii) implies (iii). Assume that $X$ is isomorphic
  in $\K{(\scrA)}$ to an acyclic complex $A$. Using the construction
  in the proof of Lemma~\ref{lem:acyclic-complexes-thick} one
  shows that $X$ is a direct summand in $\Ch{(\scrA)}$ 
  of the acyclic complex $A \oplus IX$ and we conclude by
  Lemma~\ref{lem:acyclic-complexes-thick}.

  That (iii) implies (i) follows from the fact that a null-homotopic
  complex $X$ is isomorphic in $\K{(\scrA)}$ to the (acyclic) 
  zero complex and hence $X$ is acyclic.
\end{proof}

\begin{Rem}
  Recall that a subcategory $\scrT$ of a triangulated category $\bfK$ 
  is called \emph{thick} if it is strictly full and $X \oplus Y \in
  \scrT$ implies $X, Y \in \scrT$.
\end{Rem}

\begin{Cor}
  \label{cor:idemp-compl=thick}
  The triangulated subcategory $\Ac{(\scrA)}$ of $\K{(\scrA)}$ is
  thick if and only if $\scrA$ is idempotent complete. \qed
\end{Cor}

\subsection{Boundedness Conditions}
A complex $A$ is called \emph{left bounded} if $A^{n} = 0$ for $n \ll
0$, \emph{right bounded} if $A^{n} = 0$ for $n \gg 0$ and
\emph{bounded} if $A^{n} = 0$ for $|n| \gg 0$.

\begin{Def}
  Denote by
  $\K^{+}{(\scrA)}$, $\K^{-}{(\scrA)}$ and $\K^{b}{(\scrA)}$ the full
  subcategories of $\K{(\scrA)}$ generated by the left bounded complexes,
  right bounded complexes and bounded
  complexes over $\scrA$.
\end{Def}

  Observe that $\K^{b}{(\scrA)} = \K^{+}{(\scrA)} \cap
  \K^{-}{(\scrA)}$.   
  Note further that $\K^{\ast}{(\scrA)}$ is
  \emph{not} closed under isomorphisms in $\K{(\scrA)}$ for 
  $\ast \in \{+, -, b\}$ unless $\scrA = 0$.

\begin{Def}
  For $\ast \in \{+, -, b\}$ we define
  $\Ac^{\ast}{(\scrA)} = \K^{\ast}{(\scrA)} \cap \Ac{(\scrA)}$.
\end{Def}

Plainly, $\K^{\ast}{(\scrA)}$ is a full triangulated subcategory of
$\K{(\scrA)}$ and $\Ac^{\ast}{(\scrA)}$ is a full triangulated subcategory of
$\K^{\ast}{(\scrA)}$ by Lemma~\ref{lem:cone-of-acyclics-is-acyclic}.

\begin{Prop}
  \label{prop:weakly-idemp-compl-equiv-thick}
  The following assertions are equivalent:
  \begin{enumerate}[(i)]
    \item
      The subcategories $\Ac^{+}{(\scrA)}$ and $\Ac^{-}{(\scrA)}$ of
      $\K^{+}{(\scrA)}$ and $\K^{-}{(\scrA)}$ are thick.

    \item
      The subcategory $\Ac^{b}{(\scrA)}$ of $\K^{b}{(\scrA)}$ is
      thick.
 
    \item
      The category $\scrA$ is weakly idempotent complete.
  \end{enumerate}
\end{Prop}

\begin{proof}
  Since $\Ac^{b}{(\scrA)} = \Ac^{+}{(\scrA)} \cap \Ac^{-}{(\scrA)}$,
  we see that~(i) implies~(ii). Let us prove that~(ii)
  implies~(iii). 
  Let $s: B \to A$ and $t:A \to B$ be morphisms of $\scrA$ such that
  $ts = 1_{B}$. We need to prove that $s$ has a
  cokernel and $t$ has a kernel. The complex $X$ given by
  \[
  \cdots \xrightarrow{} 0 \xrightarrow{} B \xrightarrow{s} A
  \xrightarrow{1-st} A \xrightarrow{t} B \xrightarrow{} 0
  \xrightarrow{} \cdots
  \]
  is a direct summand of
  $X \oplus \Sigma X$ and the latter complex is acyclic since there is
  an isomorphism in $\Ch{(\scrA)}$ 
  \[
  \xymatrix{
    B \ar[d]_{1} \ar[rr]^-{\mat{1 \\ 0}} & &
    B \oplus A 
    \ar[d]_-{\mat{0 & -t \\ s  & 1-st}}
    \ar[rr]^-{\mat{0 & 0 \\ 0 & 1}} & &
    A \oplus A 
    \ar[d]_-{\mat{-1 + st & st \\ st & 1 - st}}
    \ar[rr]^-{\mat{1 & 0 \\ 0 & 0}} & &
    A \oplus B
    \ar[d]_-{\mat{1 - st & -s \\ t & 0}}
    \ar[rr]^-{\mat{0 & 1}} & & B \ar[d]_{1} \\
    B \ar[rr]_-{\mat{0 \\ s}} & &
    B \oplus A \ar[rr]_-{\mat{-s & 0 \\ 0 & 1 - st}} & &
    A \oplus A \ar[rr]_-{\mat{-1 + st & 0 \\ 0 & t}} & &
    A \oplus B \ar[rr]_-{\mat{-t & 0}} & &  B
  }
  \]
  where the upper row is obviously acyclic and the lower row is
  $X \oplus \Sigma X$. Since $\Ac^{b}{(\scrA)}$ is thick, we conclude
  that $X$ is acyclic, so that $s$ has a cokernel and $t$ has a
  kernel. Therefore $\scrA$ is weakly idempotent complete.

  Let us prove that~(iii) implies~(i). Assume that $X$ is a direct
  summand in $\K^{+}{(\scrA)}$ of a complex $A \in
  \Ac^{+}{(\scrA)}$. This means that we are given a
  chain map $f: X \to A$ 
  for which there exists a chain map $s: A \to X$  and morphisms 
  $h^{n}: X^{n} \to X^{n-1}$ 
  such that $s^{n}f^{n} - 1_{X^{n}} = d_{X}^{n-1}h^{n} + h^{n+1}
  d_{X}^{n}$. On replacing $A$ by the acyclic complex $A \oplus IX$ as
  in the proof of
  Lemma~\ref{lem:acyclic-complexes-thick}, we may
  assume that $s$ is a left inverse of $f$ in $\Ch^{+}{(\scrA)}$. In
  particular, since $\scrA$ is assumed to be weakly idempotent
  complete, Proposition~\ref{prop:weakly-split-obscure-axiom} implies
  that each $f^{n}$ is an admissible monic and that each $s^{n}$ is an
  admissible epic.  Moreover, as both complexes $X$ and $A$ are left
  bounded, we may assume that $A^{n} = 0 = X^{n}$ for $n < 0$. It
  follows that $d_{A}^{0}: A^{0} \mono A^{1}$ is an admissible monic
  since $A$ is acyclic. But then $d_{A}^{0} f^{0} = f^{1} d_{X}^{0}$
  is an admissible monic, hence
  Proposition~\ref{prop:weakly-split-obscure-axiom} implies that
  $d_{X}^{0}$ is an admissible monic as well.  Let
  $e^{1}_{X}: X^{1} \epi Z^{2}X$ be a cokernel of $d_{X}^{0}$ 
  and let $e^{1}_{A}: A^{1} \epi Z^{2}A$ be a cokernel of 
  $d_{A}^{0}$. The dotted morphisms in the diagram
  \[
  \xymatrix{
    X^{0} \ar@{ >->}[r]^-{d_{X}^{0}} \ar@{ >->}[d]^-{f^{0}} &
    X^{1} \ar@{->>}[r]^-{e_{X}^{1}} \ar@{ >->}[d]^-{f^{1}} &
    Z^{2}X \ar@{ >.>}[d]^-{g^{2}} \\
    A^{0} \ar@{ >->}[r]^-{d_{A}^{0}} \ar@{->>}[d]^-{s^{0}} &
    A^{1} \ar@{->>}[r]^-{e_{A}^{1}} \ar@{->>}[d]^-{s^{1}} &
    Z^{2}A  \ar@{.>>}[d]^-{t^{2}} \\
    X^{0} \ar@{ >->}[r]^-{d_{X}^{0}} &
    X^{1} \ar@{->>}[r]^-{e_{X}^{1}}  &
    Z^{2}X 
  }
  \]
  are uniquely determined by requiring the resulting diagram to be
  commutative. Since $s^{0}f^{0} = 1_{X^{0}}$ and 
  $s^{1}f^{1} = 1_{X^{1}}$ it follows that $t^{2}g^{2} = 1_{Z^{2}X}$,
  so $t^{2}$ is an admissible epic and $g^{2}$ is an admissible monic by
  Proposition~\ref{prop:weakly-split-obscure-axiom}. 
  
  Now since $A$ and $X$ are complexes, there are unique maps
  $m^{2}_{X}: Z^{2}X \to X^{2}$ and $m^{2}_{A}: Z^{2}A \to A^{2}$
  such that $d^{1}_{X} = m^{2}_{X} e^{1}_{X}$ and
  $d^{1}_{A} = m^{2}_{A} e^{1}_{A}$. Note that $m^{2}_{A}$ is an
  admissible monic since $A$ is acyclic. The upper square in the diagram
  \[
  \xymatrix{
    Z^{2}X \ar[r]^-{m^{2}_{X}} \ar@{ >->}[d]^-{g^{2}} &
    X^{2} \ar@{ >->}[d]^{f^{2}} \\
    Z^{2}A \ar@{ >->}[r]^-{m^{2}_{A}} \ar@{->>}[d]^-{t^{2}} & 
    A^{2} \ar@{->>}[d]^-{s^{2}}  \\
    Z^{2}X \ar[r]^-{m^{2}_{X}} & X^{2}
  }
  \]
  is commutative because $e^{1}_{X}$ is epic and the lower square is
  commutative because $e^{1}_{A}$ is epic. From the commutativity of
  the upper square it follows in particular that $m^{2}_{X}$
  is an admissible monic by
  Proposition~\ref{prop:weakly-split-obscure-axiom}. An easy induction
  now shows that $X$ is acyclic.  The
  assertion about $\Ac^{-}{(\scrA)}$ follows by duality.
\end{proof}

\begin{Rem}
  The isomorphism of complexes in the proof that~(ii) implies~(iii)
  appears in Neeman~\cite[1.9]{MR1080854}.
\end{Rem}

\subsection{Quasi-Isomorphisms}
In abelian categories, quasi-isomorphisms are defined to be chain maps
inducing an isomorphism in homology. Taking the observation in
Exercise~\ref{exer:str-tr-long-ex-seq} and
Proposition~\ref{prop:ac-cxes-strictly-full-iff-karoubian} into
account, one arrives at the following generalization for general exact
categories:

\begin{Def}
  A chain map $f: A \to B$ is called a \emph{quasi-isomorphism} if its
  mapping cone is homotopy equivalent to an acyclic
  complex.
\end{Def}

\begin{Rem}
  Assume that $\scrA$ is idempotent complete. 
  By Proposition~\ref{prop:ac-cxes-strictly-full-iff-karoubian}, a
  chain map $f$ is a quasi-isomorphism if and only if $\cone{(f)}$ is
  acyclic. In particular, for abelian categories, a
  quasi-isomorphism is the same thing as a chain map inducing an
  isomorphism on homology.
\end{Rem}

\begin{Rem}
  If $p: A \to A$ is an idempotent in $\scrA$ which does not
  split, then the complex $C$ given by
  \[
  \cdots \xrightarrow{1-p} A \xrightarrow{p} A \xrightarrow{1-p} A
  \xrightarrow{p} \cdots
  \]
  is null-homotopic but 
  \emph{not} acyclic. However, $f:0 \to C$ is a chain homotopy equivalence,
  hence it should be a quasi-isomorphism, but $\cone{(f)} = C$ fails
  to be acyclic.
\end{Rem}

\subsection{The Definition of the Derived Category}

The \emph{derived category} of the exact category $\scrA$ 
is defined to be the \emph{Verdier quotient}
\[
\D{(\scrA)} = \K{(\scrA)}/ \Ac{(\scrA)}
\]
as described e.g. in Neeman~\cite[Chapter~2]{MR1812507} or 
Keller~\cite[\S\S{} 10, 11]{MR1421815}. For the description of derived
functors given in section~\ref{sec:total-der-fct} it is
useful to recall that the
Verdier quotient can be explicitly described by a calculus of
fractions. A morphism $A \to B$ in $\D{(\scrA)}$ can be represented by a
\emph{fraction} $(f,s)$
\[
A \xrightarrow{f} B' \xleftarrow{s} B
\]
where $f: A \to B$ is a morphism in $\K{(\scrA)}$ and $s: B \to B'$ is
a quasi-isomorphism in $\K{(\scrA)}$. Two fractions $(f,s)$ and
$(g,t)$ are equivalent if there exists a fraction $(h,u)$ and a
commutative diagram
\[
\xymatrix@R=0.6pc{
  & B' \ar[d] \\
  A \ar[ur]^{f} \ar[r]^{h} \ar[dr]_{g} & B''' & 
  B \ar[ul]_{s} \ar[l]_{u} \ar[dl]^{t} \\
  & B'' \ar[u]
}
\]
or, in words, if the fractions $(f,s)$ and $(g,t)$ have a common expansion
$(h,u)$. We refer to Keller~\cite[\S\S{} 9, 10]{MR1421815} for further details.

When dealing with the boundedness condition $\ast \in \{+,-,b\}$ we
define
\[
\D^{\ast}{(\scrA)} = \K^{\ast}{(\scrA)} / \Ac^{\ast}{(\scrA)}.
\]
It is not difficult to prove that
the canonical functor $\D^{\ast}{(\scrA)} \to \D{(\scrA)}$ is an
equivalence between $\D^{\ast}{(\scrA)}$ and the full subcategory of
$\D{(\scrA)}$ generated by the complexes satisfying the boundedness
condition $\ast$, see Keller~\cite[11.7]{MR1421815}.

\begin{Rem}
  If $\scrA$ is idempotent complete then a chain map becomes an isomorphism
  in $\D{(\scrA)}$ if and only if its cone is acyclic by
  Corollary~\ref{cor:idemp-compl=thick}. If $\scrA$ is
  weakly idempotent complete then a chain map in $\Ch^{\ast}{(\scrA)}$
  becomes an isomorphism in $\D^{\ast}{(\scrA)}$ if and only if its
  cone is acyclic by Proposition~\ref{prop:weakly-idemp-compl-equiv-thick}.
\end{Rem}

\subsection{Derived Categories of Fully Exact Subcategories}
\label{sec:der-cats-fully-exact-subcats}
The proof of the following lemma is straightforward and left to the
reader as an exercise. That admissible
monics and epics are closed under composition follows from the
Noether isomorphism~\ref{lem:c/b=(c/a)/(b/a)}.

\begin{Lem}
  \label{lem:fully-exact-subcats-are-exact}
  Let $\scrA$ be an exact category and suppose that $\scrB$ is a full
  additive subcategory of $\scrA$ which is closed under extensions in
  the sense that the existence of a short exact sequence $B' \mono A
  \epi B''$ with $B',B'' \in \scrB$ implies that $A$ is isomorphic to
  an object of $\scrB$. The sequences in $\scrB$ which are exact in
  $\scrA$ form an exact structure on $\scrB$. \qed
\end{Lem}

\begin{Def}
  \label{def:fully-exact-subcat}
  A \emph{fully exact subcategory} $\scrB$ 
  of an exact category $\scrA$ is a
  full additive subcategory which is closed under extensions
  and equipped with the exact structure from the previous lemma.
\end{Def}

\begin{Thm}[{\cite[12.1]{MR1421815}}]
  \label{thm:der-cat-fully-exact-subcats}
  Let $\scrB$ be a fully exact subcategory of $\scrA$ and consider the
  functor $\D^{+}{(\scrB)} \to \D^{+}{(\scrA)}$ induced by the
  inclusion $\scrB \subset \scrA$.
  \begin{enumerate}[(i)]
    \item
      Assume that for every object $A \in \scrA$ there exists an admissible
      monic $A \mono B$ with $B \in \scrB$. For every left bounded
      complex $A$ over $\scrA$ there exists a left bounded complex $B$
      over $\scrB$ and a quasi-isomorphism $A \to B$. In particular
      $\D^{+}{(\scrB)} \to \D^{+}{(\scrA)}$ is essentially surjective.
      
    \item
      Assume that for every short exact sequence $B' \mono A \to A''$
      of $\scrA$ 
      with $B' \in \scrB$ there exists a commutative diagram
      with exact rows
      \[
      \xymatrix{
        B' \ar@{ >->}[r] \ar@{=}[d] & A \ar@{->>}[r] \ar[d] & A'' \ar[d]
        \\
        B' \ar@{ >->}[r] & B \ar@{->>}[r] & B''.
      }
      \]
      For every quasi-isomorphism $s:B \to
      A$ in $\K^{+}{(\scrA)}$  with $B$ a complex over $\scrB$ there
      exists a morphism $t:A \to B'$ in $\K^{+}{(\scrA)}$
      such that $ts:B \to B'$ is a quasi-isomorphism. In 
      particular, $\D^{+}{(\scrB)} \to \D^{+}{(\scrA)}$ is fully
      faithful.
  \end{enumerate}
\end{Thm}

\begin{Rem}
  The condition in~(ii) holds if condition~(i) holds and, moreover,
  for all short exact sequences
  $B' \mono B \epi A''$ with $B',B \in \scrB$ it follows that $A''$ is
  isomorphic to an object in $\scrB$.
  To see this, start with a short exact sequence $B' \mono A
  \epi A''$, then choose an admissible monic $A \mono B$, form the
  push-out $AA''BB''$ and apply Proposition~\ref{prop:pushout-exact}
  and Proposition~\ref{prop:pb-adm-monic-adm-monic}.
\end{Rem}

\begin{Exm}
  Let $\scrI$ be the full subcategory spanned by the \emph{injective}
  objects of the exact category $(\scrA,\scrE)$, see
  Definition~\ref{def:projective-injective}. Clearly, $\scrI$ is
  a fully exact subcategory of $\scrA$ (the induced exact structure
  consists of the split exact sequences) and it satisfies 
  condition~(ii) of Theorem~\ref{thm:der-cat-fully-exact-subcats}. If
  $\scrI$ satisfies condition~(i) then there are 
  \emph{enough injectives} in $(\scrA,\scrE)$, see
  Definition~\ref{def:enough-projectives}. A quasi-isomorphism of
  left bounded complexes of injectives is a chain homotopy
  equivalence, hence $\K^{+}{(\scrI)}$ is equivalent to
  $\D^{+}{(\scrI)}$. By Theorem~\ref{thm:der-cat-fully-exact-subcats}
  $\K^{+}{(\scrI)}$ is equivalent to the full subcategory of
  $\D{(\scrA)}$ spanned by the left bounded complexes with injective
  components. Moreover, if $(\scrA,\scrE)$ has enough injectives, then
  the functor $\K^{+}{(\scrI)} \to \D^{+}{(\scrA)}$ is an equivalence
  of triangulated categories.
\end{Exm}
\if{0}
\begin{proof}
  Point~(i) follows from the proof of
  Theorem~\ref{thm:bounded-above-qi-cx-proj} (note that projectivity
  is not used in the argument).

  To see 
\end{proof}
\fi
\subsection{Total Derived Functors}
\label{sec:total-der-fct}

With these constructions at hand one can now introduce (total) derived
functors in the sense of Grothendieck-Verdier and Deligne, see e.g.
Keller~\cite[\S\S {13-15}]{MR1421815} or any one of the references
given in~Remark~\ref{rem:htpy-cat-triangulated}. We follow Keller's
exposition of the Deligne approach.

The problem is the following: An additive functor $F: \scrA \to
\scrB$ from an exact category to another induces functors
$\Ch{(\scrA)} \to \Ch{(\scrB)}$ and $\K{(\scrA)} \to \K{(\scrB)}$ in
an obvious way. By abuse of notation we still denote these functors by
$F$. The next question to ask is whether the functor descends to a
functor of derived categories, i.e., we look for a commutative diagram
\[
\xymatrix{
  \K{(\scrA)} \ar[d]_{Q_{\scrA}} \ar[r]^-{F} & 
  \K{(\scrB)} \ar[d]^{Q_{\scrB}} \\
  \D{(\scrA)} \ar@{.>}[r]^-{\exists ?} & \D{(\scrB)}.
}
\]
If the functor $F: \scrA \to \scrB$
is exact, this problem has a solution by the
universal property of the derived category since then $F(\Ac{(\scrA)})
\subset \Ac{(\scrB)}$.

However, if $F$ fails to be exact, it will not send acyclic complexes to
acyclic complexes, or, in other words, it will not send
quasi-isomorphisms to quasi-isomorphisms and our na\"\i{}ve question
will have a negative answer. Deligne's solution consists in
constructing for each $A \in \D{(\scrA)}$ a functor
\[
\bfr\!{F}({-},A) : (\D{(\scrB)})^{\opp} \to \Ab.
\]
If the functor $\bfr\!{F}({-},A)$ is representable, a representing
object will be denoted by $\bfR\!{F}(A)$ and $\bfR\!{F}$ is said to be 
\emph{defined at $A$}. 
To be a little more specific, for $B \in \D{(\scrB)}$ we
define the abelian group $\bfr\!{F}(B,A)$ by the equivalence classes of diagrams
\[
\xymatrix{
  B \ar[r]^-{f} & F(A') & A' & A \ar[l]_-{s}
}
\]
where $f:B \to F(A')$ is a morphism of $\D{(\scrB)}$ and $s:A \to A'$ is a
quasi-isomorphism in $\K{(\scrA)}$. Observe the analogy to the
description of morphisms in $\D{(\scrA)}$; it is useful to think of
the diagram as ``$F$-fractions''. Accordingly, two $F$-fractions $(f,s)$ and
$(g,t)$ are said to be \emph{equivalent} if there exist commutative diagrams
\[
\xymatrix{
    & F(A') \ar[d]^-{F(v)} & A' \ar[d]_-{v}  \\
  B \ar[ur]^-{f} \ar[r]^-{h} \ar[dr]_-{g} & F(A''') & A''' &
  A \ar[ul]_-{s} \ar[l]_-{u} \ar[dl]^-{t} \\
  & F(A'') \ar[u]_-{F(w)} & A'' \ar[u]^-{w}
}
\]
in $\D{(\scrB)}$ and $\K{(\scrA)}$,
where $(h,u)$ is another $F$-fraction. On morphisms of $\D{(\scrB)}$
define $\bfr\!F{(-,A)}$ by pre-composition. By defining
$\bfr\!F$ on morphisms of $\D{(\scrA)}$ one obtains a \emph{functor} from
$\D{(\scrA)}$ to the category of functors
$(\D{(\scrB)})^{\opp} \to \Ab$. 

Let $\scrT \subset \D{(\scrA)}$
be the full subcategory of objects at which $\bfR\!{F}$ is defined and
choose for each $A \in \scrT$ a representing object $\bfR\!F(A)$ and
an isomorphism 
\[
\Hom_{\D{(\scrB)}}{({-}, \bfR\!F(A))} \xrightarrow{\sim}
\bfr\!F{({-},A)}.
\]
These choices force the definition of $\bfR\!F$ on morphisms and thus
$\bfR\!F: \scrT \to \D{(\scrB)}$ is a functor. Even
more is true:

\begin{Thm}[Deligne]
  Let $F: \K{(\scrA)} \to \K{(\scrB)}$ be a functor and let $\scrT$
  be the full subcategory of $\D{(\scrA)}$ at which $\bfR\!F$ is
  defined. Let $\scrS$ be the full subcategory of $\K{(\scrA)}$ spanned
  by the objects of $\scrT$. Denote by $I: \scrS \to \K{(\scrA)}$ the
  inclusion.
  \begin{enumerate}[(i)]
    \item
      The category $\scrT$ is a triangulated subcategory of
      $\D{(\scrA)}$ and $\scrS$ is a triangulated subcategory of
      $\K{(\scrA)}$.

    \item
      The functor $\bfR\!F: \scrT \to \D{(\scrB)}$ is a triangle
      functor and there is a morphism of triangle functors
      $Q_{\scrB} FI \Rightarrow \bfR\!F Q_{\scrA} I$.
      
    \item
      For every morphism $\mu: F \Rightarrow F'$ of triangle functors
      $\K{(\scrA)} \to \K{(\scrB)}$
      there is an induced morphism of triangle functors $\bfR\!\mu:
      \bfR\!F \Rightarrow \bfR\!F'$ on the intersection of the domains
      of $\bfR\!F$ and $\bfR\!F'$.
  \end{enumerate}
\end{Thm}

The only subtle part of the previous theorem is the fact that $\scrT$
is triangulated. The rest is a straightforward but rather tedious
verification. The essential details and references are given in
Keller~\cite[\S{}13]{MR1421815}.

The next question that arises is whether one can get some information
on $\scrT$. A complex $A$ is said to be \emph{$F$-split} if $\bfR\!F$ is
defined at $A$ and the canonical morphism $F(A) \to \bfR\!F(A)$ is
invertible. An object $A$ of $\scrA$ is said to be \emph{$F$-acyclic} if it
is $F$-split when considered as complex concentrated in degree zero.

\begin{Lem}[{\cite[15.1, 15.3]{MR1421815}}]
  \label{lem:cat-of-f-acyclic-obj}
  Let $\scrC$ be a fully exact subcategory of $\scrA$ satisfying
  hypothesis~(ii) of Theorem~\ref{thm:der-cat-fully-exact-subcats}.
  Assume that the restriction of $F: \scrA \to
  \scrB$ to $\scrC$ is exact. Then each object of $\scrC$
  is $F$-acyclic.
  
  Conversely, let $\scrC$ be the full subcategory of $\scrA$
  consisting of the $F$-acyclic objects.
  Then $\scrC$ is a fully exact
  subcategory of $\scrA$, it satisfies condition~(ii) of
  Theorem~\ref{thm:der-cat-fully-exact-subcats} and the restriction of
  $F$ to $\scrC$ is exact.
\end{Lem}

Now let $\scrC$ be a fully exact subcategory of $\scrA$ consisting of
$F$-acyclic objects and suppose that $\scrC$ satisfies conditions~(i)
and~(ii) of Theorem~\ref{thm:der-cat-fully-exact-subcats}. 
By these assumptions, the inclusion $\scrC \to \scrA$ induces an
equivalence $\D^{+}{(\scrC)} \to \D^{+}{(\scrA)}$. As the
restriction of $F$ to $\scrC$ is exact, it yields a triangle 
functor $F: \D^{+}{(\scrC)} \to  \D^{+}{(\scrB)}$. To choose a
quasi-inverse for the canonical functor $\D^{+}{(\scrC)} \to
\D^{+}{(\scrA)}$ amounts to choosing for each complex $A \in
\K^{+}{(\scrA)}$ a quasi-isomorphism $s:A \to C$ with $C \in
\K^{+}{(\scrC)}$ by~\cite[1.6]{MR907948}, a proof of which is given
in~\cite[6.7]{MR1102982}. As we have just seen, $C$ is $F$-split, hence $s$
yields an isomorphism $\bfR\!F(A) \to F(C) \cong \bfR\!F(C)$. Such a
quasi-isomorphism $A \to C$ exists by the construction in the proof of
Theorem~\ref{thm:bounded-above-qi-cx-proj} and our assumptions.

The admittedly concise \emph{r\'{e}sum\'{e}} given here provides the
basic toolkit for treating derived functors. We refer to
Keller~\cite[\S\S{}13--15]{MR1421815} for a much more thorough and
general discussion and precise statements of the composition formula
$\bfR\!F \circ \bfR\!G \cong \bfR\!(FG)$ and adjunction formul\ae{} of
left and right derived functors of adjoint pairs of functors,

\section{Projective and Injective Objects}
\label{sec:proj-inj-obj}

\begin{Def}
  \label{def:projective-injective}
  An object $P$ of an exact category $\scrA$ is called
  \emph{projective} if the represented functor
  $\Hom_{\scrA}{(P,{-})}: \scrA \to \Ab$ is exact. An object
  $I$ of an exact category is called \emph{injective} if the
  corepresented functor $\Hom_{\scrA}{({-},I)}: \scrA^{\opp} \to \Ab$
  is exact.
\end{Def}

\begin{Rem}
  The concepts of projectivity and injectivity are dual to each other
  in the sense that $P$ is projective in $\scrA$ if and only if $P$ is
  injective in $\scrA^{\opp}$. For our purposes it is therefore
  sufficient to deal with projective objects. 
\end{Rem}

\begin{Prop}
  \label{prop:projectivity}
  An object $P$ of an exact category is projective if and only if any
  one of the following conditions holds:
  \begin{enumerate}[(i)]
    \item
      For all admissible epics $A \epi A''$ and all morphisms
      $P \to A''$ there exists a solution to the lifting problem
      \[
      \xymatrix{
        & P \ar@{.>}[dl]_{\exists} \ar[dr] \\
        A \ar@{->>}[rr] & & A''
      }
      \]
      making the diagram above commutative.
      
    \item
      The functor $\Hom_{\scrA}{(P,{-})}: \scrA \to \Ab$ sends
      admissible epics to surjections.
      
    \item
      Every admissible epic $A \epi P$ splits (has a right inverse).
  \end{enumerate}
\end{Prop}

\begin{proof}
  Since $\Hom_{\scrA}{(P,{-})}$ transforms exact sequences to left
  exact sequences in $\Ab$ for all objects $P$ 
  (see the proof of Corollary~\ref{cor:represented-functor-sheaf}), 
  it is clear that
  conditions~(i) and~(ii) are equivalent to the projectivity of $P$. If
  $P$ is projective and $A \epi P$ is an admissible epic then 
  $\Hom_{\scrA}{(P,A)} \epi \Hom_{\scrA}{(P,P)}$ is surjective, and
  every pre-image of $1_{P}$ is a splitting map of
  $A \epi P$. Conversely, let us prove that condition~(iii) implies
  condition~(i): given a lifting problem as in~(i), form the following
  pull-back diagram
  \[
  \xymatrix{
    D \ar[d]_{f'} \ar@{->>}[r]^{a'} \ar@{}[dr]|{\text{PB}}
    & P \ar[d]^{f} \\
    A \ar@{->>}[r]^{a} & A''.
  }
  \]
  By hypothesis, there exists a right inverse $b'$ of $a'$
  and $f'b'$ solves the lifting problem because 
  $a f' b' = f a' b' = f$.
\end{proof}

\begin{Cor}
  If $P$ is projective and $P \to A$ has a right inverse then $A$ is
  projective. 
\end{Cor}
\begin{proof}
  This is a trivial consequence of condition~(i) in
  Proposition~\ref{prop:projectivity}.
\end{proof}

\begin{Rem}
  If $\scrA$ is weakly idempotent complete, the above corollary
  amounts to the familiar ``direct summands of projective objects are
  projective'' in abelian categories.
\end{Rem}

\begin{Cor}
  A sum $P = P' \oplus P''$ is projective if and only if
  both $P'$ and $P''$ are projective. \qed
\end{Cor}

More generally:

\begin{Cor}
  Let $\{P_{i}\}_{i \in I}$ be a family of objects for which the
  coproduct $P = \coprod_{i \in I} P_{i}$ exists in $\scrA$. 
  The object $P$ is projective if and only if each $P_{i}$ is
  projective. \qed
\end{Cor}

\begin{Rem}
  The dual of the previous result is that a product (if it
  exists) is injective if and only if each of its factors is injective.
\end{Rem}

\begin{Def}
  \label{def:enough-projectives}
  An exact category $\scrA$ is said to have \emph{enough projectives}
  if for every object $A \in \scrA$ there exists a projective object
  $P$ and an admissible epic $P \epi A$.
\end{Def}

\begin{Exer}[{Heller~\cite[5.6]{MR0100622}}]
  Assume that $\scrA$ has enough projectives. Prove that $A' \to A \to
  A''$ is short exact if and only if
  \[
  \Hom_{\scrA}{(P,A')} \mono \Hom_{\scrA}{(P,A)} \epi
  \Hom_{\scrA}{(P,A'')}
  \]
  is short exact for all projective objects $P$.

  \emph{Hint:} For sufficiency prove first that $A' \to A$ is a
  monomorphism, then prove that it is a kernel of $A \to A''$ and finally
  apply the obscure axiom~\ref{prop:obscure-axiom}. 
  In all three steps use that there are enough projectives.
\end{Exer}
 
\begin{Exer}[{Heller~\cite[5.6]{MR0100622}}]
  Assume that $\scrA$ is weakly idempotent complete and has enough
  projectives. Prove that the sequence
  \[
  A_{n} \to A_{n-1} \to \cdots \to A_{1} \to A_{0} \to 0
  \]
  is an exact sequence of admissible morphisms if and only if for all
  projectives $P$ the sequence
  \[
  \Hom_{\scrA}{(P, A_{n})} \to \Hom_{\scrA}{(P,A_{n-1})} \to \cdots 
  \to \Hom_{\scrA}{(P,A_{1})} \to \Hom_{\scrA}{(P,A_{0})} \to 0
  \]
  is an exact sequence of abelian groups.
\end{Exer}

\section{Resolutions and Classical Derived Functors}
\label{sec:resolutions}

\begin{Def}
  A \emph{projective resolution} of the object $A$ is a positive
  complex $P_{\bullet}$ with projective components together with a morphism
  $P_{0} \to A$ such that the \emph{augmented complex}
  \[
  \cdots \to P_{n+1} \to P_{n} \to \cdots \to P_{1} \to P_{0}
  \to A
  \]
  is exact.
\end{Def}

\begin{Prop}[Resolution Lemma]
  \label{prop:resolution-lemma}
  If $\scrA$ has enough projectives then every object
  $A \in \scrA$ has a projective resolution.
\end{Prop}
\begin{proof}
  This is an easy induction.
  Because $\scrA$ has enough projectives, there exists a projective
  object $P_0$ and an admissible epic $P_0 \epi A$ with $P_{0}$.
  Choose an admissible monic $A_0 \mono P_0$ such that
  $A_0 \mono P_0 \epi A$ is exact. Now choose a projective
  $P_1$ and an admissible epic $P_1 \epi A_0$. Continue with
  an admissible monic $A_1 \mono P_1$ such that $A_1 \mono P_1 \epi A_0$
  is exact, and so on. One thus obtains a sequence
  \[
  \xymatrix@R=0.5pc{
    & & A_1 \ar@{ >->}[dr] \\
    \cdots & P_2 \ar@{->>}[ur] \ar[rr] & & 
    P_1 \ar@{->>}[dr] \ar[rr] & & 
    P_0 \ar@{->>}[r] & A \\
    & & & & A_0 \ar@{ >->}[ur]
  }
  \]
  which is exact by construction, so $P_{\bullet} \to A$ is 
  a projective resolution.
\end{proof}

\begin{Rem}
  \label{rem:enough-p-obj-give-res}
  The defining concept of projectivity is not used in the previous
  proof. That is, we have proved:
  If $\scrP$ is a class in $\scrA$ such that for each object
  $A \in \scrA$ there is an admissible epic $P \epi A$ with $P \in
  \scrP$ then each
  object of $\scrA$ has a $\scrP$-resolution $P_{\bullet} \epi A$.
\end{Rem}

Consider a morphism $f: A \to B$ in $\scrA$.
Let $P_{\bullet}$ be a complex of projectives with $P_n = 0$ for 
$n < 0$ and let $\alpha: P_0 \to A$ be a morphism such that the composition
$P_1 \to P_0 \to A$ is zero
[e.g. $P_{\bullet} \to A$ is a projective resolution of $A$].
Let $Q_{\bullet} \xrightarrow{\beta} B$ be a resolution (not
necessarily projective).

\begin{Thm}[Comparison Theorem]
  \label{thm:comparison-theorem}
  Under the above hypotheses there exists a chain map
  $f_\bullet: P_\bullet \to Q_\bullet$
  such that the following diagram commutes:
  \[
  \xymatrix{
    \cdots \ar[r] & 
    P_2 \ar[r] \ar[d]_{\exists f_2} &
    P_1 \ar[r] \ar[d]_{\exists f_1} &
    P_0 \ar[r]^\alpha \ar[d]_{\exists f_0} &
    A \ar[d]_{f} \\
    \cdots \ar[r] &
    Q_2 \ar[r] &
    Q_1 \ar[r] &
    Q_0 \ar@{->>}[r]^\beta &
    B.
  }
  \]
  Moreover, the \emph{lift} $f_\bullet$ of $f$ is unique up to
  homotopy equivalence.
\end{Thm}

\begin{proof}
  It is convenient to put $P_{-1} = A$,
  $Q_{0}' = Q_{-1} = B$ and $f_{-1} = f$. 
  
  \emph{Existence:}
  The question of existence of $f_0$ is the lifting problem given by the map
  $f\alpha: P_0 \to B$ and the admissible epic $\beta: Q_0 \epi B$. This
  problem has a solution by projectivity of $P_0$.
  
  Let $n \geq 0$ and suppose by induction that there are morphisms
  $f_n: P_n \to Q_n$ and $f_{n-1}: P_{n-1} \to Q_{n-1}$ such that
  $d f_{n} = f_{n-1} d$.  Consider the following diagram:
  \[
  \xymatrix@R=0.8pc{
    P_{n+1} \ar@{.>}[dd]_{\exists f_{n+1}}
    \ar@{.>}[dr]^{\exists! f_{n+1}'} \ar[rr]^{d} & &
    P_{n} \ar[rr]^{d} \ar[dd]^{f_{n}} & &
    P_{n-1} \ar[dd]^{f_{n-1}} \\
    & {Q_{n+1}'} \ar@{ >->}[dr] \\
    Q_{n+1} \ar@{->>}[ur] \ar[rr]^{d} & &
    Q_{n} \ar@{->>}[dr] \ar[rr]^{d} & &
    Q_{n-1} \\
    & & & {Q_{n}'} \ar@{ >->}[ur]
  }
  \]
  By induction the right hand square is commutative, so the
  morphism $P_{n+1} \to Q_{n-1}$ is zero because the morphism
  $P_{n+1} \to P_{n-1}$ is zero.  The morphism $P_{n+1} \to Q_{n}'$ is
  zero as well because $Q_{n}' \mono Q_{n-1}$ is monic. Since
  $Q_{n+1}' \mono Q_{n} \epi Q_{n}'$ is exact, there exists a
  unique morphism $f_{n+1}': P_{n+1} \to Q_{n+1}'$ making the upper
  right triangle in the left hand square commute. Because $P_{n+1}$ 
  is projective and $Q_{n+1} \epi Q_{n+1}'$ is an admissible epic, there 
  is a morphism $f_{n+1}: P_{n+1} \to Q_{n+1}$ such that the left hand
  square commutes. This settles the existence of $f_{\bullet}$.
  
  \emph{Uniqueness:}
  Let $g_{\bullet}: P_{\bullet} \to Q_{\bullet}$ be another
  lift of $f$ and put $h_\bullet = f_\bullet - g_\bullet$. We will
  construct by induction a chain contraction $s_n: P_{n-1} \to Q_n$
  for $h$.
  For $n \leq 0$ we put $s_n = 0$. For $n \geq 0$ assume by induction
  that there are morphisms $s_{n-1}, s_n$ such that
  $h_{n-1} = d s_n + s_{n-1} d$. Because of this assumption
  and the fact that $h$ is a chain map, we have
  $d(h_{n} - s_{n} d) = h_{n-1} d - (h_{n-1} - s_{n-1}d)d = 0$
  so the following diagram commutes
  
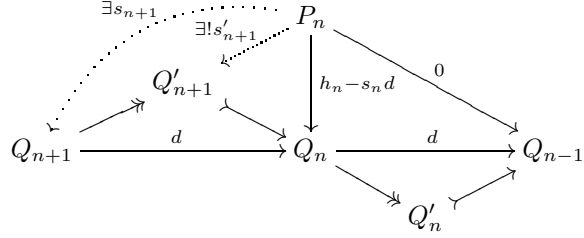
\[
  \xymatrix@R=0.8pc{
    & & P_n \ar[dd]^{h_n - s_{n}d} 
    \ar@{.>}[dl]_{\exists! s_{n+1}'} \ar@/_2pc/@{.>}[ddll]_{\exists s_{n+1}}
    \ar[ddrr]^0\\
    & {Q_{n+1}'} \ar@{ >->}[dr] \\
    Q_{n+1} \ar@{->>}[ur] \ar[rr]^d & &
    Q_{n} \ar[rr]^{d} \ar@{->>}[dr] & & Q_{n-1} \\
    & & & {Q_{n}'} \ar@{ >->}[ur]
  }
  \]
  and we get a morphism $s_{n+1}: P_{n} \to Q_{n+1}$ such that
  $ds_{n+1} = h_{n} - s_{n}d$ as in the existence proof.
\end{proof}


\begin{Cor}
  \label{cor:proj-res-htpy-equiv}
  Any two projective resolutions of
  an object $A$ are chain homotopy equivalent. \qed
\end{Cor}


\begin{Cor}
  \label{cor:no-mor-proj-cxes-to-ac-cxes}
  Let $P_{\bullet}$ be a right bounded complex of projectives and let
  $A_{\bullet}$ be an acyclic complex. Then
  $\Hom_{\K{(\scrA)}}{(P_{\bullet},A_{\bullet})} = 0$. \qed
\end{Cor}

In order to deal with derived functors on the level of the derived
category, one needs to sharpen both the resolution lemma and the
comparison theorem.

\begin{Thm}[{\cite[4.1, Lemma, b)]{MR1052551}}]
  \label{thm:bounded-above-qi-cx-proj}
  Let $\scrA$ be an exact category with enough projectives. For every
  right bounded complex $A$ over $\scrA$ exists a right bounded
  complex with projective components $P$ and
  a quasi-isomorphism $P \xrightarrow{\alpha} A$.
\end{Thm}

\begin{proof}
  Renumbering if necessary, we may suppose $A_n = 0$ for $n < 0$. The
  complex $P$ will be constructed by induction. For the inductive formulation
  it is convenient to define $P_{n} = B_{n} = 0$ for $n < 0$. Put $B_0 = A_0$,
  choose an admissible epic $p_{0}': P_{0} \epi B_0$ from a projective
  $P_0$ and define
  $p_{0}'' = d^{A}_0$. Let $B_1$ be the pull-back over $p_{0}'$
  and $p_{0}''$. Consider the following commutative diagram:
  
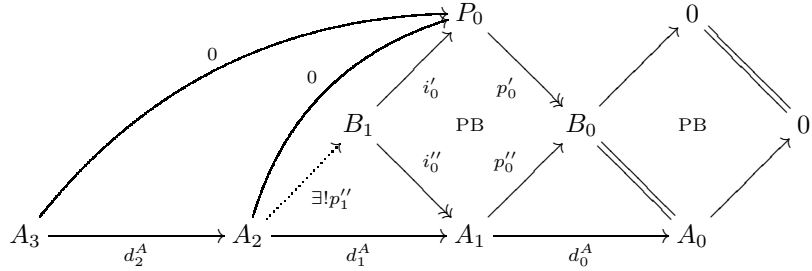
\[
  \xymatrix@!{
    & & & & P_{0} \ar@{->>}[dr]_{p_{0}'} & & 0 \ar@{=}[dr] \\
    & & &
    B_{1} \ar@{->>}[dr]^{i_{0}''} \ar[ur]_{i_{0}'} \ar@{}[rr]|{\text{PB}} & & 
    B_{0} \ar@{=}[dr] \ar[ur] \ar@{}[rr]|{\text{PB}} & &
    0 \\
    A_{3} \ar[rr]_{d^{A}_2} \ar@/^+2pc/[uurrrr]^{0} & &
    A_{2} \ar[rr]_{d^{A}_1} \ar@/^+1.5pc/[uurr]^{0} 
    \ar@{.>}[ur]_{\exists!p_{1}''} & & 
    A_{1} \ar[rr]_{d^{A}_0} \ar[ur]^{p_{0}''}& & 
    A_{0} \ar[ur]
  }
  \]
  The morphism $p_{1}''$ exists by the universal property of the pull-back
  and moreover $p_{1}''d^{A}_2 = 0$ because $d^{A}_1 d^{A}_2 = 0$.

  Suppose by induction that in the following diagram everything is constructed
  except $B_{n+1}$ and the morphisms terminating or issuing from there.
  Assume further that $P_n$ is projective and that $p_{n}''d^{A}_{n+1} = 0$.
  
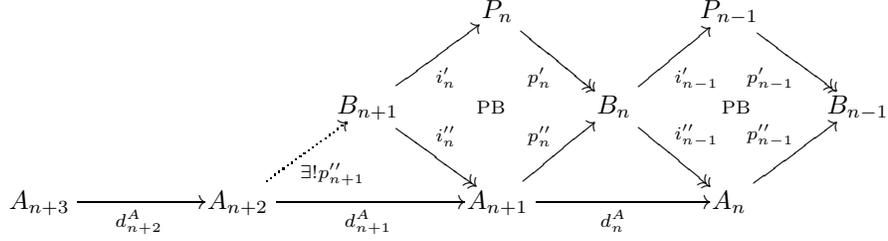
\[
  \xymatrix@1{
    & & & & P_{n} \ar@{->>}[dr]_{p_{n}'} & & P_{n-1} \ar@{->>}[dr]_{p_{n-1}'}
    \\
    & & &
    B_{n+1} \ar@{->>}[dr]^{i_{n}''} \ar[ur]_{i_{n}'} 
    \ar@{}[rr]|{\text{PB}} & & 
    B_{n} \ar@{->>}[dr]^{i_{n-1}''} \ar[ur]_{i_{n-1}'}
    \ar@{}[rr]|{\text{PB}} & &
    B_{n-1} \\
    A_{n+3} \ar[rr]_{d^{A}_{n+2}} & & 
    A_{n+2} \ar[rr]_{d^{A}_{n+1}} 
    \ar@{.>}[ur]_{\exists!p_{n+1}''} & & 
    A_{n+1} \ar[rr]_{d^{A}_{n}} \ar[ur]^{p_{n}''}& & 
    A_{n} \ar[ur]^{p_{n-1}''}
  }
  \]
  As indicated in the diagram, we obtain $B_{n+1}$ by forming the pull-back
  over $p_{n}'$ and $p_{n}''$. We complete the induction by choosing
  an admissible epic $p_{n+1}' : P_{n+1} \epi B_{n+1}$ from a projective
  $P_{n+1}$,
  constructing $p_{n+1}''$ as in the first paragraph and finally noticing that
  $p_{n+1}''d^{A}_{n+2} = 0$.

  The projective complex is given by the $P_n$'s and the
  differential $d^{P}_{n-1} = i_{n-1}' p_{n}'$, which satisfies
  $(d^{P})^2 = 0$ by construction.

  Let $\alpha$ be
  given by $\alpha_{n} = i_{n-1}''p_{n}'$ in degree $n$, manifestly
  a chain map. We claim that $\alpha$ is a quasi-isomorphism.
  The mapping cone of $\alpha$ is seen to be exact using
  Proposition~\ref{prop:pushout-exact}:
  For each $n$ there is an exact sequence
  \[
  B_{n+1} \xrightarrow{i_n = \mat{ -i_{n}' \\ i_{n}''}}  P_n \oplus A_{n+1}
  \xrightarrow{p_{n} = \mat{p_{n}' & p_{n}''}} B_n.
  \]
  We thus obtain an exact complex 
  $C$ with $C_{n} =  P_{n} \oplus A_{n+1}$ in degree~$n$ and differential
  \[
  d^{C}_{n-1} = i_{n-1}p_{n} =
  \mat{
    - i_{n-1}' p_{n}' & - i_{n-1}' p_{n}'' \\
    i_{n-1}'' p_{n}' & i_{n-1}'' p_{n}''
  } =
  \mat{ 
    - d^{P}_{n-1} & 0 \\
    \alpha_{n} & d^{A}_{n}
  }
  \]
  which shows that $C = \cone{(\alpha)}$.
\end{proof}

\begin{Thm}[Horseshoe Lemma]
  \label{thm:horseshoe-lemma}
  A horseshoe can be filled in:
  Suppose we are given a horseshoe diagram
  \[
  \xymatrix{
    \cdots \ar[r] & P_{2}' \ar[r] & P_{1}' \ar[r] & P_{0}'
    \ar@{->>}[r] & A' \ar@{ >->}[d] \\
    & & & & A \ar@{->>}[d] \\
    \cdots \ar[r] & P_{2}'' \ar[r] & P_{1}'' \ar[r] & P_{0}''
    \ar@{->>}[r] & A'',
  } 
  \]
  that is to say, the column is short exact and the horizontal rows
  are projective resolutions of $A'$ and $A''$. Then the direct sums
  $P_{n} = P_{n}' \oplus P_{n}''$ assemble to a projective resolution
  of $A$ in such a way that the horseshoe can be embedded into a
  commutative diagram with exact rows and columns
  \[
  \xymatrix{
    \cdots \ar[r]  & P_{2}' \ar[r] \ar@{ >->}[d] &
    P_{1}' \ar[r] \ar@{ >->}[d] & P_{0}' \ar@{ >->}[d]
    \ar@{->>}[r] & A' \ar@{ >->}[d] \\
    \cdots \ar[r] & P_{2} \ar[r] \ar@{->>}[d] & P_{1} \ar[r] \ar@{->>}[d]
    & P_{0} \ar@{->>}[r] \ar@{->>}[d] & A \ar@{->>}[d] \\
    \cdots \ar[r] & P_{2}'' \ar[r] & P_{1}'' \ar[r] & P_{0}''
    \ar@{->>}[r] & A''.
  }
  \]
\end{Thm}

\begin{Rem}
  All the columns except the rightmost one are split exact. However,
  the morphisms $P_{n+1} \to P_{n}$ are \emph{not} the sums of the
  morphisms $P_{n+1}' \to P_{n}'$ and $P_{n+1}'' \to P_{n}''$. This
  only happens in the trivial case that the sequence $A' \mono A \epi
  A''$ is already split exact.
\end{Rem}

\begin{proof}
  This is an easy application of the five lemma~\ref{cor:five-lemma}
  and the $3 \times 3$-lemma~\ref{cor:3x3-lemma}. 
  By lifting the morphism 
  $\varepsilon'': P_{0}'' \to A''$ over the admissible epic
  $A \epi A''$ we obtain a morphism $\varepsilon: P_{0} \to A$ and 
  a commutative diagram
  \[
  \xymatrix{
    \Ker{\varepsilon'} 
    \ar@{ >->}[r] \ar@{.>}[d] &
    P_{0}' \ar@{->>}[r]^{\varepsilon'} \ar@{ >->}[d]^{\mat{1 \\ 0}} &
    A' \ar@{ >->}[d] \\
    \Ker{\varepsilon}
    \ar@{ >.>}[r] \ar@{.>}[d] &
    P_{0} \ar[r]^{\varepsilon} \ar@{->>}[d]^{\mat{0 & 1}} &
    A \ar@{->>}[d] \\
    \Ker{\varepsilon''} \ar@{ >->}[r] &
    P_{0}'' \ar@{->>}[r]^{\varepsilon''}  &
    A''.
  }
  \]
  It follows from the five lemma that $\varepsilon$ is actually an
  admissible epic, so its kernel exists. The two vertical dotted
  morphisms exist since the second and the third column are short
  exact. Now the $3 \times 3$-lemma implies that the dotted column is
  short exact. Finally note that $P_{1}' \to P_{0}'$ and $P_{1}''
  \to P_{0}''$ factor over admissible epics to $\Ker{\varepsilon'}$ and
  $\Ker{\varepsilon''}$ and proceed by induction.
\end{proof}

\begin{Rem}
  In concrete situations it may be useful to remember that only the
  projectivity of $P_{n}''$ is used in the proof.
\end{Rem}

\begin{Rem}[Classical Derived Functors]
  Using the results of this section, the theory of classical 
  derived functors, see 
  e.g. Cartan-Eilenberg~\cite{MR1731415}, 
  Mac~Lane~\cite{MR0156879}, Hilton-Stammbach~\cite{MR1438546} or 
  Weibel~\cite{MR1269324}, is easily adapted to the following situation:
  
  Let $(\scrA, \scrE)$ be an exact category with enough
  projectives and let $F: \scrA \to \scrB$ be an additive functor
  to an abelian category. By the resolution
  lemma~\ref{prop:resolution-lemma} a projective 
  resolution $P_{\bullet} \epi A$ exists for every object $A \in \scrA$
  and is well-defined up to homotopy equivalence by the comparison
  theorem (Corollary~\ref{cor:proj-res-htpy-equiv}). It follows that for
  two projective resolutions $P_{\bullet} \epi A$ and
  $Q_{\bullet} \epi A$ the complexes $F(P_{\bullet})$ and
  $F(Q_{\bullet})$ are chain homotopy equivalent. Therefore it makes
  sense to define the \emph{left derived functors} of $F$ as
  \[
  L_{i}F (A) := H_{i} (F(P_{\bullet})).
  \]
  Let us indicate why $L_{i}F(A)$ is a functor. First observe that
  a morphism $f: A \to A'$ extends uniquely up to chain
  homotopy equivalence to a chain map $f_{\bullet}: P_{\bullet} \to
  P_{\bullet}'$ if $P_{\bullet} \epi A$ and $P_{\bullet}' \epi A'$ are
  projective resolutions of $A$ and $A'$. From this uniqueness it
  follows easily that $L_{i}F(fg) = L_{i}F(f) L_{i}F(g)$ and
  $L_{i}F(1_{A}) = 1_{L_{i}F(A)}$ as desired. Using the horseshoe
  lemma~\ref{thm:horseshoe-lemma} one proves that a short exact
  sequence $A' \mono A \epi A''$ yields a long exact sequence 
  \[
  \cdots \to L_{i+1}F(A'') \to L_{i}F(A') \to L_{i}F(A) \to
  L_{i}F(A'') \to L_{i-1}F(A') \to \cdots
  \]
  and that $L_{0}F$ sends exact sequences to right exact sequences in
  $\scrB$ so that the $L_{i}F$ are a universal
  $\delta$-functor. Moreover, $L_{0}F$ is characterized by being the
  best right exact approximation to $F$ and the $L_{i}F$ measure the
  failure of $L_{0}F$ to be exact. In particular, if $F$ sends
  exact sequences to right exact sequences then $L_{0}F \cong F$ and
  if $F$ is exact, then in addition $L_{i}F = 0$ if $i > 0$.
\end{Rem}

\begin{Rem}
  By the discussion in section~\ref{sec:total-der-fct}, the assumption
  that $(\scrA,\scrE)$ has enough projectives is unnecessarily
  restrictive. In order for the classical 
  left derived functor of $F:\scrA \to
  \scrB$ to exist, it suffices to assume that there is a fully exact
  subcategory $\scrC \subset \scrA$ satisfying the duals of the
  conditions in Theorem~\ref{thm:der-cat-fully-exact-subcats} with the
  additional property that $F$ restricted to $\scrC$ is exact (see
  Lemma~\ref{lem:cat-of-f-acyclic-obj}). These conditions ensure that
  the total derived functor $\bfL\!F: \D^{-}{(\scrA)} \to
  \D^{-}{(\scrB)}$ exists and thus it makes sense
  to define $L_{i}F(A) = H_{i}{(\bfL\!F (A))}$, where the object $A
  \in \scrA$ is considered as a complex concentrated in degree
  zero. More explicitly, choose a $\scrC$-resolution
  $C_{\bullet} \to A$ and let $L_{i}F(A) := H_{i}{(F(C_{\bullet}))}$.
  It is not difficult to check that the $L_{i}F$ are a universal
  $\delta$-functor: They form a $\delta$-functor as $\bfL\!F$ is a
  triangle functor and $H_{\ast}:\D^{-}{(\scrB)} \to \scrB$ 
  sends distinguished triangles to long exact sequences; this
  $\delta$-functor is universal because it is effa\c{c}able,
  as $L_{i}F(C) = 0$ for $i > 0$.
\end{Rem}

\begin{Exer}[{Heller~\cite[6.3, 6.5]{MR0100622}}]
  Let $(\scrA,\scrE)$ be an exact category and consider the exact
  category $(\scrE,\scrF)$ as described in
  Exercise~\ref{exer:ses-exact}.
  Prove that an exact sequence
  $P' \mono P \epi P''$ is projective in $(\scrE,\scrF)$ if $P'$ and
  $P''$ (and hence $P$) are projective in
  $(\scrA,\scrE)$. If $(\scrA,\scrE)$ has
  enough projectives then so has $(\scrE,\scrF)$ and every projective
  is of the form described before.
\end{Exer}

\section{Examples and Applications}
\label{sec:examples}

It is of course impossible to give an exhaustive list of examples. We
simply list some of the popular ones.

\subsection{Additive Categories}

Every additive category $\scrA$ is exact with respect to the class
$\scrE_{\min}$ of split exact sequences, i.e., the sequences isomorphic to 
\[
A \xrightarrow{\mat{1 \\ 0}} A \oplus B \xrightarrow{\mat{0 & 1}} B
\]
for $A,B \in \scrA$. Every object $A \in \scrA$ is both projective and
injective with respect to this exact structure.

\subsection{Quasi-Abelian Categories}

We have seen in Section~\ref{sec:quasi-ab-cats} that quasi-abelian
categories are exact with respect to the class $\scrE_{\max}$ 
of all kernel-cokernel pairs. Evidently, this class of examples
includes in particular all abelian categories. There is an abundance
of non-abelian quasi-abelian
categories arising in functional analysis:

\begin{Exm}[{Cf. e.g. \cite[IV.2]{mythesis}}]
  Let $\Ban$ be the category of Banach spaces and bounded linear
  maps over the field $k$ of real or complex numbers. It has kernels
  and cokernels---the cokernel of a morphism $f: A
  \to B$ is given by $B/\overline{\Im{f}}$. It is an easy consequence
  of the open mapping theorem that $\Ban$ is quasi-abelian. Notice
  that the forgetful functor $\Ban \to \Ab$ is exact and
  reflects exactness, it preserves monics but fails to preserve epics
  (morphisms with dense range). The ground field $k$ is projective and
  by Hahn-Banach it also is injective. More generally, it
  is easy to see that for each set $S$ the space $\ell^{1}{(S)}$
  is projective and $\ell^{\infty}{(S)}$ is injective. Since every
  Banach space $A$ is isometrically isomorphic to a quotient of
  $\ell^{1}{(B_{\leq 1}{A})}$ and to a subspace of
  $\ell^{\infty}{(B_{\leq 1} A^{\ast})}$ there are enough of both,
  projective and injective objects in $\Ban$.
\end{Exm}

\begin{Exm}
  Let $\Fre$ be the category of completely metrizable topological
  vector spaces (Fr\'echet spaces) and continuous linear maps.
  Again, $\Fre$ is quasi-abelian by
  the open mapping theorem (the proof of Theorem~2.3.3 in Chapter~IV.2
  of~\cite{mythesis} applies \emph{mutatis mutandis}),
  and there are exact functors
  $\Ban \to \Fre$ and $\Fre \to \Ab$. It is still true that $k$ is
  projective, but $k$ fails to be injective
  (Hahn-Banach breaks down). 
\end{Exm}

\begin{Exm}
  Consider the category $\Pol$ of polish abelian groups (i.e., second
  countable and completely metrizable topological groups) and
  continuous homomorphisms. From the open mapping theorem---which is a
  standard consequence of Pettis' theorem (cf. 
  e.g.~\cite[(9.9), p.~61]{MR1321597}) 
  stating that for a non-meager set $A$ in $G$ the
  set $A^{-1}A$ is a neighborhood of the identity---it follows that
  $\Pol$ is quasi-abelian (again one easily adapts the proof of
  Theorem~2.3.3 in Chapter~IV.2 of~\cite{mythesis}).
\end{Exm}

Further functional analytic examples are discussed in detail e.g. in
Rump~\cite{MR1749013} and
Schneiders~\cite{MR1779315}. Rump~\cite{rump-counterexample} gives a
rather long list of examples.

\subsection{Fully Exact Subcategories}

Recall from section~\ref{sec:der-cats-fully-exact-subcats} that a
\emph{fully exact subcategory} $\scrB$ of an exact category $\scrA$ is a
full subcategory $\scrB$ which is closed under extensions
and equipped with the exact structure formed by the sequences which
are exact in $\scrA$ (see Lemma~\ref{lem:fully-exact-subcats-are-exact}).

\begin{Exm}
  By the embedding theorem~\ref{thm:embedding-thm},
  every small exact category is a fully
  exact subcategory of an abelian category.
\end{Exm}

\begin{Exm}
  The full subcategories of projective or injective objects of an
  exact category $\scrA$ are fully exact. The induced exact structures
  are the split exact structures.
\end{Exm}

\begin{Exm}
  Let $\prtens$ be the projective tensor product of Banach spaces. A
  Banach space $F$ is \emph{flat} if $F \prtens {-}$ is exact. It is
  well-known that the flat Banach spaces are precisely the
  $\scrL_{1}$-spaces of Lindenstrauss-Pe\l{}czy\'n{}ski.
  The category of flat Banach spaces is a fully exact subcategory of 
  $\Ban$. The exact
  structure is the pure exact structure consisting of the
  short sequences whose Banach dual sequences are split exact, 
  see~\cite[Ch.~IV.2]{mythesis} for further information and references.
\end{Exm}

\subsection{Frobenius Categories}
\label{sec:frobenius-cats}
An exact category is said to be \emph{Frobenius} provided that it has
enough projectives and injectives and, moreover, the classes of
projectives and injectives coincide~\cite{MR0116045}. Frobenius
categories $\scrA$ give rise to \emph{algebraic triangulated
  categories} (see \cite[3.6]{MR2275593}) by
passing to the \emph{stable category} $\underline{\scrA}$ of
$\scrA$. By definition, $\underline{\scrA}$ is the category consisting
of the same objects as $\scrA$ and in which a morphism of $\scrA$ is
identified with zero if it factors over an injective object. It is not
hard to prove that $\underline{\scrA}$ is additive and it
has the structure of a triangulated category as follows:

The translation functor is obtained by choosing for each object $A$ a
short exact sequence $A \mono I(A) \epi \Sigma(A)$
where $I(A)$ is injective. The assignment $A \mapsto \Sigma(A)$
induces an auto-equivalence of $\underline{\scrA}$. Given a morphism
$f:A \to B$ in $\scrA$ consider the push-out diagram
\[
\xymatrix{
  A \ar@{ >->}[r] \ar[d]_{f} \ar@{}[dr]|{\text{PO}} & 
  I(A) \ar[d] \ar@{->>}[r] & \Sigma(A) \ar@{=}[d] \\
  B \ar@{ >->}[r] & C(f) \ar@{->>}[r] & \Sigma(A)
}
\]
and call the sequence $A \to B \to C(f) \to \Sigma(A)$ a
\emph{standard triangle}. The class $\Delta$ of
\emph{distinguished triangles} consists of the triangles 
which are isomorphic in $\underline{\scrA}$ to
(the image of) a standard triangle.

\begin{Thm}[{Happel~\cite[2.6, p.16]{MR935124}}]
  \label{thm:happel}
  The triple $(\underline{\scrA}, \Sigma, \Delta)$ is a triangulated
  category.
\end{Thm}

\begin{Exm}
  Consider the category $\Ch{(\scrA)}$ of complexes over the additive
  category $\scrA$ equipped with the degreewise
  split exact sequences. It turns out that $\Ch{(\scrA)}$ is a
  \emph{Frobenius category}. The complex $I(A)$ introduced in the
  proof of Lemma~\ref{lem:acyclic-complexes-thick} is injective. It is
  not hard to verify that the stable category
  $\underline{\Ch{(\scrA)}}$ coincides with the homotopy category
  $\K{(\scrA)}$ and that the triangulated structure provided by
  Happel's theorem~\ref{thm:happel} is the same as the
  one mentioned in Remark~\ref{rem:htpy-cat-triangulated}
  (see also Exercise~\ref{exer:comparing-dist-triangles}).
\end{Exm}

The reader may consult Happel~\cite{MR935124} for further information,
examples and applications.

\subsection{Further Examples}

\begin{Exm}[Vector bundles]
  Let $X$ be a scheme. The category of algebraic vector bundles over
  $X$, i.e., the category of locally free and coherent
  $\scrO_{X}$-modules, is an exact category with the exact structure
  consisting of the locally split short exact sequences.
\end{Exm}

\begin{Exm}[Chain complexes]
  If $(\scrA,\scrE)$ is an exact category then the category of chain
  complexes $\Ch{(\scrA)}$ is an exact category with respect to the
  exact structure $\Ch{(\scrE)}$ of short sequences of complexes which
  are exact in each degree, see Lemma~\ref{lem:cxes-exact}.
\end{Exm}

\begin{Exm}[Diagram Categories]
  Let $(\scrA, \scrE)$ be an exact category and let $\scrD$ be a small
  category. The category $\scrA^{\scrD}$ of functors $\scrD \to \scrA$
  is an exact category with the exact structure $\scrE^{\scrD}$. 
  The verification of the axioms of an exact category
  for $(\scrA^{\scrD}, \scrE^{\scrD})$ is straightforward, as
  limits and colimits in $\scrA^{\scrD}$ are
  formed pointwise, see e.g.~Borceux~\cite[2.15.1, p. 87]{MR1291599}.
\end{Exm}

\begin{Exm}[Filtered Objects]
  Let $(\scrA, \scrE)$ be an exact category. A (bounded)
  \emph{filtered object} $A$ in $\scrA$ is a sequence of admissible
  monics in $\scrA$
  \[
  A = (\xymatrix{
    \cdots \ar@{ >->}[r] & A^{n} \ar@{ >->}[r]^-{i^{n}_{A}} & 
    A^{n+1} \ar@{ >->}[r] & \cdots
  })
  \]
  such that $A^{n} = 0$ for $n \ll 0$ and that $i^{n}_{A}$ is an
  isomorphism for $n \gg 0$. A \emph{morphism}~ $f$ from the filtered
  object $A$ to the filtered object~$B$ is
  a collection of morphisms $f^{n}: A^{n} \to B^{n}$ in $\scrA$ satisfying
  $f^{n+1}i^{n}_{A} = i^{n}_{B} f^{n}$. Thus there is a category
  $\scrF\!\scrA$ of filtered objects. It follows from 
  Proposition~\ref{prop:sum-exact} that $\scrF\!\scrA$ is additive.
  The $3 \times 3$-lemma~\ref{cor:3x3-lemma} implies that the class
  $\scrF\!\scrE$ consisting of the pairs of morphisms $(i,p)$ of
  $\scrF\!\scrA$ such that $(i^{n},p^{n})$ is in $\scrE$ for each
  $n$ is an exact structure on $\scrF\!\scrA$. Notice that
  for a nonzero abelian category $\scrA$ the category of filtered
  objects $\scrF\!\scrA$ is \emph{not} abelian.
\end{Exm}

\begin{Exm}
  Paul Balmer~\cite{MR2181829} (following Knebusch) 
  gives the following definition:
  An \emph{exact category with duality} is a triple
  $(\scrA,\ast,\varpi)$ consisting of an exact category $\scrA$, a
  contravariant and exact endofunctor $\ast$ on $\scrA$ together with 
  a natural isomorphism $\varpi: \id_{\scrA} \Rightarrow \ast \circ
  \ast$ satisfying $\varpi_{M}^{\ast} \varpi_{M^{\ast}} =
  \id_{M^{\ast}}$ for all $M \in \scrA$.
  There are natural notions of symmetric spaces and isometries of
  symmetric spaces, (admissible) Lagrangians of a symmetric space and
  hence of metabolic spaces
  If $\scrA$ is
  essentially small it makes sense to speak of the set 
  ${\rm MW}(\scrA, \ast, \varpi)$ of isometry classes of
  symmetric spaces and the subset ${\rm NW}{(\scrA,\ast,\varpi)}$ 
  of isometry classes of metabolic spaces and both turn out to be
  abelian monoids with respect to the \emph{orthogonal sum} of
  symmetric spaces. The \emph{Witt group} is
  ${\rm W}{(\scrA, \ast, \varpi)} =
  {\rm MW}{(\scrA,\ast,\varpi)}/{\rm NW}{(\scrA,\ast,\varpi)}$.
  In case $\scrA$ is the category
  of vector bundles over a scheme $(X, \scrO_{X})$ and $\ast =
  \Hom_{\scrO_{X}}({-}, \scrO_{X})$ is the usual duality functor, one
  obtains the classical Witt group of a scheme.

  Extending these considerations to the level of the derived category 
  leads to \emph{Balmer's triangular Witt groups}
  which had a number of striking applications to the theory of
  quadratic forms and $K$-theory, we refer the interested reader to
  Balmer's survey~\cite{MR2181829}. For
  a beautiful link to algebraic $K$-theory we refer to
  Schlichting~\cite{schlichting}. 
\end{Exm}

\subsection{Higher Algebraic $K$-Theory}

Let $(\scrA,\scrE)$ be a small exact category.
The \emph{Grothendieck group} $K_{0}(\scrA,\scrE)$
is defined to be the quotient of the free (abelian) group
generated by the isomorphism classes of objects of $\scrA$ modulo the
relations $[A] = [A'][A'']$ for each short exact sequence
$A' \mono A \epi A''$ in $\scrE$. This generalizes the $K$-theory of a
ring, where $(\scrA,\scrE)$ is taken to be the category of finitely
generated projective modules over $R$ with the split exact
structure. If $(\scrA,\scrE)$ is the category of algebraic vector
bundles over a scheme $X$ then by definition
$K_{0}{(\scrA,\scrE)}$ is the (na\"\i{}ve) Grothendieck group 
$K_{0}{(X)}$ of the scheme (cf. \cite[3.2, p.~313]{MR1106918}).

Quillen's landmark paper~\cite{MR0338129} introduces today's
definition of higher algebraic $K$-theory and proves its basic
properties. Exact
categories enter via the $Q$-construction, which we outline briefly. 
Given a small exact category
$(\scrA, \scrE)$ one forms a new category $Q\scrA$:
The objects of $Q\scrA$ are the objects of $\scrA$ and
$\Hom_{Q\scrA}(A, A')$ is defined to be the set of 
equivalence classes of diagrams
\[
\xymatrix{
  A & B \ar@{->>}[l]_{p} \ar@{ >->}[r]^{i} & A',
}  
\]
in which $p$ is an admissible epic and $i$ is an admissible
monic, where two
diagrams are considered equivalent if there is an isomorphism of
such diagrams inducing the identity on $A$ and $A'$. The composition
of two morphisms $(p,i)$, $(p',i')$ is given by the following 
construction: form the
pull-back over $p'$ and $i$ so that by
Proposition~\ref{prop:pb-adm-monic-adm-monic} there is a diagram
\[
\xymatrix@R=0.3pc{
  & & B'' \ar@{.>>}[dl]_{q} \ar@{ >.>}[dr]^{j'} \ar@{}[dd]|{\text{PB}} \\
  & B \ar@{->>}[dl]_{p} \ar@{ >->}[dr]^{i} & & 
  B' \ar@{->>}[dl]_{p'} \ar@{ >->}[dr]^{i'} 
  \\
  A & & A' & & A''
}
\]
and put $(p',i') \circ (p,i) = (pq, i'j')$. This is easily checked to
yield a category and it is not hard to make sense of the statement that the
morphisms $A \to A'$ in $Q\scrA$ correspond to the 
different ways that $A$ arises as an admissible
subquotient of $A'$.

Now any small category $\scrC$ gives rise to a simplicial set
$N\scrC$, called the \emph{nerve of $\scrC$}
whose $n$ simplices are given by sequences of composable
morphisms 
\[
C_{0} \to C_{1} \to \cdots \to C_{n},
\]
where the $i$-th face map is obtained by deleting the object $C_{i}$
and the $i$-th degeneracy map is obtained by replacing $C_{i}$ by
$1_{C_{i}}: C_{i} \to C_{i}$. The \emph{classifying space}
$B\scrC$ of $\scrC$ is the geometric realization of the nerve
$N\scrC$.

Quillen proves the fundamental result that
\[
K_{0}(\scrA,\scrE) \cong \pi_{1}(B(Q\scrA), 0)
\]
which motivates the definition
\[
K_{n}{(\scrA,\scrE)} := \pi_{n+1} (B(Q\scrA), 0).
\]
Obviously, an exact functor $F:(\scrA,\scrE) \to (\scrA',\scrE')$
yields a functor $Q\scrA \to Q\scrA'$ and hence a homomorphism
$F_{\ast}: K_{n}{(\scrA,\scrE)} \to K_{n}(\scrA',\scrE')$
which is easily seen to depend only on the isomorphism class of $F$. 

We do not discuss $K$-theory any further and recommend the lecture of
Quillen's original article~\cite{MR0338129} and Srinivas's
book~\cite{MR1382659} expanding on Quillen's article. For a good overview
over many topics of current interest we refer to the
handbook of $K$-theory~\cite{MR2182598}.

\appendix

\section{The Embedding Theorem}
\label{sec:the-embedding-theorem}

For abelian categories, one has the Freyd-Mitchell embedding theorem, 
see~\cite{MR0166240} and \cite{MR0167511}, allowing one to prove
diagram lemmas in abelian categories ``by chasing elements''. 
In order to prove diagram lemmas in exact categories, a similar
technique works. More precisely, one has:

\begin{Thm}[{\cite[A.7.1, A.7.16]{MR1106918}}]
  \label{thm:embedding-thm}
  Let $(\scrA, \scrE)$ be a small exact category.
  \begin{enumerate}[(i)]
    \item
      There is an abelian category $\scrB$ and a fully faithful
      exact functor $i: \scrA \to \scrB$ that reflects
      exactness. Moreover, $\scrA$ is closed under extensions in
      $\scrB$. 

    \item
      The category $\scrB$ may canonically be chosen to be the
      category of left exact functors $\scrA^{\opp} \to \Ab$ and $i$
      to be the Yoneda embedding $i(A) = \Hom_{\scrA}{({-},A)}$.

    \item 
      Assume moreover that $\scrA$ is weakly idempotent complete.
      If $f$ is a morphism in $\scrA$ and $i(f)$ is epic
      in $\scrB$ then $f$ is an admissible epic.
  \end{enumerate}
\end{Thm}

\begin{Rem}
  In order for~(iii) to hold it is necessary to assume weak idempotent
  completeness of $\scrA$. Indeed, if $\scrA$ fails to be weakly
  idempotent complete, there must be a retraction $r$ without
  kernel. By definition there exists $s$ such that $rs = 1$, but then
  $i(rs)$ is epic, so $i(r)$ is epic as well.
\end{Rem}

\begin{Rem}
  Let $\scrB$ be an abelian category and assume that $\scrA$ is a full
  subcategory which is closed under extensions, i.e., $\scrA$ is fully
  exact subcategory of $\scrB$
  in the sense of Definition~\ref{def:fully-exact-subcat}. Then, by 
  Lemma~\ref{lem:fully-exact-subcats-are-exact},
  $\scrA$ is an exact category with respect to the class $\scrE$ of
  short sequences in $\scrA$ which are exact in $\scrB$. This is a
  basic recognition principle of exact categories, for many examples
  arise in this way. The embedding
  theorem provides a partial converse to this recognition principle.
\end{Rem}

\begin{Rem}
  Quillen states in \cite[p.~``92/16/100'']{MR0338129}: 
  \begin{quote}
    Now suppose given an exact category $\scrM$. Let $\scrA$ be the
    additive category of additive contravariant functors from $\scrM$
    to abelian groups which are left exact, i.e. carry
    [an exact sequence
    $ M' \mono M \epi M''$]
    to an exact sequence
    \[
    0 \to F(M'') \to F(M) \to F(M').
    \]
    (Precisely, choose a universe containing $\scrM$, and let $\scrA$
    be the category of left exact functors whose values are abelian
    groups in the universe.) Following well-known ideas
    (e.g.~\cite{MR0232821}), one can prove $\scrA$ is an abelian
    category, that the Yoneda functor $h$ embeds $\scrM$ as a full
    subcategory of $\scrA$ closed under extensions, and finally that a
    [short] sequence [\ldots] 
    is in $\scrE$ if and only if $h$
    carries it into an exact sequence in $\scrA$. The details will be
    omitted, as they are not really important for the sequel.
  \end{quote}

  Freyd stated a similar theorem in \cite{MR0146234}, again without
  proof, and with the additional assumption of idempotents completeness,
  since he uses Heller's axioms. The first proof
  published is in Laumon~\cite[1.0.3]{MR726427}, relying on
  the Grothendieck-Verdier theory of sheaf\mbox{}ification
  \cite{MR0354652}. 
  The sheaf\mbox{}ification approach was also used and further
  refined by Thomason~\cite[Appendix~A]{MR1106918}.
  A quite detailed sketch of the proof alluded to by Quillen is given in
  Keller~\cite[A.3]{MR1052551}.
\end{Rem}

The proof given here is due to Thomason~\cite[A.7]{MR1106918} 
amalgamated with the proof in Laumon~\cite[1.0.3]{MR726427}. 
We also take the opportunity to fix a
slight gap in Thomason's argument (our Lemma~\ref{lem:dir-cat},
compare with the first sentence after \cite[(A.7.10)]{MR1106918}).
Since Thomason fails to spell out the nice
sheaf-theoretic interpretations of his construction and since
referring to SGA~4 seems rather brutal, we use the terminology of the more
lightweight Mac~Lane-Moerdijk~\cite[Chapter~III]{MR1300636}. 
Other good introductions to
the theory of sheaves may be found in Artin~\cite{artin-grothtop} or
Borceux~\cite{MR1315049}, for example.

\subsection{Separated Presheaves and Sheaves}

Let $(\scrA,\scrE)$ be a small exact category. For each object
$A \in \scrA$, let
\[
\scrC_{A} = \{(p': A' \epi A)\,: \,A' \in \scrA\}
\]
be the set of admissible epics onto $A$. The elements of $\scrC_{A}$
are the \emph{coverings} of $A$.

\begin{Lem}
  \label{lem:basis-topology}
  The family $\{\scrC_{A}\}_{A \in \scrA}$ is a
  \emph{basis for a Grothendieck topology} $J$ on
  $\scrA$:
  \begin{enumerate}[(i)]
    \item
      If $f: A \to B$ is an isomorphism then $f \in \scrC_{B}$.
      
    \item
      If $g: A \to B$ is arbitrary and $(q': B' \epi B) \in \scrC_{B}$
      then the pull-back
      \[
      \xymatrix{
        A' \ar[r] \ar@{->>}[d]_{p'} \ar@{}[dr]|{\text{\emph{PB}}} &
        B' \ar@{->>}[d]^{q'} \\
        A \ar[r]^{g} & B
      }
      \]
      yields a morphism $p' \in \scrC_{A}$. \emph{(``Stability under
        base-change'')} 

    \item
      If $(p: B \epi A) \in \scrC_{A}$ and $(q: C \epi B) \in \scrC_{B}$
      then $pq \in \scrC_{A}$. \emph{(``Transitivity'')}
  \end{enumerate}
  In particular, $(\scrA,J)$ is a \emph{site}.
\end{Lem}
\begin{proof}
  This is obvious from the definition, 
  see~\cite[Definition~2, p.~111]{MR1300636}.
\end{proof}

The Yoneda functor $y: \scrA \to \Ab^{\scrA^{\opp}}$ associates to
each object $A \in \scrA$ the presheaf (of abelian groups)
$y(A) = \Hom_{\scrA}{({-},A)}$. In general, a \emph{presheaf} is the
same thing as a functor
$G: \scrA^{\opp} \to \Ab$, which we will assume to
be additive except in the next lemma. We will see shortly that $y(A)$ is in
fact a \emph{sheaf} on the site $(\scrA,J)$.

\begin{Lem}
  \label{lem:separated-sheafs}
  Consider the site $(\scrA,J)$ and let
  $G: \scrA^{\opp} \to \Ab$ be a functor.
  \begin{enumerate}[(i)]
    \item
      The presheaf $G$ is \emph{separated} if and only if for each
      admissible epic $p$ the morphism $G(p)$ is monic.
 
    \item
      The presheaf $G$ is a \emph{sheaf} if and only if for each
      admissible epic $p: A \epi B$ the diagram
      \[
      \xymatrix{
        G(B) \ar[r]^-{G(p)} &
        G(A) \ar@<-1ex>[r]_-{d^{0} = G(p_{0})}
        \ar@<1ex>[r]^-{d^{1} = G(p_{1})} &
        G(A \times_{B} A)
      }
      \]
      is an \emph{equalizer} (difference kernel), where
      $p_{0}, p_{1} : A \times_{B} A \epi A$ denote the two
      projections.
      In other words, the presheaf $G$ is a sheaf if and only if for
      all admissible epics $p: A \epi B$ the diagram
      \[
      \xymatrix{
        G(B) \ar[r]^-{G(p)} \ar[d]^-{G(p)} & G(A) \ar[d]^-{d^{1}} \\
        G(A) \ar[r]^-{d^{0}} & G(A \times_{B} A)
      }
      \]
      is a pull-back.
  \end{enumerate}
\end{Lem}

\begin{proof}
  Again, this is obtained by making the definitions
  explicit. Point~(i) is the definition, \cite[p.~129]{MR1300636}, and
  point~(ii) is \cite[Proposition~1{[bis]}, p.~123]{MR1300636}.
\end{proof}

The following lemma shows that the sheaves on the site $(\scrA,J)$
are quite familiar gadgets. 

\begin{Lem}
  \label{lem:sheaf=left-exact}
  Let $G: \scrA^{\opp} \to \Ab$ be an additive functor. The following
  are equivalent:
  \begin{enumerate}[(i)]
    \item
      The presheaf $G$ is a sheaf on the site $(\scrA, J)$.
      
    \item
      For each admissible epic $p: B \epi C$ the sequence
      \[
      0 \xrightarrow{} G(C) \xrightarrow{G(p)} 
      G(B) \xrightarrow{d^{0}-d^{1}} G(B \times_{C} B)
      \]
      is exact.

    \item
      For each short exact sequence $A \mono B \epi C$ in $\scrA$ 
      the sequence 
      \[
      0 \xrightarrow{} G(C) \xrightarrow{} G(B) \xrightarrow{} G(A)
      \]
      is exact, i.e., $G$ is \emph{left exact}.
  \end{enumerate}
\end{Lem}

\begin{proof}
  By Lemma~\ref{lem:separated-sheafs}~(ii) we have that $G$ is a sheaf if
  and only if the sequence
  \[
  0 \xrightarrow{} G(C) \xrightarrow{\mat{G(p) \\ G(p)}}
  G(B) \oplus G(B) \xrightarrow{\mat{G(p_{0}) & & -G(p_{1})}} 
  G(B \times_{C} B)
  \]
  is exact. Since $p_{1}: B \times_{C} B \epi B$ is a split epic with
  kernel $A$, there is an isomorphism
  $B \times_{C} B \to A \oplus B$ and it is easy to check that
  the above sequence is isomorphic to 
  \[
  0 \xrightarrow{} G(C) \xrightarrow{} G(B) \oplus G(B) \xrightarrow{}
  G(A) \oplus G(B).
  \]
  Because left exact sequences are stable under taking direct sums
  and passing to direct summands, (i) is equivalent to (iii).
  That~(i) is equivalent to~(ii) is obvious by
  Lemma~\ref{lem:separated-sheafs}~(ii).
\end{proof}

\begin{Cor}[{\cite[A.7.6]{MR1106918}}]
  \label{cor:represented-functor-sheaf}
  The represented functor $y(A) = \Hom_{\scrA}{({-},A)}$ is a sheaf
  for every object $A$ of $\scrA$.
\end{Cor}
\begin{proof}
  Given an exact sequence $B' \mono B \epi B''$ we need to prove that 
  \[
  0 \xrightarrow{} \Hom_{\scrA}{(B'', A)} \xrightarrow{}
  \Hom_{\scrA}{(B, A)} \xrightarrow{} \Hom_{\scrA}{(B', A)}
  \]
  is exact. That the sequence is exact at $\Hom_{\scrA}{(B,A)}$
  follows from the fact that $B \epi B''$ is a cokernel of
  $B' \mono B$. That the sequence is exact at $\Hom_{\scrA}{(B'', A)}$ 
  follows from the fact that $B \epi B''$ is epic.
\end{proof}

\subsection{Outline of the Proof}
Let now $\scrY$ be the category of additive functors
$\scrA^{\opp} \to \Ab$ and let $\scrB$ be the category of (additive)
sheaves on the site $(\scrA,J)$. Let $j_{\ast}: \scrB \to \scrA$ be
the inclusion. By
Corollary~\ref{cor:represented-functor-sheaf}, the Yoneda functor $y$
factors as
\[
\xymatrix{
  \scrA \ar[dr]_{y} \ar[r]^{i} & \scrB \ar[d]^{j_{\ast}} \\
  & \scrY
}
\]
via a functor $i : \scrA \to \scrB$. We will prove that the 
category $\scrB = \Sheaves{(A,J)}$ is abelian and
we will check that the functor $i$ has the
properties asserted in the embedding theorem.

The category $\scrY$ is a Grothendieck abelian category (there is a
generator, small products and coproducts exist and
filtered colimits are exact)---as a
functor category, these properties are inherited from $\Ab$, as limits
and colimits are taken pointwise. The crux of the proof of the
embedding theorem is to show
that $j_{\ast}$ has a left adjoint $j^{\ast}$ such that
$j^{\ast}j_{\ast} = \id_{\scrB}$, namely sheaf\mbox{}ification. As
soon as this is established, the rest will be relatively painless.

\subsection{Sheaf\mbox{}ification}

The goal of this section is to construct the sheaf\mbox{}ification
functor on the site $(\scrA,J)$ and to prove its basic properties. We
will construct an endo\-functor $L: \scrY \to \scrY$ which associates to
each presheaf a separated presheaf and to each separated presheaf a
sheaf. The sheaf\mbox{}ification functor will then be given by
$j^{\ast} = LL$.

We need one more concept from the theory of sites:

\begin{Lem}
  Let $A \in \scrA$. A covering $p'': A'' \epi A$ is a
  \emph{refinement} of the covering $p': A' \epi A$ if
  and only if there exists a morphism $a: A'' \to A'$ such that
  $p' a = p''$.
\end{Lem}

\begin{proof}
  This is the specialization of a
  \emph{matching family} as given in \cite[p.~121]{MR1300636} 
  in the present context.
\end{proof}

By definition, refinement gives the structure of a filtered category
on $\scrC_{A}$ for each $A \in \scrA$. More precisely, let $\scrD_{A}$
be the following category: the objects are the coverings $(p': A' \epi A)$
and there exists at most one morphism between any two objects of
$\scrD_{A}$: there exists a morphism
$(p': A' \epi A) \to (p'': A'' \epi A)$ in $\scrD_{A}$ if and only if
there exists $a: A'' \to A'$ such that $p'a = p''$. To see that
$\scrD_{A}$ is filtered, let $(p':A' \epi A)$ and $(p'': A'' \epi A)$
be two objects and put $A''' = A' \times_{A} A''$, so there
is a pull-back diagram
\[
\xymatrix{
  A''' \ar@{->>}[d]^{a} \ar@{->>}[r]^{a'} \ar@{}[dr]|{\text{PB}} &
  A'' \ar@{->>}[d]^{p''} \\
  A' \ar@{->>}[r]^{p'} & A.
}
\]
Put $p''' = p'a = p''a'$, so the object
$(p''' : A''' \epi A)$ of $\scrD_{A}$ is
a common refinement of $(p':A' \epi A)$ and $(p'':A'' \epi A)$.

\begin{Lem}
  \label{lem:dir-cat}
  Let $A_{1}, A_{2} \in \scrA$ be any two objects.
  \begin{enumerate}[(i)]
    \item
      There is a functor
      $Q: \scrD_{A_{1}} \times \scrD_{A_{2}} \to
      \scrD_{A_{1} \oplus A_{2}},\;
      (p_{1}', p_{2}') \mapsto (p_{1}' \oplus p_{2}')$.

    \item
      Let $(p':A' \epi A_{1} \oplus A_{2})$
      be an object of $\scrD_{A_{1} \oplus A_{2}}$ and for $i=1,2$ let 
      \[
      \xymatrix{
        A_{i}' \ar[r] \ar@{}[dr]|{\text{\emph{PB}}}
        \ar@{->>}[d]^{p_{i}'} &
        A' \ar@{->>}[d]^{p'} \\
        A_{i} \ar[r] & A_{1} \oplus A_{2}
      }
      \]
      be a pull-back diagram in which the bottom arrow is the
      inclusion. This construction defines a functor
      \[
      P: \scrD_{A_{1} \oplus A_{2}} \longrightarrow
      \scrD_{A_{1}} \times \scrD_{A_{2}},
      \quad p' \longmapsto (p_{1}', p_{2}').
      \]
      
    \item
      There are a natural transformation
      $\id_{\scrD_{A_{1} \oplus A_{2}}} \Rightarrow PQ$
      and a natural isomorphism 
      $QP \cong \id_{\scrD_{A_{1}} \times \scrD_{A_{2}}}$. In particular,
      the images of $P$ and $Q$ are cofinal.
  \end{enumerate}
\end{Lem}
\begin{proof}
  That $P$ is a functor follows from its construction and the
  universal property of pull-back diagrams in conjunction with
  axiom~[E2$^{\opp}$].
  That $Q$ is well-defined follows from Proposition~\ref{prop:sum-exact} and
  that $PQ \cong \id_{\scrD_{A_{1}} \times \scrD_{A_{2}}}$ is easy to check.
  That there is a natural transformation
  $\id_{\scrD_{A_{1} \oplus A_{2}}} \Rightarrow QP$ follows from the
  universal property of products.
\end{proof}

Let $(p'': A'' \epi A)$ be a
refinement of $(p': A' \epi A)$, and let $a : A'' \to A'$ be such that 
$p'a = p''$. By the universal property of pull-backs, $a$ yields a
unique morphism $A'' \times_{A} A'' \to A' \times_{A} A'$ which we
denote by $a \times_{A} a$. Hence, for every
additive functor $G: \scrA^{\opp} \to \Ab$, we
obtain a commutative diagram in~$\Ab$:
\[
\xymatrix{
  \Ker{(d^{0} - d^{1})} \ar@{.>}[d]^{\exists!} \ar[r] &
  G(A') \ar[r]^-{d^{0} - d^{1}} \ar[d]^{G{(a)}} & 
  G(A' \times_{A} A') \ar[d]^{G(a \times_A a)} \\
  \Ker{(d^{0} - d^{1})} \ar[r] &
  G(A'') \ar[r]^-{d^{0} - d^{1}} & 
  G(A'' \times_{A} A'').
}
\]
The next thing to observe is that the dotted morphism does not depend
on the choice of $a$. Indeed, if $\tilde{a}$ is another morphism such that
$p' \tilde{a} = p''$, consider the diagram
\[
\xymatrix{
  A''  \ar@/_1.2pc/[ddr]_-{a} \ar@/^1.2pc/[drr]^-{\tilde{a}} 
  \ar@{.>}[dr]^{\exists! b} \\
  & A' \times_{A} A' \ar[r]^-{p_{1}'} \ar[d]^-{p_{0}'}
  \ar@{}[dr]|{\text{PB}} & A' \ar[d]^{p'}
  \\
  & A' \ar[r]^{p'} & A   
}
\]
and $b: A'' \to A' \times_{A} A'$ is such that
\[
G(b) (d^{0} - d^{1}) = G(b) G(p_0') - G(b) G(p_1') = G(a) - G(\tilde{a}),
\]
so 
$G(a) - G(\tilde{a}) = 0$ on $\Ker{(d^{0} - d^{1})}$.

For $G: \scrA^{\opp} \to \Ab$, we
put $\ell G(p': A' \epi A) :=
\Ker{(G(A') \xrightarrow{d^{0} - d^{1}} G(A' \times_{A} A'))}$ and we
have just seen that this
defines a functor $\ell G: \scrD_{A} \to \Ab$.

\begin{Lem}
  \label{lem:functoriality-L}
  Define
  \[
  LG(A) = \varinjlim_{\scrD_{A}} \ell G(p':A' \epi A).
  \]
  \begin{enumerate}[(i)]
    \item
     $LG$ is an additive contravariant functor in $A$.
    \item
     $L$ is a covariant functor in $G$.
  \end{enumerate}
\end{Lem}
\begin{proof}
  This is immediate from going through the definitions:
  
  To prove~(i), let $f: A \to B$ be an arbitrary morphism. 
  Lemma~\ref{lem:basis-topology}~(ii) shows that by taking
  pull-backs we obtain a functor
  \[
  \scrD_{B} \xrightarrow{f^{\ast}} \scrD_{A}
  \]
  which, by passing to the colimit, induces a unique
  morphism $LG(B) \xrightarrow{LG(f)} LG(A)$
  compatible with $f^{\ast}$. 
  From this uniqueness, we deduce $LG(fg) = LG(g) LG(f)$. The
  additivity of $LG$ is a consequence of Lemma~\ref{lem:dir-cat}.
  
  To prove~(ii), let $\alpha: F \Rightarrow G$ be a natural
  transformation between two (additive) presheaves. Given an object
  $A \in \scrA$, we obtain a morphism between the colimit
  diagrams defining $LF(A)$ and $LG(A)$ and we denote the unique resulting 
  map by $L(\alpha)_{A}$. Given a morphism
  $f: A \to B$, there is a commutative diagram
  \[
  \xymatrix{
    LF(B) \ar[d]^{L(\alpha)_{B}} \ar[r]^{LF(f)} &
    LF(A) \ar[d]^{L(\alpha)_{A}} \\
    LG(B) \ar[r]^{LG(f)} & LG(A),
  }
  \]
  as is easily checked. The uniqueness in the definition of
  $L(\alpha)_{A}$ implies that for each $A \in \scrA$
  the equation
  \[
  L(\alpha \circ \beta)_{A} =
  L(\alpha)_{A} \circ L(\beta)_{A}
  \]
  holds. The reader in need of more
  details may consult~\cite[p.~206f]{MR1315049}.
\end{proof}

\begin{Lem}[{\cite[A.7.8]{MR1106918}}]
  \label{lem:L-additive-finite-limits}
  The functor $L: \scrY \to \scrY$ has the following properties:
  \begin{enumerate}[(i)]
    \item
      It is additive and preserves finite limits.

    \item
      There is a natural transformation
      $\eta: \id_{\scrY} \Rightarrow L$.
  \end{enumerate}
\end{Lem}
\begin{proof} 
  That $L$ preserves finite limits follows from the fact that filtered colimits
  and kernels in $\Ab$ commute with finite limits, as limits
  in $\scrY$ are formed pointwise, see also
  \cite[Lemma~3.3.1]{MR1315049}. Since $L$ preserves finite
  limits, it preserves in particular finite products, hence it is
  additive. This settles point~(i).

  For each $(p':A' \epi A) \in \scrD_{A}$ the morphism 
  $G(p'): G(A) \to G(A')$ factors uniquely over
  \[
  \tilde{\eta}_{p'}:G(A) \to \Ker{(G(A') \to G(A' \times_{A} A'))}.
  \]
  By passing to the colimit over $\scrD_{A}$, this induces a morphism
  $\tilde{\eta}_{A}: G(A) \to LG(A)$
  which is clearly natural in $A$. In other words, the $\tilde{\eta}_{A}$
  yield a natural transformation 
  $\eta_{G}: G \Rightarrow LG$, i.e., a morphism in $\scrY$. We
  leave it to the reader to check that the 
  construction of $\eta_{G}$ is compatible with natural
  transformations $\alpha: G \Rightarrow F$ so that the
  $\eta_{G}$ assemble to yield a
  natural transformation $\eta: \id_{\scrY} \Rightarrow L$,
  as claimed in point~(ii).
\end{proof}


\begin{Lem}[{\cite[A.7.11, (a), (b), (c)]{MR1106918}}]
  \label{lem:technical-lemma-elements}
  Let $G \in \scrY$ and let $A \in \scrA$.
  \begin{enumerate}[(i)]
    \item
      For all $x \in LG(A)$ there exists an admissible epic
      $p' : A' \epi A$ and $y \in G(A')$ such that
      $\eta(y) = LG(p')(x)$ in $LG(A')$.

    \item
      For all $x \in G(A)$, we have $\eta(x) = 0$ in $LG(A)$ if and
      only if there exists an admissible epic $p': A' \epi A$ such
      that $G(p')(x) = 0$ in $G(A')$.

    \item
      We have $LG = 0$ if and only if for all $A \in \scrA$ and all
      $x \in G(A)$ there exists an admissible epic $p': A' \epi A$
      such that $G(p')(x) = 0$.
 
  \end{enumerate}
\end{Lem}
\begin{proof}
  Points~(i) and~(ii) are immediate from the definitions. Point~(iii)
  follows from~(i) and~(ii).
\end{proof}

\begin{Lem}[{\cite[Lemma~2, p.~131]{MR1300636}, %
    \cite[A.7.11, (d), (e)]{MR1106918}}]
  \label{lem:separated-sheaves-via-eta}
  Let $G \in \scrY$.
  \begin{enumerate}[(i)]
    \item
      The presheaf $G$ is separated if and only if 
      $\eta_{G}: G \to LG$ is monic.
 
    \item
      The presheaf $G$ is a sheaf if and only if $\eta_{G}: G \to LG$
      is an isomorphism.
  \end{enumerate}
\end{Lem}
\begin{proof}
  Point~(i) follows from Lemma~\ref{lem:technical-lemma-elements}~(ii) and
  point~(ii) follows from the definitions.
\end{proof}

\begin{Prop}[{\cite[A.7.12]{MR1106918}}]
  \label{prop:LG-separated-LLG-sheaf}
  Let $G \in \scrY$.
  \begin{enumerate}[(i)]
    \item
      The presheaf $LG$ is separated.

    \item
      If $G$ is separated then $LG$ is a sheaf.
  \end{enumerate}
\end{Prop}

\begin{proof}
  Let us prove~(i) by applying
  Lemma~\ref{lem:separated-sheafs}~(i), so
  let $x \in LG(A)$ and let $b:B \epi A$ be an admissible
  epic for which $LG(b)(x) = 0$. We have to prove that then $x = 0$
  in $LG(A)$. By the definition of $LG(A)$, we know that $x$ is
  represented by some
  $y \in \Ker{(G(A') \xrightarrow{d^{0} - d^{1}} G(A'\times_{A} A'))}$ 
  for some admissible epic $(p':A' \epi A)$ in $\scrD_{A}$.
  Since $LG(b)(x) = 0$ in $LG(B)$, we know that the image of $y$ 
  in 
  \[
  \Ker{(G(A' \times_{A} B) \xrightarrow{d^{0} - d^{1}} G((A'
    \times_{A} B) \times_{B} (A' \times_{A} B)))}
  \]
  is equivalent to zero in the filtered colimit over $\scrD_{B}$
  defining $LG(B)$. Therefore there exists a morphism 
  $D \to A' \times_{A} B$ in $\scrA$ such that its composite with the
  projection onto $B$ is an admissible epic $D \epi B$. By
  Lemma~\ref{lem:technical-lemma-elements}~(ii), it follows that $y$
  maps to zero in $G(D)$. Now the composite $D \epi B \epi A$ is in
  $\scrD_{A}$ and hence $y$ is equivalent to zero in the filtered colimit
  over $\scrD_{A}$ defining $LG(A)$. Thus, $x = 0$ in $LG(A)$ as
  required.

  Let us prove~(ii). 
  If $G$ is a separated presheaf, we have to check
  that for every admissible epic $B \epi A$ the diagram
  \[
  \xymatrix{
    LG(A) \ar[r] &
    LG(B) \ar@<-1ex>[r]_>>>>>{d^{0} = G(p_{0})}
        \ar@<1ex>[r]^>>>>>{d^{1} = G(p_{1})} &
    LG(B \times_{A} B)
  }
  \]
  is a difference kernel. By~(i) $LG$ is separated, so
  $LG(A) \to LG(B)$ is monic, and it remains to prove that
  every element $x \in LG(B)$ with $(d^{0} - d^{1})x = 0$ is in the
  image of $LG(A)$. By Lemma~\ref{lem:technical-lemma-elements}~(i)
  there is an admissible epic $q:C \epi B$ and $y \in G(C)$ such
  that $\eta(y) = LG(q)(x)$. It follows that
  $\eta G(p_{0}) (y) = \eta G(p_{1})(y)$ in $LG(C \times_{A} C)$. Now,
  $G$ is separated, so $\eta: G \Rightarrow LG$ is
  monic by Lemma~\ref{lem:separated-sheaves-via-eta},
  and we conclude from this that
  $G(p_{0})(y) = G(p_{1})(y)$ in $G(C \times_{A} C)$. In
  other words,
  $y \in \Ker{(G(C) \xrightarrow{d^{0} - d^{1}} G(C \times_{A} C))}$
  yields a class in $LG(A)$ representing $x$.
\end{proof}

\begin{Cor}
  \label{cor:LG=0-iff-LLG=0}
  For a presheaf $G \in \scrY$ we have $LG = 0$ if and only if
  $LLG = 0$.
\end{Cor}
\begin{proof}
  Obviously $LG = 0$ entails $LLG = 0$ as $L$ is additive by 
  Lemma~\ref{lem:L-additive-finite-limits}. Conversely,
  as $LG$ is separated by
  Proposition~\ref{prop:LG-separated-LLG-sheaf}, it follows that
  the morphism $\eta_{LG}: LG \to LLG$ is monic by
  Lemma~\ref{lem:separated-sheaves-via-eta}~(i), so if $LLG = 0$ we
  must have $LG = 0$.
\end{proof}

\begin{Def}
  The \emph{sheaf\mbox{}ification functor} is $j^{\ast} = LL: \scrY \to \scrB$.
\end{Def}

\begin{Lem}
  The sheaf\mbox{}ification functor $j^{\ast}: \scrY \to \scrB$
  is left adjoint to the
  inclusion functor $j_{\ast}: \scrB \to \scrY$ and
  satisfies $j^{\ast}j_{\ast} \cong \id_{\scrB}$. Moreover,
  sheaf\mbox{}ification is exact.
\end{Lem}
\begin{proof}
  By Lemma~\ref{lem:separated-sheaves-via-eta}~(ii)
  the morphism $\eta_{G}: G \to LG$ is an isomorphism if
  and only if $G$ is a sheaf, so
  it follows that $j^{\ast} j_{\ast} \cong \id_{\scrB}$.
  
  Let $Y \in \scrY$ be a presheaf and let $B \in \scrB$ be a sheaf.
  The natural transformation $\eta: \id_{\scrY} \Rightarrow L$ gives 
  us on the one hand a natural transformation
  \[
  \varrho_{Y} = \eta_{LY} \eta_{Y} : Y \longrightarrow LLY
  = j_{\ast}j^{\ast}Y
  \]
  and on the other hand a natural isomorphism
  \[
  \lambda_{B} = (\eta_{LB} \eta_{B})^{-1} : 
  j^{\ast}j_{\ast}B = LLB \longrightarrow B.
  \]
  Now the compositions
  \[
  j_{\ast}B \xrightarrow{\varrho_{j_{\ast}B}}
  j_{\ast}j^{\ast}j_{\ast}B \xrightarrow{j_{\ast}\lambda_{B}}
  j_{\ast}B \qquad \text{and} \qquad
  j^{\ast}Y \xrightarrow{j^{\ast}\varrho_{Y}}
  j^{\ast}j_{\ast}j^{\ast} Y \xrightarrow{\lambda_{j^{\ast}Y}} 
  j^{\ast}Y
  \]
  are manifestly equal to $\id_{j_{\ast}B}$ and $\id_{j^{\ast}Y}$ so
  that $j^{\ast}$ is indeed left adjoint to $j_{\ast}$. In
  particular $j^{\ast}$ preserves cokernels.
  That $j^{\ast}$ preserves kernels follows from the
  fact that $L: \scrY \to \scrY$ has this property by
  Lemma~\ref{lem:L-additive-finite-limits}~(i) and the fact that
  $\scrB$ is a full subcategory of $\scrY$. Therefore $j^{\ast}$ is exact.
\end{proof}

\begin{Rem}
  It is an illuminating exercise to prove
  exactness of $j^{\ast}$ directly by going through the definitions.
\end{Rem}

\begin{Lem}
  The category $\scrB$ is abelian.
\end{Lem}
\begin{proof}
  It is clear that $\scrB$ is additive.
  The sheaf\mbox{}ification functor $j^{\ast} = LL$ preserves kernels by 
  Lemma~\ref{lem:L-additive-finite-limits}~(i) and as a left adjoint it
  preserves cokernels. To prove $\scrB$
  abelian, it suffices to check that every morphism $f: A \to B$ has an
  analysis
  \[
  \xymatrix@R=0.5pc{
    & A \ar@{->>}[dr] \ar[rrr]^{f} & & & B \ar@{->>}[dr] &
    \\
    \Ker{(f)} \ar@{ >->}[ur] & & \Coim{(f)} \ar[r]^{\cong} & \Im{(f)}
    \ar@{ >->}[ur] & & \Coker{(f)}.
  }
  \]
  Since $j^{\ast}$ preserves kernels and cokernels and
  $j^{\ast}j_{\ast} \cong \id_{\scrB}$ such an analysis can be obtained
  by applying $j^{\ast}$ to an analysis of $j_{\ast}f$ in $\scrY$.
\end{proof}

\subsection{Proof of the Embedding Theorem}

Let us recapitulate: one half of the axioms of an exact structure yields
that a small exact category $\scrA$ becomes a \emph{site}
$(\scrA, J)$. We denoted the Yoneda category of contravariant functors
$\scrA \to \Ab$ by $\scrY$ and the Yoneda embedding
$A \mapsto \Hom{({-},A)}$ by $y: \scrA \to \scrY$. We have shown that
the category $\scrB$ of sheaves on the site $(A,J)$ is abelian, being
a full reflective subcategory of $\scrY$ with sheaf\mbox{}ification
$j^{\ast}: \scrY \to \scrB$ as reflector (left adjoint). 
Following Thomason, we denoted the inclusion $\scrB \to \scrY$ by
$j_{\ast}$. Moreover, we have shown that the Yoneda embedding takes
its image in $\scrB$, so we obtained a commutative diagram of
categories
\[
\xymatrix{
  \scrA \ar[r]^-{i} \ar[dr]_-{y} & \scrB \ar[d]^{j_{\ast}} \\
  & \scrY,
}
\]
in other words $y = j_{\ast}i$. By the Yoneda lemma, $y$ is
fully faithful and $j_{\ast}$ is fully faithful, hence $i$ is fully
faithful as well. This settles the first part of the following lemma:

\begin{Lem}
  \label{lem:embedding-exact}
  The functor $i:\scrA \to \scrB$ is fully faithful and exact.
\end{Lem}
\begin{proof}
  By the above discussion, it remains to prove exactness.
  
  Clearly, the Yoneda embedding sends exact sequences in $\scrA$ to
  left exact sequences in $\scrY$. Sheaf\mbox{}ification
  $j^{\ast}$ is exact and since $j^{\ast} j_{\ast} \cong \id_{\scrB}$, we have
  that $j^{\ast} y = j^{\ast} j_{\ast} i \cong i$ is left exact as
  well. It remains to prove that for each
  admissible epic $p: B \epi C$ the morphism $i(p)$ is epic. 
  By Corollary~\ref{cor:LG=0-iff-LLG=0}, it suffices to prove that
  $G = \Coker{y(p)}$ satisfies $LG = 0$, because 
  $\Coker{i(p)} = j^{\ast}\Coker{y(p)} = LLG = 0$ then implies that
  $i(p)$ is epic. To this end we use the criterion in
  Lemma~\ref{lem:technical-lemma-elements}~(iii), so let
  $A \in \scrA$ be any object and $x \in G(A)$. We have an exact
  sequence 
  $\Hom{(A,B)} \xrightarrow{y(p)_{A}} \Hom{(A,C)} \xrightarrow{q_{A}} G(A)
  \xrightarrow{} 0$,
  so $x = q_{A}(f)$ for some morphism $f: A \to C$. Now form the pull-back
  \[
  \xymatrix{
    A' \ar@{->>}[r]^{p'} \ar@{}[dr]|{\text{PB}} \ar[d]^{f'} &
    A \ar[d]^{f} \\
    B \ar@{->>}[r]^{p} & C
  }
  \]
  and observe that $G(p')(x) = G(p')(q_{A}(f)) = q_{A'}(fp') =
  q_{A'}(pf') = 0$.
\end{proof}

\begin{Lem}[{\cite[A.7.15]{MR1106918}}]
  \label{lem:epics-onto-representables-compose-to-adm-epics}
  Let $A \in \scrA$ and $B \in \scrB$ and suppose there is an epic
  $e: B \epi i(A)$. There exist $A' \in \scrA$ and $k: i(A') \to B$
  such that $ek: A' \to A$ is an admissible epic.
\end{Lem}

\begin{proof}
  Let $G$ be the cokernel of $j_{\ast}e$ in $\scrY$. 
  Then we have $0 = j^{\ast}G = LLG$ because
  $j^{\ast}j_{\ast}e \cong e$ is
  epic. By Corollary~\ref{cor:LG=0-iff-LLG=0} it follows that $LG = 0$
  as well. Now observe that
  $G(A) \cong \Hom{(A,A)}/\Hom{(i(A), B)}$ and let 
  $x \in G(A)$ be the class of $1_{A}$. From 
  Lemma~\ref{lem:technical-lemma-elements}~(iii) we conclude 
  that there is an
  admissible epic $p':A' \epi A$ such that $G(p')(x) = 0$ in 
  $G(A') \cong \Hom{(A',A)} / \Hom{(i(A'),B)}$. But this means that
  the admissible epic $p'$ factors as $ek$ for some $k \in
  \Hom{(i(A'),B)}$ as claimed.
\end{proof}

\begin{Lem}
  \label{lem:embedding-reflects-exactness}
  The functor $i$ reflects exactness.
\end{Lem}

\begin{proof}
  Suppose $A \xrightarrow{m} B \xrightarrow{e} C$ is a sequence in
  $\scrA$ such that 
  $i(A) \xrightarrow{i(m)} i(B) \xrightarrow{i(e)} i(C)$
  is short exact in $\scrB$. In particular, $i(m)$ is a kernel of
  $i(e)$. Since $i$ is fully faithful, it follows that $m$
  is a kernel of $e$ in $\scrA$, hence we are done as soon as we can
  show that $e$ is an admissible epic. Because $i(e)$ is epic, 
  Lemma~\ref{lem:epics-onto-representables-compose-to-adm-epics}
  allows us to find $A' \in \scrA$
  and $k: i(A') \to i(B)$ such that $ek$ is an admissible epic and
  since $e$ has a kernel we conclude by the dual of
  Proposition~\ref{prop:obscure-axiom}.
\end{proof}

\begin{Lem}
  \label{lem:embedding-closed-under-extensions}
  The essential image of $i:\scrA \to \scrB$ is closed
  under extensions.
\end{Lem}
\begin{proof}
  Consider a short exact sequence $i(A) \mono G \epi i(B)$ in $\scrB$,
  where $A,B \in \scrA$. By
  Lemma~\ref{lem:epics-onto-representables-compose-to-adm-epics} we
  find an admissible epic $p:C \epi B$ such that $i(p)$ factors
  over $G$. Now consider the pull-back diagram
  \[
  \xymatrix{
    D \ar@{->>}[d] \ar@{->>}[r] \ar@{}[dr]|{\text{PB}} & G \ar@{->>}[d] \\
    i(C) \ar@{->>}[r]^-{i(p)} & i(B)
  }
  \]
  and observe that $D \epi i(C)$ is a split epic because $i(p)$
  factors over $G$. Therefore we have isomorphisms
  $D \cong i(A) \oplus i(C) \cong i(A \oplus C)$. If $K$ is a kernel
  of $p$ then $i(K)$ is a kernel of $D \epi G$, so we
  obtain an exact sequence
  \[
  \xymatrix{
    i(K) \ar@{ >->}[r]^-{\mat{i(a) \\ i(c)}} &
    i(A) \oplus i(C)  \ar@{->>}[r] & G,
  }
  \]
  where $c = \ker{p}$, which shows that $G$ is the push-out
  \[
  \xymatrix{
    i(K) \ar@{ >->}[r]^-{i(c)} \ar[d]_-{i(a)} \ar@{}[dr]|{\text{PO}}&
    i(C) \ar[d] \\
    i(A) \ar@{ >->}[r] & G.
  }
  \]
  Now $i$ is exact by Lemma~\ref{lem:embedding-exact} and hence
  preserves push-outs along admissible monics by 
  Proposition~\ref{prop:exact-functors-push-out-pull-back}, so $i$
  preserves the push-out $G' = A \cup_{K} C$ of $a$ along the
  admissible monic $c$ and
  thus $G$ is isomorphic to $i(G')$.
\end{proof}

\begin{proof}[Proof of the Embedding Theorem~\ref{thm:embedding-thm}]
  Let us summarize what we know:
  the embedding $i:\scrA \to \scrB$ is fully faithful and exact by
  Lemma~\ref{lem:embedding-exact}. It reflects exactness by
  Lemma~\ref{lem:embedding-reflects-exactness} and its image is closed
  under extensions in $\scrB$ by
  Lemma~\ref{lem:embedding-closed-under-extensions}. This settles
  point~(i) of the theorem.

  Point~(ii) is taken care of by Lemma~\ref{lem:sheaf=left-exact} and
  Corollary~\ref{cor:represented-functor-sheaf}. 

  It remains to prove~(iii). Assume that $\scrA$ is weakly idempotent 
  complete. We claim that every
  morphism $f: B \to C$ such that $i(f)$ is epic is in fact an
  admissible epic. Indeed, by
  Lemma~\ref{lem:epics-onto-representables-compose-to-adm-epics} we
  find a morphism $k: A \to B$ such that $fk: A \epi C$ is an
  admissible epic and we conclude by
  Proposition~\ref{prop:weakly-split-obscure-axiom}.
\end{proof}

\section{Heller's Axioms}
\label{sec:hellers-axioms}

\begin{Prop}[Quillen]
  Let $\scrA$ be an additive category and let
  $\scrE$ be a class of kernel-cokernel pairs in $\scrA$. The pair
  $(\scrA, \scrE)$ is a weakly idempotent complete exact category
  if and only if $\scrE$ satisfies Heller's axioms:
  \begin{enumerate}[(i)]
    \item
      Identity morphisms are both admissible monics and
      admissible epics;
      
    \item
      The class of admissible monics and the class of admissible epics
      are closed under composition;

    \item 
      Let $f$ and $g$ be composable morphisms.
      If $gf$ is an admissible monic then so is $f$ and 
      if $gf$ is an admissible epic then so is $g$;

    \item 
      Assume that all rows and the second two columns of the
      commutative diagram
      \[
      \xymatrix{
        A' \ar@{ >->}[r]^{f'} \ar[d]^{a} &
        B' \ar@{->>}[r]^{g'} \ar@{ >->}[d]^{b} &
        C' \ar@{ >->}[d]^{c} \\
        A \ar@{ >->}[r]^{f} \ar[d]^{a'} &
        B \ar@{->>}[r]^{g} \ar@{->>}[d]^{b'} &
        C \ar@{->>}[d]^{c'} \\
        A'' \ar@{ >->}[r]^{f''} &
        B'' \ar@{->>}[r]^{g''}  &
        C''
      }
      \]
      are in $\scrE$ then the first column is also in $\scrE$.
  \end{enumerate}
\end{Prop}

\begin{proof}
  Note that~(i) and~(ii) are just axioms [E$0$], [E$1$] and their
  duals.
  
  For a weakly idempotent complete exact category
  $(\scrA,\scrE)$, point~(iii) is proved in
  Proposition~\ref{prop:weakly-split-obscure-axiom} and point~(iv) 
  follows from the $3 \times 3$-lemma~\ref{cor:3x3-lemma}.
  
  Conversely, assume that $\scrE$ has properties~(i)--(iv) and let us
  check that $\scrE$ is an exact structure.
  
  By properties~(i) and~(iii) an isomorphism is both an admissible
  monic and an admissible epic since by definition $f^{-1}f = 1$ and
  $ff^{-1} = 1$. If the short sequence $\sigma = (A' \to A \to A'')$ is
  isomorphic to the short exact sequence $B' \mono B \epi B''$ then
  property~(iv) tells us that $\sigma$ is short exact. Thus, $\scrE$
  is closed under isomorphisms. 

  Heller proves \cite[Proposition~4.1]{MR0100622} that~(iv) implies
  its dual, that is: if the commutative diagram in~(iv) has exact rows
  and both $(a,a')$ and $(b,b')$ belong to $\scrE$ then
  so does $(c,c')$.\footnote{Indeed, by~(iii) $c'$ is an admissible epic and so
    it has a kernel $D$. Because $c'gb = 0$, there is a morphism 
    $B' \to D$ and replacing $C'$ by $D$ in the diagram of~(iv) 
    we see that $A' \mono B' \epi D$ is short exact. 
    Therefore $C' \cong D$ and we
    conclude by the fact that $\scrE$ is closed under isomorphisms.} It
  follows that Heller's axioms are self-dual.
  
  Let us prove that [E$2$] holds---the remaining axiom~[E$2^{\opp}$]
  will follow from the dual argument. Given the diagram
  \[
  \xymatrix{
    A' \ar@{ >->}[r]^{f'} \ar[d]^{a} & B' \\
    A
  }
  \]
  we want to construct its push-out $B$ and prove that the morphism $A
  \to B$ is an admissible monic. 
  Observe that $\mat{a \\ f'}: A' \to A \oplus B'$ is the composition 
  \[
  \xymatrix{
    A' \ar@{ >->}[r]^-{\mat{0 \\ 1}} &
    A \oplus A' \ar[r]_-{\cong}^-{\mat{1 & a \\ 0 & 1}} &
    A \oplus A' \ar@{ >->}[r]^-{\mat{1 & 0 \\ 0 & f'}} &
    A \oplus B'.
  }
  \]
  By~(iii) split exact sequences belong to $\scrE$, and the proof of
  Proposition~\ref{prop:sum-exact} shows that the direct sum of
  two sequences in $\scrE$ also belongs to $\scrE$. Therefore
  $\mat{a \\ f'}$ is an admissible monic and it has a cokernel
  $\mat{-f & b}: A \oplus B' \epi B$. It follows that the
  left hand square in the diagram
  \[
  \xymatrix{
    A' \ar@{ >->}[r]^-{f'} \ar[d]^-{a} \ar@{}[dr]|{\text{BC}} & 
    B' \ar[d]^-{b} \ar@{->>}[r]^-{g'} & C' \ar@{=}[d] \\
    A \ar[r]^-{f} & B \ar[r]^{g} & C'
  }
  \]
  is bicartesian. Let $g': B' \epi C'$ be a cokernel of $f'$ and let $g$
  be the morphism $B \to C'$ such that $gf = 0$ and $gb = g'$.
  Now consider the commutative diagram
  \[
  \xymatrix{
    A' \ar@{ >->}[r]^-{\mat{0 \\ 1}} \ar@{=}[d] & 
    A \oplus A' \ar@{->>}[r]^-{\mat{-1 & 0}}
    \ar@{ >->}[d]^-{\mat{1 & a \\ 0 & f'}} & A \ar[d]^{f} \\
    A' \ar@{ >->}[r]^-{\mat{a \\ f}} &
    A \oplus B' \ar@{ ->>}[r]^-{\mat{-f & b}} \ar@{->>}[d]^{\mat{0 & g'}} & 
    B \ar[d]^{g} \\
    & C' \ar@{=}[r] & C'
  }
  \]
  in which the rows are exact and the first two columns are exact. It
  follows that the third column is exact and hence $f$ is an
  admissible monic.
  
  Now that we know that $(\scrA,\scrE)$ is an exact category, we
  conclude from~(iii) and
  Proposition~\ref{prop:weakly-split-obscure-axiom} that
  $\scrA$ must be weakly idempotent complete.
\end{proof}

\section*{Acknowledgments}

I would like to thank Paul Balmer for introducing me to exact categories
and Bernhard Keller for answering numerous questions via email or via his
excellent articles. Part of the paper was written
in November 2008 during a conference at the Erwin--Schr\"odinger--Institut
in Vienna. The final version of this paper was prepared at the
Forschungsinstitut f\"ur Mathematik of the ETH Z\"urich. I would like
to thank Marc Burger and his team for their support and for
providing excellent working conditions. 
Matthias K\"unzer and an anonymous referee
made numerous valuable suggestions which led to substantial
improvements of the text. 
Finally, I am very grateful to Ivo Dell'Ambrogio for his
detailed comments on various preliminary versions of this paper.

\bibliographystyle{amsplain}
\bibliography{bibliography}

\end{document}